\documentclass{article}
\usepackage{tcolorbox}
\usepackage{pgfplots}
\pgfplotsset{compat=1.18}

\usepackage[margin=2.5cm]{geometry}
\usepackage{tikz, pgfplots}
\usepackage{comment}
\usepackage{enumitem}
\usepackage{bm, bbm, mathtools, amsmath, amsfonts, amsthm, mathrsfs, amssymb} 
\usepackage{nicefrac}
\usepackage{ulem}
\usepackage{nameref}
\usepackage{ifthen, etoolbox}
\usepackage[algo2e, ruled]{algorithm2e}
\usepackage{float}
\usepackage{caption}
\usepackage{subcaption}
\usepackage{pifont}
\usepackage{natbib}
\usepackage{cases}
\usepackage{setspace}

\usepackage{hyperref, xcolor}
\definecolor{linkcolour}{rgb}{0,0,0.6}
\hypersetup{colorlinks=true,
	linkcolor=black,
	citecolor=mydarkblue,
	urlcolor=magenta,
	linktocpage,
	plainpages=false}
\usepackage{cleveref}

\numberwithin{equation}{section}
\theoremstyle{plain}
\newtheorem{theorem}{Theorem}[section]
\newtheorem{lemma}[theorem]{Lemma}

\newtheorem*{example*}{Example}

\theoremstyle{defintion} 
\newtheorem{definition}{Definition}[section]
\newtheorem{assum}{Assumption}
\newtheorem*{assum*}{Assumption}

\newtheoremstyle{boldremark}
{\topsep} 		
{\topsep} 		
{\normalfont} 	
{}          	
{\bfseries} 	
{.}         	
{.5em}      	
{}          	
\newtheorem{remark}{Remark}[section]
\newtheorem*{remark*}{Remark}

\newcommand{\E}{\mathbb{E}}
\newcommand{\N}{\mathbb{N}}
\newcommand{\R}{\mathbb{R}}

\newcommand{\bv}{\mathbf{v}}

\newcommand{\eps}{\epsilon}

\newcommand{\asymeq}{{\asymp}}
\newcommand{\asymleq}{{\;\lesssim\;}}
\newcommand{\asymgeq}{{\;\gtrsim\;}}
\newcommand{\mmin}{\wedge}
\newcommand{\mmax}{\vee}

\newcommand{\wrt}{{\rm w.r.t.\,}}

\newcommand{\romanOne}{\uppercase\expandafter{\romannumeral1}}
\newcommand{\romanTwo}{\uppercase\expandafter{\romannumeral2}}

\newcommand{\inner}[2]{\left\langle #1, #2 \right\rangle}

\definecolor{myorange}{RGB}{255,99,71}

\definecolor{mydarkgreen}{RGB}{0, 110, 110}

\definecolor{mydarkblue}{rgb}{0,0.08,0.45}

\definecolor{mydarkred}{RGB}{178,34,34}

\newcommand{\algname}[1]{{\small \sf #1}}

\newcommand{\dom}[1]{\mathrm{dom}\left(#1\right)}

\newcommand{\ri}[1]{\mathrm{ri}\left(#1\right)}
\newcommand{\norCone}[2]{\mathsf{N}_{#1}\left(#2\right)}
\newcommand{\epi}[1]{\mathrm{epi}\left(#1\right)}
\newcommand{\cl}[1]{\mathrm{cl}\left(#1\right)}

\newcommand{\condnum}{\smarttext{\kappa}}

\newcommand{\yes}{{\color{green} \ding{51}}}
\newcommand{\no}{{\color{red} \ding{55}}}

\newcommand{\norm}[1]{\generalnorm{#1}} 

\newcommand{\dist}[1]{\mathrm {dist}\xkh{#1}}
\newcommand{\pf}{\Phi}

\newcommand{\xerror}[1]{\zeta_{\x, #1}}
\newcommand{\yerror}[1]{\zeta_{\y, #1}}
\newcommand{\acc}{\varepsilon}

\newcommand{\errorxiter}{\xerror{\iter}}

\newcommand{\errorxiterM}{\xerror{\iterM}}

\newcommand{\erroryiter}{\yerror{\iter}}

\newcommand{\erroryiterM}{\yerror{\iterM}}

\newcommand{\etax}{\eta_\x}
\newcommand{\etay}{\eta_\y}

\newcommand{\disX}[1]{D_{\x,#1}}
\newcommand{\resX}[1]{\varepsilon_{\x,#1}}

\newcommand{\z}{\mathbf{z}}

\newcommand{\nnl}{&\notag\\}

\newenvironment{proofof}[1][Proof]{
  \par\noindent\textbf{#1.}\quad\ignorespaces
}{\hfill$\blacksquare$\par}

\newcommand{\prox}[2]{\mathrm{prox}_{#1}\xkh{#2}}
\newcommand{\ystar}{\y^\star}
\newcommand{\yiterstar}{\yiter^\star}
\newcommand{\yiterMstar}{\yiterM^\star}

\newcommand{\diaY}{D_{\Y}}

\newcommand{\gradx}[2]{\pgradX f(#1,#2)}
\newcommand{\grady}[2]{\pgradY f(#1,#2)}

\newcommand{\dkh}[1]{{ \left\{ #1 \right\} }}
\newcommand{\xkh}[1]{{ \left( #1 \right) }}
\newcommand{\zkh}[1]{\left[ #1 \right]}

\newcommand{\Exp}[1]{{\E}\left[#1\right]}

\newcommand{\Epiter}[1]{{\E_\iter}\left[#1\right]}

\newcommand{\sumIter}{\sum_{t=1}^T }
\newcommand{\sumIterM}{\sum_{t=1}^{T-1}}

\DeclareMathOperator*{\argmin}{arg\,min}
\DeclareMathOperator*{\argmax}{arg\,max}
\newcommand{\generalnorm}[1]{\left\lVert#1\right\rVert}

\newcommand{\grad}{\nabla}

\newcommand{\loround}[1]{\left\lfloor #1 \right\rfloor}
\newcommand{\upround}[1]{\left\lceil #1 \right\rceil}

\newcommand{\x}{\mathbf{x}}
\newcommand{\y}{\mathbf{y}}

\newcommand{\pgradX}{\nabla_{\x}}

\newcommand{\pgradY}{\nabla_{\y}}

\newcommand{\smarttext}[1]{%
  \ifmmode
    #1%
  \else
    $#1$%
  \fi
}

\newcommand{\iter}{\smarttext{t}} 
\newcommand{\iterP}{\smarttext{{t+1}}} 
\newcommand{\iterM}{\smarttext{{t-1}}} 

\newcommand{\momx}[1]{\mathbf{m}_{\mathbf{x},#1}}

\newcommand{\momy}[1]{\mathbf{m}_{\mathbf{y},#1}}

\newcommand{\momxiter}{\momx{\iter}}
\newcommand{\momxiterM}{\momx{\iterM}}

\newcommand{\momyiter}{\momy{\iter}}
\newcommand{\momyiterM}{\momy{\iterM}}

\newcommand{\fil}{\mathcal{F}} 

\newcommand{\fiterM}{f(\xiterM,\yiterM)}

\newcommand{\stovec}{g}
\newcommand{\gxiter}{\stovec_{\x,\iter}}
\newcommand{\gxiterM}{\stovec_{\x,\iterM}}
\newcommand{\gyiter}{\stovec_{\y,\iter}}
\newcommand{\gyiterM}{\stovec_{\y,\iterM}}

\newcommand{\subgradx}[1]{g_{r}(#1)}
\newcommand{\subgradxiter}{\subgradx{\xiter}}
\newcommand{\subgradxiterM}{\subgradx{\xiterM}}
\newcommand{\subgradxiterP}{\subgradx{\xiterP}}

\newcommand{\subgrady}[1]{g_{h}(#1)}
\newcommand{\subgradyiter}{\subgrady{\yiter}}

\newcommand{\ball}[2]{\mathbb{B}(#1,#2)}
\newcommand{\ballxiter}{\ball{\xiter}{\etax}}
\newcommand{\ballyiter}{\ball{\yiter}{\etay}}

\newcommand{\twoDvec}[2]{
\begin{bmatrix}
#1 \\
#2
\end{bmatrix}
}

\newcommand{\xiter}{\x_{\iter}}
\newcommand{\xiterP}{\x_{\iterP}}
\newcommand{\xiterM}{\x_{\iterM}}
\newcommand{\yiter}{\y_{\iter}}
\newcommand{\yiterP}{\y_{\iterP}}
\newcommand{\yiterM}{\y_{\iterM}}

\newcommand{\yiterk}[1]{\y_{\iter,#1}}

\newcommand{\wtyk}{\widetilde{\y}_k}
\newcommand{\wtykP}{\widetilde{\y}_{k+1}}

\newcommand{\wtxik}{\widetilde{\xi}_{k}}
\newcommand{\wterrory}[1]{\widetilde{\zeta}_{\y,#1}}

\newcommand{\whatf}{\widehat{f}}
\newcommand{\whatdelta}{\widehat{\delta}}

\newcommand{\distri}{\mathcal{D}}

\newcommand{\X}{\mathcal{X}}
\newcommand{\Y}{\mathcal{Y}}

\newcommand{\real}[1]{\mathbb{R}^{#1}}
\newcommand{\realx}{\real{d_\x}}
\newcommand{\realy}{\real{d_\y}}

\newcommand{\abs}[1]{{\left| #1 \right|}}

\newcommand{\var}{\bm{\sigma}}

\newcommand{\stackAlign}[2]{
	\hspace*{(\widthof{$\mathsurround=0pt #1$} - \widthof{$\mathsurround=0pt #2$})/2} \stackrel{\mathclap{#1}}{#2} &
	\hspace*{(\widthof{$\mathsurround=0pt #1$} - \widthof{$\mathsurround=0pt #2$})/2} 
}

\newcommand{\eqdef}{\overset{\Delta}{=}}

\newcommand{\algref}[2][]{%
  \hyperref[#2]{#1}%
}

\newcommand{\trsgdam}{\algref[\algname{TR-SGDAM}]{alg:TR-SGDAM}}

\newcommand{\trsgdamax}{\algref[\algname{TR-SGDAmax}]{alg:TR-SGDAmax}}

\newcommand{\strgdam}{{\small \sf Stoc-TRGDAM}\ } 
\newcommand{\sgda}{{\small \sf SGDA}\ }
\newcommand{\strgdamax}{{\small \sf Stoc-TRGDAmax}\ } 

\title{The Role of Gradient Modification in Heavy-Tailed Nonconvex Stochastic Min-Max Optimization }
\author{Tianxi Zhu\footnotemark[1] \and 
		Yi Xu\footnotemark[2]\and Xiangyang Ji\footnotemark[3]}
\date{}

\begin{document}
\maketitle

\renewcommand{\thefootnote}{\fnsymbol{footnote}}
\footnotetext[1]{{School of Control Science and Engineering, Dalian University of Technology; email: \texttt{zhutianxi3291@dlut.mail.edu.cn}}}
\footnotetext[2]{{School of Control Science and Engineering, Dalian University of Technology; email: \texttt{yxu@dlut.edu.cn}}}
\footnotetext[3]{{Department of Automation, Tsinghua University; email: \texttt{xyji@tsinghua.edu.cn}}}
\renewcommand{\thefootnote}{\arabic{footnote}}

\begin{abstract}
    

Stochastic min-max optimization has attracted increasing attention due to its applications in modern machine learning, while existing theoretical studies mainly rely on the bounded variance assumption for stochastic gradients. Under heavy-tailed noise, where stochastic gradients only possess a finite $p$-th moment for $p\in(1,2]$, gradient clipping or normalization is commonly believed to be necessary to guarantee convergence. In this work, we revisit stochastic min-max optimization under heavy-tailed noise and provide a comprehensive theoretical study of stochastic gradient descent ascent ({\small\sf SGDA}). We first show that vanilla {\small\sf SGDA}, without any modification to its update rule, can converge under heavy-tailed noise in both nonconvex-strongly-concave (NC-SC) and nonconvex-concave (NC-C) settings, establishing the first convergence guarantees for {\small\sf SGDA} in these regimes. Beyond unregularized problems, we further investigate regularized stochastic min-max optimization, where directly incorporating gradient normalization into proximal updates is nontrivial due to the incompatibility between normalization and proximal structures. We overcome this difficulty by developing new clipping-free algorithms, i.e., {\small \sf Stoc-TRGDAM} and {\small \sf Stoc-TRGDmax}, and they both can achieve the optimal dependence on the target accuracy without using gradient clipping. 

\end{abstract}

\section{Introduction}

This paper first focuses on the following min-max optimization problem:
\begin{align}\label{eq:non-regularized-min-max-problem}
    \min_{\x\in\realx}\max_{\y\in\realy} f(\x,\y),
\end{align}
where \(f:\realx\times\realy\mapsto \real{} \) is a real-valued function that is differentiable \wrt\footnote{``\wrt" is written as ``with respect to ".} both \(\x\) and \(\y\). The above minimax problem can be associated with many important areas in machine learning, such as generative adversarial networks (GANs)~\cite{GAN1}, multi-agent reinforcement learning (MARL)~\cite{MARL1,MARL2}, and AUC maximization~\cite{AUC1,AUC2}, among others. In most cases, it is difficult to guarantee that the objective function $f$ is convex \wrt $\x$. If $f$ \wrt $\x$ can be nonconvex, this paper refers to Problem~\eqref{eq:non-regularized-min-max-problem} as \textit{nonconvex min-max problems,} which is also the problem setting considered in this paper.

One of the simplest and most widely adopted optimization framework for solving Problem~\eqref{eq:non-regularized-min-max-problem} is the \textit{stochastic gradient descent-ascent} (\algname{SGDA}) method.  It extends the idea of \textit{stochastic gradient descent} (\algname{SGD}) method by performing one \algname{SGD} step on $\x$ and one \textit{stochastic gradient ascent} (\algname{SGA}) step on $\y$ simultaneously in each iteration. 
\begin{align}
    \begin{cases}
        \xiterP =\xiter -\etax\pgradX f(\x,\y,\xi)\\
        \yiterP =\yiter + \etay \pgradY f(\x,\y,\xi)
    \end{cases}\tag{\text{{\sf SGDA}}}
\end{align}
Here, $\pgradX f(\x,\y,\xi)$ and $\pgradY f(\x,\y,\xi)$ are the stochastic estimators of $\pgradX f(\x,\y)$ and $\pgradX f(\x,\y)$, respectively, obtained by sampling the random variable $\xi$ from an unknown distribution $\mathcal{D}_\xi$. The stochastic estimators are usually unbiased, i.e., their expectations are true gradients. For nonconvex minimax problems, the convergence of \algname{SGDA} and its variants has already been extensively studied~\cite{Two-timescale_gradient_descent_ascent_algorithms_for_nonconvex_minimax_optimization,delving_into_the_convergence_of_generalized_smooth_minmax_optimiation,tight_analysis_of_extra-gradient_and_optimistic_gradient_methods_for_nonconvex_minimax_problems,efficient_mirror_descent_ascent_methods_for_nonsmooth_minimax_problems,unified-convergennce-analysis-for-adaptive-optimization-with-moving-average-estimator,faster_single_loop_algo_mminmax_without_concavity-SAGDA,sapd+,alternating_proximal-gradient_steps_for_stochastic_nonconvex_concave_minmax_problems,unified-convergennce-analysis-for-adaptive-optimization-with-moving-average-estimator, adaptive_algorithms_with_sharp_convergence_rates_for_stochastic_hierarchical_optimizatio,tiada}. Most previous work assumes that the gradient noises are variance-bounded, i.e., there exists a positive real number $\sigma$, if for any $\forall (\x,\y)\in\realx\times\realy$, we have
\begin{align}
    \E\zkh{\|\pgradX f(\x,\y,\xi)-\pgradX f(\x,\y)\|^2}\leq \sigma^2,\quad \E\zkh{\|\pgradY f(\x,\y,\xi)-\pgradY f(\x,\y)\|^2}\leq \sigma^2. \notag
\end{align}
The above assumption also means that the stochastic gradients are \textit{light-tailed.} Nevertheless, some recent empirical studies, e.g.,~\cite{why_are_adaptive_methods_good_for_attention_models,on_proximal_policy_optimization_heavy-tailed-gradients,linear_attention_is_all_you_need,revisiting_the_noise_medel_of_stochastic_gradient_descent}, have observed that, during model training, the noise of stochastic gradients should be \textit{heavy-tailed}. That is the following assumption: 
\begin{align}
    \E\zkh{\|\pgradX f(\x,\y,\xi)-\pgradX f(\x,\y)\|^p}\leq \sigma^p,\quad \E\zkh{\|\pgradY f(\x,\y,\xi)-\pgradY f(\x,\y)\|^p}\leq \sigma^p, \notag
\end{align}
where $p\in(1,2]$. Note that the above noise assumption rigorously extends the typical bounded variance assumption because the variance may be infinite when $p < 2$. 

The popularity of the heavy-tailed noise assumption began with the study of minimization problems. In the smooth nonconvex setting, many previous works have show that the vanilla \algname{SGD} with gradient clipping (\algname{Clip-SGD}), e.g.,~\cite{why_are_adaptive_methods_good_for_attention_models,high_probability_convergence_of_clip-sgd_under_heavy-tailed_noise,high_probability_bounds_for_stochastic_optimization_and_variational_inequalities}, can  achieve the lower bound $\Omega(\eps^{-\frac{3p-2}{p-1}})$~\cite{why_are_adaptive_methods_good_for_attention_models,Zijian_Liu_nsgd} that is for general first-order algorithms under heavy-tailed noise. 
\begin{align}
    \xiterP =\xiter -\eta\nabla f(\xiter,\xi_\iter)\cdot\min\dkh{\frac{\lambda}{\|\nabla f(\xiter,\xi_\iter)\|},1},\;\lambda>0 \tag{\text{\sf Clip-SGD}}
\end{align}
The reason these works studied \algname{Clip-SGD} rather than \algname{SGD} itself is that they believed vanilla \algname{SGD} could not converge under heavy-tailed noise. For example, consider a specific objective $f = x^2/2$ and a stochastic oracle $f'(x,\xi) = x + \xi$, where $\xi$ can be any heavy-tailed random variable with zero mean but unbounded variance. For the \algname{SGD} with initial point $x_1=0$ and a deterministic step-size $\eta$, we have $x_2 = x_1 - \eta f'(x_1,\xi) = -\eta\xi$, and obtain $\E[|f'(x_2)|^2] = \eta^2\E[|\xi|^2] = +\infty$, which seems to imply that \algname{SGD} cannot achieve convergence under heavy-tailed noise even for a simple quadratic function. However, we found that if we consider $\E[|f'(x_2)|^p]$, its value can still be finite. Until recently, \cite{in-expectation-convergence-sgd-heavy-tailed_Zijian-Liu,Can-sgd-handle-heavy-tailed-noise_Ilyas_Fatkhullin} proved that \algname{SGD} can in fact converge under heavy-tailed noise with the complexity $\mathcal{O}(\eps^{-\frac{2p}{p-1}})$ in the nonconvex smooth setting. To the best of our knowledge, for minimax problems, although some recent works, such as~\cite{high-probability_convergence_guarantees_of_stochastic_gradient_descent_ascent_in_structured_nonconvex_min-max_games,nonconvex_decentralized_stochastic_bilevel_optimization_uner_heavy_tailed_noises,stochastic_bilevel_optimization_with_heavy-tailed_noise,federated_stochastic_minmax_optimization_under_heavy_tailed_noises,zeroth-order_methods_for_non-smooth_stochastic_problems_under_heavy_tailed_noise}, have considered heavy-tailed noise, \textbf{they all used gradient clipping or gradient normalization and no work gives an answer of whether \algname{SGDA} itself can converge.} Therefore, we pose the first question studied in this paper:
\begin{center}
    \textit{Q1: Can vanilla \algname{SGDA} converge under heavy-tailed noise? If it can converge, then compared with methods equipped with gradient clipping, is its complexity better or worse?}
\end{center}
Regarding the second part of Q1, that is, whether the complexity of \algname{SGDA} is better or worse than that of the method using gradient clipping, readers may think the conclusion is obvious because, as we mentioned, there is indeed a gap between the complexity $\mathcal{O}(\eps^\frac{-2p}{p-1})$ of \algname{SGD} and the complexity $\mathcal{O}(\eps^{-\frac{3p-2}{p-1}})$ of \algname{Clip-SGD}. However, if we consider a broader class of functions, such as nonconvex, nonsmooth, weakly convex functions, i.e., \( f(x)+\rho/2\|x\|^2 \) is convex, under bounded constraints, \cite{sgd-weakly-convex-heavy-tailed-noises} has shown that both \algname{SGD} and \algname{Clip-SGD} can only achieve $\mathcal{O}(\eps^\frac{-2p}{p-1})$, and the complexity is difficult to improve for \algname{Clip-SGD}. \textbf{This indicates that, in some cases, gradient clipping does not always improve complexity.} It should be noted that \( \min_\x\max_\y f(\x,\y) \) can degenerate into \( \min_\x f(\x)\), \textbf{but we do not assume that \(\x\) must lie within a bounded region.} In fact, this paper studies \algname{SGDA} under the following two common minimax problem settings:
\begin{align}
    \text{$f(\cdot,\cdot)$ is $L$-smooth and }\begin{cases}
        \text{$f(\x,\cdot)$ is $\mu$-strongly concave,} &\quad {\text{NC-SC setting}}\\
        \text{$f(\x,\cdot)$ is concave.} &\quad {\text{NC-C setting}}
    \end{cases}\notag 
\end{align}
The formal statement of the NC-SC setting can be found in Assumption~\ref{assum: strongly_concave_for_y}, and the formal statement of the NC-C setting can be found in Assumption~\ref{assum:prelims/nonconvex-concave-setting}. \textbf{Here we emphasize that, compared with works that study minimax problems in the NC-SC or NC-C settings under light-tailed noise, we do not impose any nontrivial assumptions.} Here, “nontrivial” refers to additional assumptions beyond those required for the \textit{general Danskin theorem} (Lemmas~\ref{lem:prelims/smoothness_dual_optimizers},~\ref{lem:prelims/smoothness_primal_function} and~\ref{lem:prelims/weakly-convexity-primal-function}) to hold. 

Next, if the complexity of \algname{SGDA} under heavy-tailed noise is indeed worse than that of methods equipped with gradient clipping under certain settings, \textbf{we hope to develop new methods without using clipped gradients} to achieve the same complexity as the methods using clipped gradients. The reason is that \textbf{we do not want to introduce hyperparameter $\lambda$ that require fine-tuning in practice.} Moreover, we aim to address the following more general \textbf{regularized} minimax problem:
\begin{align}\label{eq:regularized-min-max-problem} 
      \min_{\x\in\realx}\max_{\y\in\realy}F(\x,\y),\quad F(\x,\y)\eqdef  f(\x,\y)-h(\y)+r(\x),
\end{align}
where the regularizers $r:\realx\mapsto (-\infty,+\infty]$ and $h:\realy\mapsto (-\infty,+\infty]$ are both proper, closed and convex on their domain. To illustrate our motivation, let us first consider the \textit{composite minimization problem}, i.e., $\min_\x\{f(\x)+r(\x)\}$. For such a problem, we can adopt the \textit{stochastic proximal gradient descent} (\algname{SPGD}) method, i.e., $\xiterP = \prox{\eta r}{\xiter-\eta\nabla f(\xiter,\xi_\iter)}$, where $\prox{\eta r}{\cdot}\eqdef \argmin_\x\{r(\x)+(1/2\eta)\|\x-\cdot\|^2\}$. For the noncomposite minimization problem under heavy-tailed noise, normalized \algname{SGD} (\algname{NSGD}), compared with \algname{Clip-SGD}, can achieve the same complexity without introducing additional hyperparameters, and can also guarantee convergence even when $p$ is unknown~\cite{Zijian_Liu_nsgd,From_Gradient_Clipping_to_Normalization}. Therefore, we hope to incorporate the gradient normalization module into the \algname{SPGD} method.
\begin{align}
    \xiterP = \xiter - \eta\frac{\nabla f(\xiter,\xi_\iter)}{\|\nabla f(\xiter,\xi_\iter)\|}\tag{\text{\sf NSGD}}
\end{align}
We have two natural ways of incorporation: the first is the \textit{proximal normalized gradient}, i.e., $\xiterP = \prox{\eta r}{\xiter-\eta\frac{\nabla f(\xiter,\xi_\iter)}{\|\nabla f(\xiter,\xi_\iter)\|}}$; the second is the \textit{normalized gradient mapping}, i.e., $\xiterP = \xiter-\eta\mathcal{G}(\xiter)/\|\mathcal{G}(\xiter)\|$, where $\mathcal{G}(\xiter) \eqdef \eta_t^{-1}(\xiter-\widetilde{\x}_\iterP), \widetilde{\x}_\iterP\eqdef \prox{\eta_t r}{\xiter-\eta_t\nabla f(\xiter,\xi_\iter)}$. However, through the discussion in Section~\ref{sec:how_use_normalized_gradient_to_composite_optimization}, we realize that even in the deterministic setting, both of these ways are \textbf{naive}. \textbf{Therefore, we believe that the barrier to incorporate exists even for composite minimization problems.} Let us review the existing literature on composite optimization under heavy-tailed noise, although~\cite{high_probability_convergence_for_composite_and_distributed_stochastic_minmization,clipped_gradient_methods_for_nonsmooth_convex_optim,zeroth-order_proximal_clipped_gradient_method_with_shifts_for_distributed_stochastic_composite_optimization} have developed new methods that can achieve the same complexity as \algname{Clip-SGD} for noncomposite optimization, these methods all used gradient clipping.~\cite{normalized_stochastic_proximal_approximation_methods_for_nonsmmoth_composite_optimization} studied composite and compostional stochastic minimization under heavy-tailed noises, i.e., $\min_\x \{\E_\xi[f(g(\x,\xi))]+r(\x)\}$, which includes composite stochastic minimization. Their method is $\xiterP = \eta_\iter \mathcal{G}(\xiter)\cdot\min\{\rho_\iter/\|\mathcal{G}(\xiter)\|,1\}$\footnote{Note that they actually use $\mu$ to represent the step-szie. We replace $\mu$ with $\eta_\iter$ here to avoid confusion with the strong concavity coefficient.}. We point out that although they refer to such a update as ``normalized stochastic proximal approximation," we consider this equivalent to clipping $\mathcal{G}(\xiter)$ by a hyperparameter $\rho_\iter$.~\cite{stochastic_compositional_optimization_via_hybrid_momentum_frank_wolfe}'s method indeed reaches the lower bound for nonconvex smooth optimization without using gradient clipping, but the \textit{Frank-Wolfe method} they used, i.e., $\xiterP \in \argmin_{\x\in\mathcal{X}}\{\inner{\nabla f(\xiter,\xi_\iter)}{\x}+r(\x)\}$, requires the constraint region $\mathcal{X}$ to be compact; otherwise, the Frank-Wolfe update may not be well-defined\footnote{Let $r(\x)\eqdef 1/\x$ and $\mathcal{X}\eqdef [1,+\infty)$, then for any $k\leq0$, the optimal solution set of $\min_{\x\in\mathcal{X}}\{1/\x+k\x\}$ is empty.}. For minimax problems, although \cite{nonconvex_decentralized_stochastic_bilevel_optimization_uner_heavy_tailed_noises, federated_stochastic_minmax_optimization_under_heavy_tailed_noises} have used gradient normalization to design their methods to avoid using gradient clipping, they only consider unregularized minimax problems, i.e., Problem~\eqref{eq:non-regularized-min-max-problem}. \textbf{To the best of our knowledge, for regularized min-max problems, there is currently no known work that has developed methods without using gradient clipping that achieve the same complexity as the methods use gradient clipping.} Therefore, we pose the second question studied in this paper:
\begin{center}
    \textit{Q2: For the \textbf{regularized} min-max problem, i.e., Problem~\eqref{eq:regularized-min-max-problem}, can we develop provable (first-order) algorithms without using gradient clipping, while achieving the same complexity as those algorithms that use gradient clipping?}
\end{center}
\subsection{Contributions}
Around the two questions we raised for min-max stochastic optimization under heavy-tailed noise, let us summarize the main results we obtained:
\begin{itemize}
    \item \textit{In the NC-SC setting:} 
    \begin{itemize}
        \item We prove the convergence of vanilla min-batch \algname{SGDA}
        for Problem~\eqref{eq:non-regularized-min-max-problem} under heavy-tailed noise. Theorem~\ref{thm:NC-SC/SGDA/thm_Complexity_SGDA} shows that the complexity of min-batch \algname{SGDA} is 
            $\mathcal{O}\xkh{{\condnum^\frac{3p}{2(p-1)}\sigma^\frac{p}{p-1}}{\eps^\frac{-2p}{p-1}}},$ 
        with batch-size $B=\mathcal{O}\xkh{ \max\dkh{\upround{{\condnum^\frac{4-p}{2(p-1)}\sigma^\frac{p}{p-1}}{\eps^{-\frac{2}{p-1}}}},1}}$. Here, $\condnum$ is denoted as $L/\mu$. When $p=2$, the complexity reduces to the best known result for \algname{SGDA} under the bounded variance assumption. \textbf{When $p<2$, to the best of our knowledge, we show the fisrt upper bound result of vanilla \algname{SGDA} under heavy-tailed noise in the NC-SC setting.} 
        \item We realized that there is still a gap between the upper bound result we obtained for \algname{SGDA} and the lower bound $\Omega(\eps^{-\frac{3p-2}{p-1}})$ in Theorem~\ref{thm:lower_bound} for the NC-SC setting. Therefore, we propose two novel algorithms without using gradient clipping, \strgdam (Algorithm~\ref{alg:TR-SGDAM})  and \strgdamax (Algorithm~\ref{alg:TR-SGDAmax}), which not only achieve the \textbf{optimal complexity on $\eps$}, but also both handle \textbf{regularized} min-max problems, i.e., Problem~\eqref{eq:regularized-min-max-problem}. Specifically, as a single-loop method, \strgdam has a complexity of $\mathcal{O}(\sigma^{\frac{p}{p-1}}\condnum^\frac{3p-2}{p-1}\eps^{-\frac{3p-2}{p-1}})$, and the batch-size is $\mathcal{O}(1)$. \strgdamax is a nested loop method and its total gradient complexity is $\mathcal{O}\xkh{\sigma^{\frac{p}{p-1}}\condnum^\frac{2p-1}{p-1}\eps^{-\frac{3p-2}{p-1}}}$ with $\mathcal{O}\xkh{\max\dkh{\upround{\sigma^\frac{p}{p-1}{\eps}^\frac{-p}{p-1}},1}}$ batch-size. Although \strgdamax involves nested loops and requires a mini-batch, it improves the dependence of the complexity on the condition number $\condnum$ compared with \strgdam and requires an smaller batch-size than \algname{SGDA} when $p<2$. Through Table~\ref{tab:comparasion_NCSCorPL_heavy-tails}, we compare the proposed methods with some methods under similar settings from previous studies. 
    \end{itemize}
    \item \textit{In the NC-C setting}, we considered a slightly broader Problem~\eqref{eq:semi-regularized-min-max-problem} than Problem~\eqref{eq:non-regularized-min-max-problem}, i.e., replacing $r$ in Problem~\eqref{eq:regularized-min-max-problem} with the indicator function $\delta_{\mathcal{X}}:\mathcal{X}\mapsto (-\infty,+\infty]$ on a unbounded closed convex set $\mathcal{X}$, while $h$ remains a general closed convex function. For Problem~\eqref{eq:semi-regularized-min-max-problem}, Theorem~\ref{thm:UpperBound_SPGDA_NC-C} shows the complexity of \algname{SGDA} under heavy-tailed noise is $\mathcal{O}(\sigma^\frac{3p-2}{p-1}\eps^{-\frac{6p-4}{p-1}})$, \textbf{which is the first upper bound result of \algname{SGDA} under heavy-tailed noise in NC-C setting.} This complexity matches that of the method proposed in~\cite{high-probability_convergence_guarantees_of_stochastic_gradient_descent_ascent_in_structured_nonconvex_min-max_games}, which uses gradient clipping for the updates of $\xiter$ and $\yiter$. 
    Our results show that, \textbf{for NC-C min-max problems, gradient clipping neither improves the complexity under heavy-tailed noise, but can instead converge under a stronger convergence measure} (a more careful comparison with~\cite{high-probability_convergence_guarantees_of_stochastic_gradient_descent_ascent_in_structured_nonconvex_min-max_games} can be found in Table~\ref{tab:comparision_SGDA_(Clip)_NC-C}).
\end{itemize}
\subsection{Related Work}\label{sec:related work}
\paragraph{Stochastic min-max optimization under heavy-tailed noise}
For min-max problems under light-tailed noise, the recent survey \cite{avoid_overclaims_summary_of_complexity_bounds_minimax_optimization} reviews state-of-the-art upper and lower bound results under various settings. We focus on min-max problems under heavy-tailed noise. In the convex-concave (C-C) setting, ~\cite{zeroth-order_methods_for_non-smooth_stochastic_problems_under_heavy-tailed_noise} develops a single-loop zeroth-order method using extra-gradients and gradient clipping, whose complexity in terms of the duality gap, i.e., $\max_{\y'} f(\x,\y')-\min_{\x'}f(\x',\y)$, can achieve $\mathcal{O}(\eps^{-\frac{p}{p-1}})$. Under the NC-SC setting, ~\cite{stochastic_bilevel_optimization_with_heavy-tailed_noise} designs a nested-loop method that uses gradient clipping to update $\yiter$, and gradient normalization to update $\xiter$. The total gradient complexity of their method is $\mathcal{O}(\condnum^{\frac{2p-1}{p-1}}\eps^{-\frac{3p-2}{p-1}})$ to return a point $\x$ satisfying $\|\nabla \pf(\x)\|\leq\eps$, where $\condnum$ is the condition number, and $\pf(\x)\eqdef \max_\y f(\x,\y)$. Note that their method requires a batch-size of $\mathcal{O}(\eps^{-\frac{p}{p-1}})$. In the NC-PL setting, i.e., $f(\x,\cdot)$ is not necessarily $\mu$-strongly concave \wrt $\y$, but satisfies the \textit{Polyak–Łojasiewicz inequality}, which relaxes strong concavity, ~\cite{federated_stochastic_minmax_optimization_under_heavy_tailed_noises} designs a single-loop method without gradient clipping. Their idea is to replace all gradients in \algname{SGDA} with normalized momentum. They proves an upper bound $\mathcal{O}(\condnum^{\frac{2p}{p-1}}\eps^{-\frac{2p}{p-1}})$ for this method. Still in the NC-PL setting, ~\cite{high-probability_convergence_guarantees_of_stochastic_gradient_descent_ascent_in_structured_nonconvex_min-max_games} also proposes a single-loop method. Although their method uses clipped gradients to update $\xiter$ and $\yiter$, they proved that the method achieves a tighter bound $\mathcal{O}(\condnum^{\frac{3p-2}{p-1}}\eps^{-\frac{3p-2}{p-1}})$, compared to the one presented in~\cite{federated_stochastic_minmax_optimization_under_heavy_tailed_noises}, and their method can be extended to the high-probability convergence as well as NC-C setting with the complexity $\mathcal{O}(\eps^{-\frac{6p-4}{p-1}})$. For \textit{Bilevel optimization} problems which recover the min-max problems, \cite{nonconvex_decentralized_stochastic_bilevel_optimization_uner_heavy_tailed_noises} proposes new method to handle heavy-tailed noise, but they did not provide the upper bound result of their method for any specific min-max problem settings. We note also that~\cite{Solving_SVI_without_BV_assumption} studies nonconvex stochastic variational inequalities (SVI), which include stochastic min-max problems, without assuming bounded variance. However, their assumptions on stochastic gradients are different from ours. Specifically, they assume the existence of $B,\sigma>0$ such that $\E_\xi[\norm{\nabla f(\z,\xi)-\nabla f(\z)}^2]\leq B^2\norm{\z-\z_0}^2+\sigma^2$, where $\z_0$ denotes the initial point of the algorithm. Overall, we find that existing studies on min-max problems under heavy-tailed noise reveal two limitations: \textit{i) they do not consider the convergence of classical gradient methods (such as \algname{SGDA}) under heavy-tailed noise. ii) they only study unregularized min-max problems, rather than Problem~\eqref{eq:regularized-min-max-problem}.}

\paragraph{Classic stochastic gradient methods under heavy-tailed noise} We also review the existing work on vanilla gradient methods, e.g., \algname{SGD} and \textit{stochastic mirror descent} (\algname{SMD}), for stochastic minimization problems, i.e., $\min_\x f(\x)$ under heavy-tailed noise. 
    For convex optimization,~\cite{mirror_descent_strikes_again_optimal_stochastic_convex_optimization_under_infinite_noise_variance,mirror_descent_strikes_again_optimal_stochastic_convex_optimization_under_infinite_noise_variance,accelerated_stochastic_first-order_method_for_convex_optimizatio_under_heavy-tailed_noise,Can-sgd-handle-heavy-tailed-noise_Ilyas_Fatkhullin,revisiting_the_last_iterate_convergence_of_stochastic_gradient_methods,online_convex_optimization_with_heavy_tails,vanilla_sgd_with_momentum_survives_heavy-tailed_noise,in-expectation-convergence-sgd-heavy-tailed_Zijian-Liu} have considered the convergence of classical \algname{SGD} or \algname{SMD} or their variants under heavy-tailed noise in certain settings, while this paper mainly focuses on progress in nonconvex optimization. 
    For nonconvex and smooth optimization, ~\cite{Can-sgd-handle-heavy-tailed-noise_Ilyas_Fatkhullin} gave the first gradient complexity $\mathcal{O}(\eps^{\frac{-2p}{p-1}})$ of mini-batch \algname{SGD} for finding $\eps$-stationary points of nonconvex and standard $L$-smooth functions, with batch-size $\mathcal{O}(\eps^\frac{-2}{p-1})$ and stepsize $\eta=1/2L$. In particular, they also considered {Hölder smooth} functions with curvature exponent $\nu\in(0,1]$. Through a more refined analysis, \cite{in-expectation-convergence-sgd-heavy-tailed_Zijian-Liu} showed that, for $L$-smooth functions, both \algname{SGD} and \algname{SGD} with momentumn (\algname{SGDM}) can converge with the same complexity $\mathcal{O}(\eps^\frac{-2p}{p-1})$, but with batch-size $\mathcal{O}(1)$. Note that although ~\cite{vanilla_sgd_with_momentum_survives_heavy-tailed_noise} studied \algname{SGD} and \algname{SGDM} for nonconvex  Hölder smooth functions under heavy-tailed noise, a key step in their analysis ~\cite[Lemma 2.4]{vanilla_sgd_with_momentum_survives_heavy-tailed_noise} requires $\nu+1\leq p$, which means that when $\nu=1$, corresponding to standard $L$-smooth optimization, $p$ must be $2$, reducing to the bounded-variance assumption. 
    For nonconvex and nonsmooth optimization, ~\cite{online_convex_optimization_with_heavy_tails} studies nonconvex, nonsmooth, but $G$-Lipshchitz continuous objectives through the \textit{online-to-non-convex-conversation} framework (\algname{O2NC}), which was proposed by~\cite{O2NC}.~\cite{sgd-weakly-convex-heavy-tailed-noises} shows that, for nonsmooth, nonconvex, and weakly convex optimization, \algname{SGD} converges within a compact region with a complexity of $\mathcal{O}(\eps^\frac{-2p}{p-1})$. Interestingly, they also shows that \algname{Clip-SGD} converges over an unbounded region with the same complexity as \algname{SGD} with bounded region.

\section{Preliminaries}
\paragraph{Notations:} We first introduce some standard notation in asymptotic analysis. Given two real numbers $a$ and $b$, if there exists a positive numerical constant $c$ such that $a \leq cb$ or $a \geq cb$ holds, we will express the inequality as $a \asymleq b$ or $a \asymgeq b$. If, in addtion, $a\asymleq b$ and $a\asymgeq b$ are both hold, we say $a\asymeq b$.  If $a$ satisfies $a \asymleq b$, $a \asymgeq b$, or $a\asymeq b$, we also use $\mathcal{O}(b)$, $\Omega(b)$, or $\Theta(b)$ to represent $a$. Specifically, $\widetilde{\mathcal{O}}(b)$ represents $\mathcal{O}(b)$ omitting some polynomial logarithmic factors related to $b$.
In this paper, we discuss the standard inner product $\inner{\cdot}{\cdot}$ in Euclidean space, and the norm $\norm{\cdot}$ represents the Euclidean norm, i.e., $\norm{\cdot} = \sqrt{\inner{\cdot}{\cdot}}$. We denote $\mathbb{B}(\z,\eta)$ as the closed ball $\dkh{\x\mid\norm{\x-\z}\leq\eta}$.
If we study a (Fréchet) differentiable function $f(\x,\y)$, we denote $\nabla f(\x,\y)$, $\pgradX f(\x,\y)$, and $\pgradY f(\x,\y)$ as the full gradient, the partial gradient \wrt $\x$, and the partial gradient \wrt $\y$, respectively. For a proper function $r$, we denote the (regular) subdifferential set of $r$ at the point $\x$ as $\partial r(\x)$. Given a closed convex set $\mathcal{S}$, we denote $\Pi_{\mathcal{S}}(\cdot)$ as the projection operator onto $\mathcal{S}$.
Additionally, the following notations have been widely used in previous literature on min-max optimization:
\begin{itemize}
    \item We define the primal function $\pf(\x)$ as $ \max_\y\dkh{f(\x,\y)-h(\y)}$ and $\Psi(\x)$ as $\pf(\x)+r(\x)$.
    \item Given a $\x$, if $\argmin_\y\dkh{f(\x,\y)-h(\y)}$ is unqiue, we let $\ystar(\x)$ be the maxmizer of $f(\x,\y)-h(\y)$. If $\x_t$ is the $t$-th iteration point of a certatin algorithm, we let $\yiterstar$ be $\ystar(\xiter)$.  
    \item We let $\Delta_t\eqdef \pf(\xiter)-\min_\x\pf(\xiter)$, $\Lambda_t \eqdef \Psi(\xiter) - \min_\x {\Psi(\xiter)}$ and $\delta_t \eqdef \norm{\yiter-\yiterstar}$.
\end{itemize}

Through out this paper, we assume that the primal function $\pf(\x)$ is lower bounded, i.e. $\inf_x \pf(\x)> -\infty$. 
Moreover, we make following assumption
\begin{assum}\label{assum: smoothness_heavy-tailed_noise}
For function $f:\realx\times\realy\mapsto \real{}$, we assume
\begin{enumerate}[label = (\alph*)]
    \item\label{assum:smoothness} $f(\z)$ is $L$-smooth. That is, for any $\z_1,\z_2\in \realx\times\realy$, we have $\norm{\nabla f(\z_1)-\nabla f(\z_2)}\leq L\norm{\z_1-\z_2}$, whcih indicates the following inequality
    \begin{align}
        |\nabla f(\z_1)-\nabla f(\z_2)-\inner{\nabla f(\z_2)}{\z_1-\z_2}|\leq \frac{L}{2}\norm{\z_1-\z_2}^2
    \end{align}
    \item\label{assum:unbiased} We assume that the stochastic oracle can access an unbiased partial gradient estimators, i.e., $\E_{\xi}[\pgradX f(\x,\y,\xi)] = \pgradX f(\x,\y)$ and $\E_{\xi}[\pgradX f(\x,\y,\xi)] = \pgradX f(\x,\y)$.
    \item\label{assum:p-BCM} We assume that stochastic partial gradients have $p$-th bounded central moment ($p$-BCM) for $p\in(1,2]$, which means $\E_{\xi}[\norm{\pgradX f(\x,\y,\xi)-\pgradX f(\x,\y)}^p], \E_{\xi}[\norm{\pgradY f(\x,\y,\xi)-\pgradY f(\x,\y)}^p]\leq \sigma^p$, where $0<\sigma<+\infty$.
\end{enumerate}
\end{assum}



\section{Convergence of {\sf SGDA} under heavy-tailed noise}\label{sec:sgda_under_heavy-tails}
\subsection{NC-SC min-max problems}

\begin{algorithm2e}[t]
	\caption{Stochastic Gradient Descent Ascent (\algname{SGDA})}
	\label{alg:SGDA}
	\DontPrintSemicolon
	\SetKwInOut{Input}{Input}
	\Input{Starting point $(\x_1,\y_1)\in\realx \times \realy $ , step-sizes $\etax,\etay >0$ and batch-sizes $B\in\mathbb{N}_{+}$ }.
	\For{$\iter = 1, 2, \,\hdots, T$}{
    \CommentSty{{\color{blue} \% Case \romanOne: $r = h\equiv0$}}\\
        Independently sample $\xi_{\iter,1},\dots, \xi_{\iter, B} \sim\distri_\xi$ to get the stochastic gradients $\{\nabla f(\xiter,\yiter,\xi_{\iter,i})\}_{i=1}^{B}$. \; 
        Compute $ g_{\x,t} =\frac{1}{B} \sum_{i=1}^{B}\pgradX f(\xiter,\yiter,\xi_{\iter,i})$ and $g_{\y,t} =\frac{1}{B} \sum_{i=1}^{B}\pgradY f(\xiter,\yiter,\xi_{\iter,i})$.\;
        Update $\xiterP =\xiter -\etax\gxiter,\quad \yiterP = \yiter +\etay\gyiter.$\\ 
        \CommentSty{{\color{blue} \% Case \romanTwo: $r(x)=\delta_\mathcal{X}$ and $h(\y)$ is a general proper, closed and convex function}}\\
         Sample $\xi_\iter \sim\distri_\xi$ to get the stochastic gradients $\nabla f(\xiter,\yiter,\xi_\iter)$. \;
        Update $\xiterP = \Pi_\mathcal{X}\xkh{\xiter-\etax\pgradX f(\xiter,\yiter,\xi_\iter)},\quad \yiterP = \prox{\etay r}{\yiter + \etay\pgradY f(\xiter,\yiter,\xi_\iter)}$. 
	}
\end{algorithm2e}
To analyze the convergence of \algname{SGDA} (Algorithm~\ref{alg:SGDA}) under heavy-tailed noise, we begin with the case: the non-regularized minimax problem, i.e. Problem~\ref{eq:non-regularized-min-max-problem}, and we study such a problem in NC-SC setting. In NC-SC setting, we assume that $f(\x,\cdot)$ is $\mu$-strongly concave \wrt $\y$. This is the following assumption.
\begin{assum}\label{assum: strongly_concave_for_y}
    For any $\x\in\realx$, we assume that $f(\x,\cdot)$ is $\mu$-strongly concave \wrt $\y$, i.e., 
    \begin{align}
        f(\x,\y_1)-f(\x,\y_2) \leq \inner{\pgradY f(\x,\y_2)}{\y_1-\y_2} - \frac{\mu}{2}\norm{\y_1-\y_2}^2,\;\forall \y_1,\y_2\in\realy
    \end{align}
\end{assum}
 Going further, if we also assume that $f$ is $L$-smooth and $\mu$-strongly concave, then we define $\condnum\eqdef L/\mu$ as the condition number of $f$. Note that $\condnum\geq 1$.
If $f(\x,\cdot)$ is $\mu$-strongly concave \wrt $\y$, then for any $\x\in\realx$, the optimal solution set of $\max_\y f(\x,\y)$ must be a singleton. If, in addition, the sub-level set $\{ \x \mid \pf(\x)\leq\alpha,\alpha \in\R\}$ is compact, we have the following two key lemmas.
\begin{lemma}{\cite[Proposition 1.1]{proximal_gradient_descent_ascent_variable_under_KL_geometry}}\label{lem:prelims/smoothness_dual_optimizers}
    If Assumptions~\ref{assum: smoothness_heavy-tailed_noise}\ref{assum:smoothness} and~\ref{assum: strongly_concave_for_y} hold, then for any $\x_1,\,\x_2\in\realx$, we have $\norm{\ystar(\x_1)-\ystar(\x_2)}\leq \condnum\norm{\x_1-\x_2}$. 
\end{lemma}
\begin{lemma}{\cite[Proposition 1.2]{proximal_gradient_descent_ascent_variable_under_KL_geometry}}\label{lem:prelims/smoothness_primal_function}
    If Assumptions~\ref{assum: smoothness_heavy-tailed_noise}\ref{assum:smoothness} and~\ref{assum: strongly_concave_for_y} hold, then $\pf(\x)$ is differentiable over $\realx$ and $\pgradX f(\x,\ystar(\x)) = \nabla \pf(\x)$. Moreover, $\pf(\x)$ is a $2\condnum L$-smooth function.
\end{lemma}
Lemmas~\ref{lem:prelims/smoothness_primal_function} means that we can use the gradient of primal function $\pf(\x)$ to define $\eps$-stationary point and the stochastic gradient oracle complexity of algorithms.
\begin{definition}\label{def:complexity-of-primal-stationary}
For Problem~\eqref{eq:regularized-min-max-problem}, we assume that Assumptions~\ref{assum: smoothness_heavy-tailed_noise}\ref{assum:smoothness} and~\ref{assum: strongly_concave_for_y} hold. Given $\eps>0$ (usually required to be sufficiently small), we define $\#\textsc{Grad}(\eps)$ as the minimum number of the algorithm's stochastic gradient oracle queries to achieve $\E[(1/T)\sumIter\|\nabla \pf(\xiter)\|\leq \eps]$, where $T$ is the number of iterations, and $\xiter$ is the iteration points. 
\end{definition}

Our theoretical results of \algname{SGDA} under heavy tailed noise in NC-SC setting can be presented by the following theorem.
\begin{theorem}\label[theorem]{thm:NC-SC/SGDA/thm_Complexity_SGDA}
 Let Assumption~\ref{assum: smoothness_heavy-tailed_noise} (\textbf{$L$-smoothness + heavy-tailed noise}) and~\ref{assum: strongly_concave_for_y} (\textbf{NC-SC Setting}) hold. For  \algname{SGDA} (Algorithm~\ref{alg:SGDA}) with {\color{blue} Case \romanOne}, if we set $\etay\asymeq 1/L$, $\etax\asymeq 1/\condnum^2L$, and 
 \begin{align}
  &\  T \asymeq \max\dkh{\frac{\condnum^2\Delta_1 L}{\eps^2},\frac{(L\condnum\delta_1)^2} {\eps^2}},\quad\ B\asymeq \max\dkh{\upround{\frac{\condnum^\frac{4-p}{2(p-1)}(\sqrt{\Delta_1L})^\frac{2-p}{p-1}\sigma^\frac{p}{p-1}}{\eps^\frac{2}{p-1}}},\upround{\frac{\condnum^\frac{4-p}{2(p-1)}(L\delta_1)^\frac{2-p}{p-1}\sigma^{\frac{p}{p-1}}}{\eps^\frac{2}{p-1}}},1}& \notag 
\end{align}
 with a given sufficiently small $\eps>0$, then we have
 \begin{align}
 \#\textsc{Grad}(\eps)
    \asymleq\max\dkh{\frac{\condnum^\frac{3p}{2(p-1)}(\sigma\cdot\max\dkh{\sqrt{\Delta_1L},(L\delta_1)})^\frac{p}{p-1}}{\eps^\frac{2p}{p-1}}, \frac{\condnum^2\Delta_1 L}{\eps^2},\frac{(L\condnum\delta_1)^2} {\eps^2} }.
\end{align}
\end{theorem}
The formal proof of~\Cref{thm:NC-SC/SGDA/thm_Complexity_SGDA} can be found in Appendix~\ref{sec:app-proof-of-sgda}. First, when $p=2$, i.e., under the bounded variance assumption, the gradient complexity in~\Cref{thm:NC-SC/SGDA/thm_Complexity_SGDA} degenerates to $\mathcal{O}({\condnum^3(L\delta_1\sigma)^2\Delta_1}{\eps^{-4}})$, which matches the best known upper bound of~\algname{SGDA} type methods under the bounded variance assumption for NC-SC min-max optimization, e.g.,~\cite{Two-timescale_gradient_descent_ascent_algorithms_for_nonconvex_minimax_optimization,delving_into_the_convergence_of_generalized_smooth_minmax_optimiation,tight_analysis_of_extra-gradient_and_optimistic_gradient_methods_for_nonconvex_minimax_problems,efficient_mirror_descent_ascent_methods_for_nonsmooth_minimax_problems,unified-convergennce-analysis-for-adaptive-optimization-with-moving-average-estimator}. Secondly, \textbf{when $p<2$, \Cref{thm:NC-SC/SGDA/thm_Complexity_SGDA} show the first complexity result for~\algname{SGDA} under heavy-tailed noise and NC-SC setting.}~\Cref{thm:NC-SC/SGDA/thm_Complexity_SGDA} shows that~\algname{SGDA} itself can converge under heavy-tailed noise without any modification. Specially, we found that the step size ratio $\etay/\etax$ is independent of the noise parameter $p$ and always keeps $\Theta(\condnum^2)$ in $p \in (1,2]$, which indicates that heavy-tailed noise does not actually affect the algorithm's efficiency trade-off with respect to step sizes. 

For the lower bound of the first-order algorithms under heavy-taild noise in min-max optimization, let's consider a specific objective $f(\x,\y)\eqdef F(\x)-\frac{\mu}{2}\|\y\|^2$, where $F(\x)$ can be any standard smooth function with $F(\x_1)-\min_\x F(\x)\leq\Delta_1$, and we have $\pf(\x) = \max_\y\dkh{F(\x)-(\mu/2)\|\y\|^2} = F(\x)$. Then, for any stochastic oracle $\pgradX f(\x,\y,\xi)$ satisfying the assumptions~\ref{assum: smoothness_heavy-tailed_noise}\ref{assum:unbiased} and ~\ref{assum: smoothness_heavy-tailed_noise}\ref{assum:p-BCM}, we have $\E_\xi[\pgradX f(\x,\y,\xi)] = \pgradX f(\x,\y) = \nabla F(\x)$, which means that $\pgradX f(\x,\y,\xi)$ is also an unbiased and heavy-tailed stochastic gradient estimator of $\nabla F(\x)$, and vice versa. Therefore, any first-order algorithm\footnote{The first-order algorithms here refer to the \textit{first-order zero-respecting algorithms}, see Definition 4 in~\cite{complexity_lower_bounds_for_nonconvex_strongly_concave_min_max_optimization}.} for $f$ can be seen as the first-order algorithm for $F$, regardless of the choices we make with $\etay$ and $\pgradY f(\x,\y,\xi)$. Note that $f$ satisfies the NC-SC setting, which means that $\Omega(\eps^{-\frac{3p-2}{p-1}})$ in~\cite{Zijian_Liu_nsgd} as the lower bound of the first-order algorithm for solving non-convex smooth minimization problems can be used as a lower bound of the algorithm for solving the min-max problem under the NC-SC setting. However, due to the instance we used, we cannot determine the dependence of the lower bound of complexity on other parameters such as $\condnum$ and $\sigma$. The following theorem extends the lower bound of variance-bounded noise provided by~\cite{complexity_lower_bounds_for_nonconvex_strongly_concave_min_max_optimization} to heavy-tailed noise.
\begin{theorem}\label{thm:lower_bound}
    Consider a min-max problem $\min_{\x\in\X}\max_{\y\in\Y} f(\x,\y)$, for any $p\in(1,2]$, $\sigma$, $\Delta_1$, $\condnum$, $L$ and sufficiently small $\eps>0$, there always exist $\X$, $\Y$, $f(\x,\y)$ and a stochastic oracle $\grad f(\x,\y,\xi)$, where $f(\x,\y)$ satisfies the Assumption~\ref{assum: smoothness_heavy-tailed_noise}\ref{assum:smoothness}, \ref{assum: strongly_concave_for_y}, and $\grad f(\x,\y,\xi)$ is unbiased and satisfies $\E[\|\grad f(\x,\y,\xi)-\grad f(\x,\y)\|^p]\leq\sigma^p$. Furthermore, $\pf(\mathbf{0})-\min_{\x\in\X}\pf(\x)\leq \Delta_1$. For any first-order, zero-respecting algorithm {\small \sf A} (see definition 4 in ~\cite{complexity_lower_bounds_for_nonconvex_strongly_concave_min_max_optimization}) with initial point $(\mathbf{0},\mathbf{0})$, let $\{\xiter\}_{t\geq 1}$ be generated by {\small \sf A}, then we have
    \begin{align}
        \inf_t \dkh{t\mid\E\zkh{\ell_{\pf}\|\Pi_{\mathcal{X}}(\xiter-\ell_{\pf}^{-1}\nabla \pf(\xiter))-\xiter\|}<\eps} \geq \Omega\xkh{\Delta_1L\condnum^{\frac{1}{3}}\sigma^{\frac{p}{p-1}}{\acc^{-\frac{3p-2}{p-1}}}},
    \end{align}
    where $\ell_{\pf} = cL$ is the smoothness parameter of $\pf$, and $\ell_{\pf}$ differs from $L$ by a numerical constant. 
\end{theorem}
\begin{remark}
    We notice that the lower bound depends on the condition number $\condnum$ is $\Omega(\condnum^{1/3})$, and its order is a fixed constant. This is because the proof in Appdenix~\ref{sec:proof_of_lower_bound} follows the same framework based on \textit{zero-chain instance} in~\cite{lower_bounds_for_non-convex_stochastic_optimization,complexity_lower_bounds_for_nonconvex_strongly_concave_min_max_optimization}. Under such a framework, the lower bound of the first-order (zero-respecting) algorithm is determined by ``\texttt{Effective length of the chain} $\times$ \texttt{Bound of $p$-th cenetral moment of the constructed oracle}". However, $\condnum$ will only affect the effective length, rather than the bound of the moment, this causes the dependence on $\condnum$ in lower bound we obtained cannot be unable to adapt to $p$. In fact, even for light-tailed noise, the currently best known dependence on $\condnum$ is $\mathcal{O}(\condnum)$ for {\small \sf SAPD$^+$} in~\cite{sapd+}, which is still far from the tightest known $\mathcal{O}(\condnum^{1/3})$ dependence provided by~\cite{complexity_lower_bounds_for_nonconvex_strongly_concave_min_max_optimization}. We expect to be able to close this gap in the future. 
\end{remark}
Now, let's have some discussion about the tightness of complexity $\mathcal{O}\xkh{\eps^{{-2p}/{(p-1)}}}$ in~\Cref{thm:NC-SC/SGDA/thm_Complexity_SGDA}. We note that this complexity does not reach the lower bound $\Omega\xkh{\eps^{-(3p-2)/(p-1)}}$ for general first-order algorithms when $p<2$. However, this does not mean that the complexity $\mathcal{O}\xkh{\eps^{{-2p}/{(p-1)}}}$ is loose for~\algname{SGDA}. As we have mentioned, regardless of the choice of $\etay$ and $\pgradY f(\x,\y,\xi)$, the $\eps$-complexity of \algname{SGDA} with $\etax$ for NC-SC min-max problems cannot be lower than the $\eps$-complexity of \algname{SGD} with $\eta = \etax$ for smooth minimization problems. Recently, we note that~\cite{Can-sgd-handle-heavy-tailed-noise_Ilyas_Fatkhullin, sgd-weakly-convex-heavy-tailed-noises, in-expectation-convergence-sgd-heavy-tailed_Zijian-Liu} have all studied the convergence of \algname{SGD} under heavy-tailed noise for non-convex stochastic minimization problems. \cite{Can-sgd-handle-heavy-tailed-noise_Ilyas_Fatkhullin,in-expectation-convergence-sgd-heavy-tailed_Zijian-Liu} focus on nonconvex smooth fuctions while~\cite{sgd-weakly-convex-heavy-tailed-noises} focus on weakly-convex functions. Although these works consider different specific settings, they all show that the $\eps$-complexity of \algname{SGD} under heavy-tailed noise can only reach $\mathcal{O}(\eps^{-2p/(p-1)})$. Specifically, \cite[Theorem 5.10]{Can-sgd-handle-heavy-tailed-noise_Ilyas_Fatkhullin} shows that the $\eps$-complexity of \algname{SGD} with $\eta = 1/2L$ and $B = \mathcal{O}(\eps^{-2/(p-1)})$ is upper bounded by $\mathcal{O}(\eps^{-2p/(p-1)})$ for $L$-smooth functions, and~\cite[Theorem 5.5]{Can-sgd-handle-heavy-tailed-noise_Ilyas_Fatkhullin} shows that the $\eps$-complexity of \algname{SGD} with $\eta = 1/t^{1/p}$ and $B =1$ is lower bounded by $\Omega(\eps^{-2p/(p-1)})$ for $L$-smooth functions\footnote{Note that this lower bound does not apply to {\footnotesize \sf SGD} with mini-batch, i.e., $B$ is related to $\eps$, because their oracle construction is designed for the single gradient. However, to the best of our knowledge, this lower bound is currently the only one specifically for {\footnotesize \sf SGD} under heavy-tailed noise in nonconvex smooth optimization, rather than for general first-order algorithms.}. Based on these previous works on \algname{SGD}, \textbf{we believe our $\eps$-complexity $\mathcal{O}(\eps^{-2p/(p-1)})$ of \algname{SGDA} in~\Cref{thm:NC-SC/SGDA/thm_Complexity_SGDA} is unimprovable.}

\begin{remark}\label{remark:minbatch_for_SGDA_is_nescessary}
{Here we point out that in the NC-SC setting, the need for mini-batch stochastic gradients in~\algname{SGDA}, is not a ``preferential treatment" for heavy-tailed noise.} In fact, even for light-tailed noise , mini-batch is commomly used for~\algname{SGDA}, or rather, achieving a complexity on the order of $\mathcal{O}(\eps^{-4})$ is required, e.g.,~\cite{Two-timescale_gradient_descent_ascent_algorithms_for_nonconvex_minimax_optimization,delving_into_the_convergence_of_generalized_smooth_minmax_optimiation,tight_analysis_of_extra-gradient_and_optimistic_gradient_methods_for_nonconvex_minimax_problems,efficient_mirror_descent_ascent_methods_for_nonsmooth_minimax_problems}.~\cite{Two-timescale_gradient_descent_ascent_algorithms_for_nonconvex_minimax_optimization} can prove the convergence of \algname{SGDA} with $\mathcal{O}(1)$ batch-size in the NC-SC setting, but could only obtain a complexity of $\mathcal{O}(\eps^{-5})$. Although~\cite{faster_single_loop_algo_mminmax_without_concavity-SAGDA,sapd+,alternating_proximal-gradient_steps_for_stochastic_nonconvex_concave_minmax_problems} developed the algorithms that can have the complexity of $\mathcal{O}(\eps^{-4})$ with $\mathcal{O}(1)$ batch-size under light-tailed noise, Their algorithms differ from \algname{SGDA}-type methods in that they do not compute partial gradients simultaneously, but instead use an alternating update scheme. In more detail,~\cite{alternating_proximal-gradient_steps_for_stochastic_nonconvex_concave_minmax_problems} and~\cite{faster_single_loop_algo_mminmax_without_concavity-SAGDA} used $\pgradY f(\xiterP,\yiter,\xi_\iter)$ to update $\yiterP$, while~\cite{sapd+} used $\pgradX f(\xiter,\yiterP,\xi_\iter)$ to update $\xiterP$. 
To the best of our knowledge, no work has proven the complexity $\mathcal{O}(\eps^{-4})$ of~\algname{SGDA} with $\mathcal{O}(1)$ batch-size and nonadaptive step-sizes even under light-tailed noise in the NC-SC setting; this problem has already been proposed by~\cite{Two-timescale_gradient_descent_ascent_algorithms_for_nonconvex_minimax_optimization}, and it still remains open to our knowledge. 
\end{remark}

\subsection{NC-C min-max problems}
In this section,  we study the following min-max problem
\begin{align}\label{eq:semi-regularized-min-max-problem}
    \min_{\x\in\realx}\max_{\y\in\realy} \dkh{f(\x,\y)-h(\y)+\delta_\mathcal{X}(\x)},\tag{1.2'}
\end{align}
in which \(r(\x)\) in Problem~\ref{eq:regularized-min-max-problem} is replaced by the indicator function \(\delta_{\mathcal{\X}}:\mathcal{X}\mapsto (-\infty,+\infty]\), where $\mathcal{X}$ is a nonempty, closed and convex set, while \(h(\y)\) can still be a general proper, closed and convex function. We consider Problem~\eqref{eq:semi-regularized-min-max-problem} under the NC-C setting, where \(f\) remains smooth \wrt $\x$ and $\y$ but is only required to be concave \wrt \(\y\), rather than strongly concave. Concretely, we let following Assumption~\ref{assum:prelims/nonconvex-concave-setting} holds.
\begin{assum}\label{assum:prelims/nonconvex-concave-setting}
    \begin{enumerate}[label = (\alph*)]
    \item\label{assum:prelims/bounded_dual-domain} We assume that $\dom{h}$ is bounded with diameter $\diaY$.
    \item\label{assum:prelims/dual-concavity}  $f(\x,\cdot)$ is concave \wrt $\y$ for any $\x\in\mathcal{X}$, i.e., $f(\x,\y_1)\geq f(\x,\y_2) + \inner{\pgradY f(\x,\y_2)}{\y_1-\y_2}$, where $\y_1,\y_2\in\dom{h}$.
        \item\label{assum:prelims/primal-G-lipschitz} We assume that $f(\cdot,\y)$ is $G$-Lipschitz continious on $\mathcal{X}$ \wrt $\x$ for any $\y\in\dom{h}$.  
    \end{enumerate}
\end{assum}
\begin{remark}
Assumption~\ref{assum:prelims/nonconvex-concave-setting}\ref{assum:prelims/bounded_dual-domain} guarantees that the optimal solution set of $\max_\y\dkh{f(\x,\y)-h(\y)}$ is non-empty. otherwise. For example, let $f(\y)\eqdef -1/\y + \delta_{[1,+\infty)}(\y)$, where $\delta_{[1,+\infty)}(\y) = 0$ iff $\y\in[0,+\infty)$. We have $\max_\y f(\y) = 0$, while the optimal solution set is empty.
\end{remark}
If we assume that $f$ is $L$-smooth and sastisfies Assumption~\ref{assum:prelims/nonconvex-concave-setting}, we have following lemma:
\begin{lemma}{\cite[Proposition 3.1]{alternating_proximal-gradient_steps_for_stochastic_nonconvex_concave_minmax_problems}}\label{lem:prelims/weakly-convexity-primal-function}
    Suppose that Assumptions~\ref{assum: smoothness_heavy-tailed_noise}\ref{assum:smoothness} and~\ref{assum:prelims/nonconvex-concave-setting} hold. Then, $\pf(\x)$ is Fréchet subdifferentialble and $\dkh{\pgradX f(\x,\ystar)\mid \ystar\in\argmax_\y \dkh{f(\x,\y)-h(\y)}}\subset \partial \pf(\x)$ for any $\x\in\realx$. Moreover, $\pf(x)$ is $L$-weakly convex on $\mathcal{X}$, i.e., for any $\x\in\mathcal{X}$, $\pf(\x)+(L/2)\|\x\|^2$ is convex, which is equivalent to saying that $\pf(x_1)-\pf(\x_2)\geq \inner{g_\pf(\x_2)}{\x_1-\x_2}-(L/2)\|\x_1-\x_2\|^2$ holds for any $\x_1,\x_2\in\mathcal{X}$ and $g_\pf(\x_2)\in\partial \pf (\x_2)$. 
\end{lemma}
For a weakly convex function, we can define its \textit{Moreau envelope} as following:
\begin{align}
    \pf_\rho (\x)\eqdef \pf(\x^+),\; where\; \x^+\eqdef \argmin_{\x'\in\mathcal{X}} \dkh{\pf(\x')+(\rho/2)\|\x'-\x\|}, \forall \rho>L.\notag
\end{align}
By~\cite[Section 1]{sgd-weakly-convex-functions-damek_davis-2018}, $\pf_\rho (\x)$ is smooth with gradient $\nabla \pf_\rho (\x) = \rho (\x-\x^+)$. Moreover, let $\norCone{\mathcal{X}}{\x^+}$ is the normal cone of $\mathcal{X}$ defined on $\x^+$, $\|\nabla \pf_\rho (\x)\|\leq \eps$ implies that $\dist{0,\norCone{\mathcal{X}}{\x^+}+\partial\pf(\x^+)}\leq \eps$ and $\|\x^+-\x\|\leq \eps/\rho$, which means we can use $\nabla \pf_\rho (\x)$ to define convergence. Specifically, we have following definition:
\begin{definition}\label{def:complexity-of-near-primal-stationary}
For Problem~\eqref{eq:semi-regularized-min-max-problem}, we assume that Assumptions~\ref{assum: smoothness_heavy-tailed_noise}\ref{assum:smoothness} and~\ref{assum:prelims/nonconvex-concave-setting} hold. Given $\eps>0$ (usually required to be sufficiently small), we define $\#\textsc{Grad}(\eps,\rho)$ as the minimum number of the algorithm's stochastic gradient oracle queries to achieve $\E[(1/T)\sumIter\|\nabla \pf_\rho (\xiter)\| \leq \eps]$, where $T$ is the number of iterations, and $\xiter$ is the iteration points. 
\end{definition}

Our theoretical results of \algname{SGDA} under heavy tailed noise in NC-C setting can be presented by the following theorem.
\begin{theorem}\label{thm:UpperBound_SPGDA_NC-C}
    Let Assumptions~\ref{assum: smoothness_heavy-tailed_noise} (\textbf{$L$-smoothness + heavy-tailed noise}) and~\ref{assum:prelims/nonconvex-concave-setting} (\textbf{NC-C Setting}) hold. For \algname{SGDA} (Algorithm~\ref{alg:SGDA}) with {\color{blue} Case \romanTwo}, if 
    \begin{align}
       & \etax \asymeq \min\dkh{\frac{(\Delta_1/TL)^\frac{2p-1}{3p-2}}{(\diaY\sigma)^\frac{p}{3p-2}(G\sigma\mmax G^2)^\frac{p-1}{3p-2}},\sqrt{\frac{\Delta_1/TL}{(G^2\mmax\sigma^2\mmax G\sigma)}}},\etay \asymeq \min\dkh{\frac{\diaY^\frac{3p-4}{3p-2}}{\sigma^\frac{3p}{3p-2}}\xkh{\frac{\Delta_1(G\sigma\mmax G^2)}{TL}}^\frac{1}{3p-2},\frac{1}{2L}},\notag
    \end{align}
    where $a\mmax b$ represents $\max\{a,b\}$, then we have
    \begin{align}
        \#\textsc{Grad}(\eps,2L) &\asymleq \frac{\Delta_1(G\sigma\mmax\sigma^2)(\diaY\sigma)^\frac{p}{p-1}L^\frac{2p-1}{p-1}}{\eps^\frac{6p-4}{p-1}} + \frac{L^\frac{7p-5}{2p-1}\diaY^\frac{5p-4}{2p-1}}{\sigma^\frac{p}{2p-1}}\xkh{\frac{\Delta_1(G^2\mmax G\sigma)}{\eps^\frac{12p-8}{2p-1}}}\nnl 
        & \quad + \xkh{\frac{\sigma}{\eps}}^\frac{p}{p-1} + \frac{\Delta_1 L(G^2\mmax\sigma^2\mmax G\sigma)}{\eps^4} + \frac{L\whatdelta_1}{\eps^2},
    \end{align}
    where $\whatdelta_1\eqdef \pf(\x_1)-\whatf(\x_1,\y_1)$.
    If, in addition, we have $\sigma\geq G\mmax 1$, $L,\diaY\geq 1$ and $\eps\leq 1$, we can further get
    \begin{align}
      \#\textsc{Grad}(\eps,2L) \asymleq   \frac{\Delta_1\diaY^\frac{p}{p-1}L^\frac{2p-1}{p-1}(\sigma^\frac{3p-2}{p-1}\mmax G^2\sigma^\frac{p}{p-1})}{\eps^\frac{6p-4}{p-1}} + \xkh{\frac{\sigma}{\eps}}^\frac{p}{p-1} + \frac{\Delta_1L(G^2\mmax\sigma^2)}{\eps^4} + \frac{L\whatdelta_1}{\eps^2}.
    \end{align}
\end{theorem}
The proof of Theorem~\ref{thm:UpperBound_SPGDA_NC-C} can be found in Appdenix~\ref{sec:proof_of_sgda_nc-c}. Here we point out that conditions $\sigma\geq G\mmax 1$, $L,\diaY\geq 1$ and $\eps\leq 1$ in Theorem~\ref{thm:UpperBound_SPGDA_NC-C} is mild and is used only to simplify the upper bound. When $p=2$, the complexity $\mathcal{O}(\Delta_1\diaY^2L^3(\sigma^4\mmax G^2\sigma^2)\eps^{-8})$ mathches~\cite{on_gradient_descent_ascent_for_nonconvex_concave_minmax_problems}'s result for light-tailed noise; \textbf{when \(p < 2\), Theorem~\ref{thm:UpperBound_SPGDA_NC-C} provides the first complexity upper bound for \algname{SGDA} under heavy-tailed noise in the NC-C setting.}
\begin{table}
    \centering
    \small 
    \begin{tabular}{|c|c|c|c|c|}
    \hline 
        \textbf{NC-C.} & \textbf{Reg.} & \textbf{Metric} & ($\E$,$\mathbb{P}$) & $\#\textsc{Grad}(\eps,2L)$\\\hline
        \hline 
       \algname{SGDA}$_{\text{\tiny Clip}}$~\cite{high-probability_convergence_guarantees_of_stochastic_gradient_descent_ascent_in_structured_nonconvex_min-max_games}  & ($\R^d$,$\delta_{\mathcal{Y}}$) & $\frac{1}{T}\sumIter \|\nabla \pf_{1/2L}(\xiter)\|^2$ & (\yes,\yes) & $\mathcal{O}\xkh{\frac{(\Delta_1+\widehat{\delta}_1)L^\frac{2p-1}{p-1}\diaY^\frac{p}{p-1}(G^2\sigma^\frac{p}{p-1}+\sigma^\frac{p^2}{p-1})}{\eps^\frac{6p-4}{p-1}}}$ \\\hline
        \algname{SGDA} (Theorem~\ref{thm:UpperBound_SPGDA_NC-C}) & ($\delta_\mathcal{X}$,$h$) & $\xkh{\frac{1}{T}\sumIter \|\nabla \pf_{1/2L}(\xiter)\|^2}^{\nicefrac{p}{2}}$ & (\yes,\no) & $\mathcal{O}\xkh{\frac{\Delta_1L^\frac{2p-1}{p-1}\diaY^\frac{p}{p-1}(G^2\sigma^\frac{p}{p-1}\mmax{\color{blue} \sigma^\frac{3p-2}{p-1}})}{\eps^\frac{6p-4}{p-1}}+{\color{blue} \frac{L\whatdelta_1}{\eps^2}}}$\\ \hline 
    \end{tabular}
    \caption{\small \textbf{Comparison of the results of {\color{blue} vanilla} \algname{SGDA} in this paper and \algname{SGDA}$_{\text{\tiny Clip}}$ in~\cite{high-probability_convergence_guarantees_of_stochastic_gradient_descent_ascent_in_structured_nonconvex_min-max_games} using gradient clipping.} ``$(a,b)$'' in the ``Reg.'' column denotes the regularizer \(a\) for \(\x\) and the regularizer \(b\) for \(\y\). ``Metric'' denotes the convergence metric, ``$\E$'' denotes convergence in expectation, and ``$\mathbb{P}$'' denotes convergence with high probability. For the gradient complexity, We note that although the dependence on \(\eps\) is the same for both methods, our dependence on \(\sigma\) and \(\whatdelta_1\) is tighter.}
    \label{tab:comparision_SGDA_(Clip)_NC-C}
\end{table}
\paragraph{Comparison with related work} To the best of our knowledge, \cite{high-probability_convergence_guarantees_of_stochastic_gradient_descent_ascent_in_structured_nonconvex_min-max_games} is currently the only work that studies minimax optimization under heavy-tailed noise in the NC-C setting. The method they developed, \algname{SGDA}$_{\text{\tiny Clip}}$ , can be viewed as \algname{Stoc-AGDA}~\cite{faster_single_loop_algo_mminmax_without_concavity-SAGDA} equipped with gradient clipping. That is, \(\xiterP = \xiter-\widehat{\nabla}_{\x}f(\xiter,\yiter,\xi_\iter),\yiterP = \yiter - \widehat{\nabla}_{\y}f(\xiterP,\yiter,\xi_\iter)\), where \(\widehat{\nabla}_{\x}f(\xiter,\yiter,\xi_\iter)\) and \(\widehat{\nabla}_{\y}f(\xiterP,\yiter,\xi_\iter)\) are both the clipped gradients. Note that this is somewhat different from \algname{SGDA}'s strategy of simultaneously updating partial gradients. They show that this method converges with an \(\eps\)-complexity of \(\mathcal{O}(\eps^{-\frac{6p-4}{p-1}})\). \textbf{It is worth noting that this complexity is the same as the \(\eps\)-complexity we obtain for the vanilla \algname{SGDA} without using  gradient clipping and gradient normalization. In particular, compared with~\cite{high-probability_convergence_guarantees_of_stochastic_gradient_descent_ascent_in_structured_nonconvex_min-max_games}, we do not introduce any additional assumptions.} 
{As can be seen from Table 1, for min-max optimization under the NC-C setting, gradient clipping does not improve the complexity but instead improves the convergence metric}; that is, compared with \algname{SGDA}, \algname{SGDA}$_\text{\tiny Clip}$ can achieve high-probability convergence \wrt $(1/T)\sumIter \|\nabla \pf_{1/2L}(\xiter)\|^2$. Note that~\cite[Proposition 4.7]{high-probability_convergence_guarantees_of_stochastic_gradient_descent_ascent_in_structured_nonconvex_min-max_games} proved that even for \algname{SGDA} under the NC-SC setting, the dependence of the complexity on the failure probability $\delta$ under heavy-tailed noise cannot be polylogarithmic.

\subsection{Proof sketches of~\Cref{thm:NC-SC/SGDA/thm_Complexity_SGDA} and Theorem~\ref{thm:UpperBound_SPGDA_NC-C}}\label{sec:proof_skectch_of_SGDA}
Now let's explain how we analyze the convergence of \algname{SGDA} under heavy-tailed noise. 
\paragraph{Weaker convergence metrics} Inspired by~\cite{in-expectation-convergence-sgd-heavy-tailed_Zijian-Liu} on stochastic minimization problems, we replaced the convergence metric from $(1/T)\sumIter\E[\|\nabla \pf(\xiter)\|^2]$ to a weaker one, $\E[((1/T){\sumIter \|\nabla \pf(\xiter)\|^2})^{p/2}]$  (In the NC-C setting, We will replace $\|\nabla \pf(\xiter)\|$ with $\|\nabla \pf_{1/2L}(\xiter)\|$.). Note that, by the Jensen's inequality, we have $ \xkh{(1/T)\sumIter \E[\|\nabla \pf(\xiter)\|]}^p\leq (1/T)\sumIter \E[\|\nabla \pf(\xiter)\|^p] \leq \E[((1/T){\sumIter \|\nabla \pf(\xiter)\|^2})^{p/2}]$. Therefore, if $\E[((1/T){\sumIter \|\nabla \pf(\xiter)\|^2})^{p/2}] \leq \eps^p$ holds, it implies that $(1/T)\sumIter \E[\|\nabla \pf(\xiter)\|] \leq \eps$ also holds. 

Our analysis also relies on such a lemma: \textit{for a given martingale difference sequence $\{X_k\}_{k=1}^N$ \wrt filterations $\dkh{\fil_\iter}_{t\geq 0}$, if $\Exp{|X_i|^p}<\infty$, we have $\Exp{\left |\sum_{k=1}^N X_k \right|^p} \asymleq \sum_{k=1}^N \Exp{|X_k|^p}$.} (A more formal statement can be found in~\Cref{lem:app/Supp_lems/p_moment_sum_martingales})
We noticed that the vector version extended from this lemma, i.e., $\Exp{\left \|\sum_{i=1}^n X_i \right\|^p} \asymleq \sum_{i=1}^n \Exp{\|X_i\|^p}$, has been proposed by~\cite{Accelerated-Zeroth-order-Method-for_Kornilov,From_Gradient_Clipping_to_Normalization,Zijian_Liu_nsgd}, primarily used to analyze the upper bound of $p$-th central moments of mini-batch stochastic gradients. {Differently, in our analysis, we utilize~\Cref{lem:app/Supp_lems/p_moment_sum_martingales} to analyze the upper bound of the sum of martingale sequences.} 

For convenience, we additionally introduce the following notation:
\begin{itemize}
    \item For \algname{SGDA} with {\color{blue} Case \romanOne}, we denote $\zeta_{\x,t}$ and $\zeta_{\y,t}$ as $g_{\x,t}-\pgradX f(\xiter,\yiter)$, $g_{\y,t}-\pgradY f(\xiter,\yiter)$ respectively, and let $\var^p \eqdef\max_{t\in\{1,2,\hdots,T\}}\dkh{\Exp{\norm{\zeta_{\x,t}}^p}, \Exp{\norm{\zeta_{\y,t}}^p}}$, and $\mathcal{F}_{\iter}\eqdef \sigma(\overline{\xi}_1,\hdots,\overline{\xi}_{\iter})$, where $\overline{\xi}_{\iter} \eqdef \xi_{t,1},\hdots,\xi_{t,B}$.
     \item For \algname{SGDA} with {\color{blue} Case \romanTwo}, we denote $\zeta_{\x,t}$ and $\zeta_{\y,t}$ as $g\pgradX f(\xiter,\yiter,\xi_\iter)-\pgradX f(\xiter,\yiter)$, $\pgradY f(\xiter,\yiter,\xi_\iter)-\pgradY f(\xiter,\yiter)$ respectively. Moreover, we let $\whatf(\x,\y)$ be $f(\x,\y)-h(\y)$ and $\whatdelta_\iter$ be $\pf(\xiter)-\whatf(\xiter,\yiter)$. We let $\mathcal{F}_\iter\eqdef \sigma(\xi_1,\hdots,\xi_\iter)$. 
\end{itemize}
Here, we rule $\mathcal{F}_0\eqdef \{\varnothing\}$. 
\paragraph{Proof sketch of~\Cref{thm:NC-SC/SGDA/thm_Complexity_SGDA}}
Since $\pf$ is $2\condnum L$-smooth, starting from the qudratic upper bound of the smooth function, we can then obtain $\sumIter \etax\|\nabla \pf(\xiter)\|^2 \asymleq \sumIter \etax(2\condnum L\etax-1)\inner{\nabla \pf(\xiter)}{\errorxiter}+\Delta_1 +\sumIter \etax^2\|\errorxiter\|^2 + \sumIter \etax\delta_\iter^2$. Here we ignore $\condnum$ and $L$ without affecting the convergence analysis.
Note that at this point we cannot absorb $\inner{\nabla \pf(\xiter)}{\errorxiter}$ by taking the expectation, because $\E[\|\errorxiter\|^2]$ could be infinite. We first raise both sides of the inequality to the power of $p/2$, and then take the expectation to obtain 
\begin{align}
    \E[(\sumIter \etax\|\nabla \pf(\xiter)\|^2)^{p/2}]\asymleq \underbrace{\E[(|\sumIter \etax(2\condnum L\etax-1)\inner{\nabla \pf(\xiter)}{\errorxiter}|)^{p/2}]}_{\sf M} + \Delta_1^{p/2}+\etax^p T\var^p + \E[(\sumIter \etax\delta_t^2)^{p/2}].\notag
\end{align}
For {\sf M}, note that $\{\inner{\nabla \pf(\xiter)}{\errorxiter}\}_{t\geq 0}$ is indeed a martingale difference sequence \wrt $\{\mathcal{F}_\iter\}_{t\geq 0 }$, using~\Cref{lem:app/Supp_lems/p_moment_sum_martingales} and  Hölder's inequality, we can get 
\begin{align}
    {\sf M} \asymleq (1/4)\E[(\sumIter\etax\|\nabla \pf(\xiter)\|^2)^{p/2}] + \var^p T^{(2-p)/p}\etax^{p/2}\notag 
\end{align}
We now determine the upper bound of $\E[(\sumIter \etax\delta_t^2)^{p/2}]$ by analyzing a one-step stochastic gradient ascent on $f(\xiter,\cdot)$. We point out that the techniques used to analyze the convergence of \algname{SGDA} under light-tailed noise are difficult to extend to our setting. Specifically, they all first obtain $\E[\delta^2_\iter] \asymleq \gamma \E[\delta^2_\iterM] + \etax^2\E[\|\nabla \pf(\xiterM)\|^2] + \etax^2\var^2 + \var^2$, where $\gamma\leq 1$, and then obtain the upper bound of $\E[\delta_t^2]$ by recursively applying the above inequality, thus obtaining the upper bound of its sum. However, we were unable to control $\E[\|\zeta_{\x,t}\|^2]$ and $\E[\|\zeta_{\y,t}\|^2]$ from the beginning of the analysis due to the heavy-tailed noise. In fact, we first obtained the inequality
\begin{align}
    \etax\delta_\iter^2\leq \beta \etax\delta^2_\iterM + B_\iterM + \theta X_\iter +\alpha Y_\iterM, \notag
\end{align}
where $X_\iterM\eqdef \langle \nabla \pf(\xiterM),\zeta_{\x,\iterM} \rangle, Y_\iterM \eqdef \langle \zeta_{\y,\iterM},\yiterM-\ystar_\iterM \rangle$ and the definition of $\beta, B_\iterM, \alpha$ and $\theta$ can be seen in~\eqref{eq:app/Analysis_SGDA_v2/Exp_dual_error_SGDA/notaions}. Then, We recursively apply the obtained inequality, and then raise both sides of the inequality to the power of $p/2$, and take the expectation to get
\begin{align}
    \E[(\sumIter \etax\delta_t^2)^{p/2}] &\asymleq \xkh{\frac{\delta_1^2}{\etay}}^{p/2} + \E[(\sumIter \etax\|\nabla \pf(\xiter)\|^2)^{p/2}] +T\etax^{p/2}\var^p\nnl &+ \E[|{\sumIterM(\sum_{s=0}^{T-1-t}\beta^s) \theta X_t}|^{p/2}] + \E[|{\sumIterM({\sum_{s=0}^{T-1-t}\beta^s})\alpha Y_t}|^{p/2}].\notag
\end{align}
Next, since that $\{X_\iter\}_{t\geq 1}$ and $\{Y_\iter\}_{t\geq 1}$ are both martingale sequences \wrt $\{\mathcal{F}_t\}_{t\geq 0}$, by~\Cref{lem:app/Supp_lems/p_moment_sum_martingales} again, we can get following inequality: 
\begin{align}
    &\ \E[|{\sumIterM(\sum_{s=0}^{T-1-t}\beta^s) \theta X_t}|^{p/2}] + \E[|{\sumIterM({\sum_{s=0}^{T-1-t}\beta^s})\alpha Y_t}|^{p/2}] & \nnl 
    &\ \asymleq \E[(\sumIter \etax\|\nabla \pf(\xiter)\|^2)^{p/2}] + \E[(\sumIter \etax\delta_t^2)^{p/2}] + \frac{(\theta^p+(\alpha/L)^p)\var^pT^{1-p/2}}{((1-\beta)\sqrt{\etax})^p}.\notag
\end{align}
This leads to  
\begin{align*}
    \E[(\sumIter \etax\delta_t^2)^{p/2}] \asymleq \E[(\sumIter \etax\|\nabla \pf(\xiter)\|^2)^{p/2}] + \xkh{\frac{\delta_1^2}{\etay}}^{p/2} + \etax^{p/2}T\var^p+T^{1-p/2}\var^p\etax^{p/2}.
\end{align*}
Putting these pieces together, we finally have 
\begin{align}
    \E[((1/T)\sumIter \|\nabla \pf(\xiter)\|^2)^{p/2}] \asymleq \frac{1}{(T\etax)^{p/2}} + \etax^{p/2}T^{1-p/2}\var^p + T^{1-p/2}\var^p + \frac{\var^p}{T^{p-1}}. \notag
\end{align}
Note that if we set $B=1$, i.e., $\var=\sigma$, the variance term $T^{1-p/2}\var^p$ cannot be controlled, which is why we use mini-batches because $\var = \sigma/B^{(p-1)/p}$.
\paragraph{Proof sketch of Theorem~\ref{thm:UpperBound_SPGDA_NC-C}}
First, we attempt to derive a descent inequality for  envelope $\pf_{1/2L}(\xiter)$ of primal function $\pf(\xiter)$ under heavy-tailed noise. This part of the analysis is similar to that of \algname{SGD} under heavy-tailed noise for weakly convex (stochastic) optimization. We note that~\cite{sgd-weakly-convex-heavy-tailed-noises} obtained the first complexity upper bound for \algname{SGD} in weakly convex optimization, but they considered optimization problems with constraints given by a compact convex set. In our setting, however, we only require the constraint region for \(\x\) to be a general closed and convex set, without requiring it to be bounded. \textbf{Therefore, \cite{sgd-weakly-convex-heavy-tailed-noises}'s technique cannot be used to the analysis of Theorem~\ref{thm:UpperBound_SPGDA_NC-C}.} Similar to the analysis of Theorem~\ref{thm:NC-SC/SGDA/thm_Complexity_SGDA}, we first sum the  descent inequality, then raise both sides to the power of \(p/2\), and finally take the expectation of the resulting inequality to obtain
\begin{align}
      \Exp{\xkh{\sumIter{\etax\|\nabla \pf_{1/2L}(\xiter)\|^2}}^{\nicefrac{p}{2}}} & \asymleq  \underbrace{\Exp{\xkh{\abs{\sumIter \etax\inner{\nabla \pf_{1/2L}(\xiter)}{\errorxiter}}}^{\nicefrac{p}{2}}}}_{\sf M'}\nnl 
     &\quad +(L\etax)^{\nicefrac{p}{2}}\Exp{\xkh{\sumIter \whatdelta_\iter}^{\nicefrac{p}{2}}} + \Delta_1^{\nicefrac{p}{2}} + + (TL)^{\nicefrac{p}{2}}\etax^p(G^p+\sigma^p),\notag
\end{align}
where $\whatf(\x,\y)\eqdef f(\x,\y)-h(\y)$ and $\whatdelta_\iter \eqdef \pf(\xiter)-\whatf(\xiter,\yiter)$.
For ${\sf M'}$, note that \(\{\etax\inner{\nabla \pf_{1/2L}(\xiter)}{\errorxiter}\}_{t\geq 1}\) is a martingale difference sequence \wrt \(\{\mathcal{F}_t\}_{t\geq0}\), where $\mathcal{F}_t\eqdef \sigma(\xi_1,\hdots,\xi_t)$. Therefore, by applying Lemma~\ref{lem:app/Supp_lems/p_moment_sum_martingales} and Hölder's inequality, we obtain:
\begin{align}
    {\sf M'}\asymleq   \Exp{\xkh{\sumIter{\etax\|\nabla \pf_{1/2L}(\xiter)\|^2}}^{\nicefrac{p}{2}}} + \sigma^p T^{1-p/2}\etax^{\nicefrac{p}{2}}\notag
\end{align}
For the function value gap $\Exp{({\sumIter \whatdelta_\iter})^{\nicefrac{p}{2}}}$, we first obtain the following inequality
\begin{align}
    &\ \Exp{ \whatdelta_\iter}\asymleq \Exp{\pf(\xiterM) -\whatf(\xiterM,\y)} + \frac{\Exp{\|\yiterM-\y\|^2-\|\yiter-\y\|^2}}{\etay}+ \Exp{\whatf(\xiter,\yiter) -\whatf(\xiterM,\yiterM)}\nnl
    &\ + \diaY^{2-p}\etay^{p-1}\sigma^p + G(\sigma+G)\etax
\end{align}
for any $\mathcal{F}_{t-2}$-measurable $\y$. Note that if we let $\y = \yiterMstar$, then we cannot bound $\sumIter \E[\|\yiterM-\yiterMstar\|^2-\|\yiter-\yiterMstar\|^2]$ by telescope summation. Inspired by~\cite{on_gradient_descent_ascent_for_nonconvex_concave_minmax_problems}, we partition the sequence \(\{\whatdelta_\iter\}_{t=1}^T\) into \(N\) subsequences \(\{\whatdelta_{nM+j}\}_{j=1}^M\), where $n=0,\hdots,N-1$. For the \(n+1\)-th subsequence \(\{\whatdelta_{nM+j}\}_{j=1}^M\), let \(\y = \ystar_{nM+1}\), then obtain:
\begin{align}
    \sumIter \Exp{\whatdelta_t} =\sum_{n=1}^N\sum_{j=1}^M \whatdelta_{nM+j} \leq \whatdelta_1+ (M+1)TG(\sigma+G)\etax + T\diaY^{2-p}\etay^{p-1}\sigma^p + \frac{T\diaY^2}{2M\etay}.\notag
\end{align}
Here we can set $M = \upround{\nicefrac{\diaY}{\sqrt{2G(G+\sigma)\etax\etay}}}$ to further simplify the upper bound. Note that $\E[{({\sumIter \whatdelta_\iter})^{\nicefrac{p}{2}}}] \leq \xkh{\sumIter \E[{\whatdelta_t}]}^{\nicefrac{p}{2}}$ for $p\in(1,2]$.
Finally, by combining all the pieces we obtained and choosing the appropriate step sizes, we complete the proof of Theorem~\ref{thm:UpperBound_SPGDA_NC-C}.

\section{({\sf Stoc-}){\sf TRGDA} for NC-SC regularized min-max problems}\label{sec:trgda_for_nc-sc}
Since we have already shown that under heavy-tailed noise, \sgda can only achieve complexity $\mathcal{O}(\eps^{-\frac{2p}{p-1}})$ in the NC-SC setting, there is clearly a gap compared with the lower bound $\Omega(\eps^{-\frac{3p-2}{p-1}})$. Therefore, the goal of this section is to design algorithms that can converge with complexity $\mathcal{O}(\eps^{-\frac{3p-2}{p-1}})$ under heavy-tailed noise. In particular, it can handle the \textbf{regularized minimax problem}, i.e. $h$ and $r$ in Problem~\ref{eq:regularized-min-max-problem} are both generally proper, closed and convex functions. Additionally, we do not wish to use modules such as gradient clipping that require introducing an additional hyperparameter.
\subsection{How to incorporate normalized gradient into {\sf SPGD}?}\label{sec:how_use_normalized_gradient_to_composite_optimization}
To better understand how we design the algorithm, we first consider the composite minimization, i.e., $\min_\x \{f(\x)+r(\x)\}$. Furthermore, we hope to incorporate gradient normalization technique into stochastic proximal gradient method (\algname{SPGD}). We can do this in the following two ways:
\begin{itemize}
    \item \textit{Proximal normalized gradient.} A natural way is to consider $\xiterP = \prox{\eta_\iter r}{\xiter - \frac{\eta\nabla f(\xiter,\xi_\iter)}{\|\nabla f(\xiter,\xi_\iter)\|}}$. However, this way is naive even for deterministic settings because it would violate the \textit{fixed-point property} of the proximal operator. That is, let $\x_\star =\argmin_\x \{f(\x)+r(\x)\} $, we have $\x_\star \neq \prox{\eta_\iter r}{\x_\star - \frac{\eta\nabla f(\x_\star)}{\|\nabla f(\x_\star)\|}}$, which precludes the sequence of iteration points from becoming stationary around $\x_\star$. 
    \item \textit{Normalized gradient mapping.} Another way of incorporation is to consider the normalized gradient mapping, i.e., \(\xiterP = \xiter-\eta\mathcal{G}(\xiter)/\|\mathcal{G}(\xiter)\|\), where \(\mathcal{G}(\xiter) \eqdef \eta_t^{-1}(\xiter-\widetilde{\x}_\iterP), \widetilde{\x}_\iterP\eqdef \prox{\eta_t r}{\xiter-\eta_t\nabla f(\xiter,\xi_\iter)}\). We still first consider the deterministic setting, i.e., $\nabla f(\xiter,\xi_\iter) = \nabla f(\xiter)$. Assuming that \(f\) is smooth and nonconvex, we then have
    \begin{align}
        f(\xiterP) + r(\xiterP) &\leq f(\xiter) + r(\xiterP) -\eta\inner{\nabla f(\xiter)}{\frac{\mathcal{G}(\xiter)}{\|\mathcal{G}(\xiter)\|}} + \frac{L\eta^2}{2} \nnl 
        &=f(\xiter) -\eta\|\mathcal{G}(\xiter)\| +\eta\inner{g_r(\widetilde{\x}_\iterP)}{\frac{\mathcal{G}(\xiter)}{\|\mathcal{G}(\xiter)\|}} + r(\xiterP)+ \frac{L\eta^2}{2}\nnl 
        &\leq f(\xiter) + r(\xiter) -\eta\|\mathcal{G}(\xiter)\| {\color{blue} - \eta\inner{g_r(\xiterP)-g_r(\widetilde{\x}_\iterP)}{\frac{\mathcal{G}(\xiter)}{\|\mathcal{G}(\xiter)\|}}} + \frac{L\eta^2}{2},
    \end{align}
    where $g_r(\xiterP)\in\partial r(\xiterP)$ and $g_r(\widetilde{\x}_\iterP) \in\partial r(\widetilde{\x}_\iterP)$. The equality holds becasue $\mathcal{G}(\xiter) = \nabla f(\xiter)+g_r(\widetilde{\x}_\iterP)$ by the optimality condition, and the last inequalty holds becasue $r(\xiterP)-r(\xiter)\leq \inner{g_r(\widetilde{\x}_\iterP)}{\xiterP-\xiter}+\inner{g_r(\xiterP)-g_r(\widetilde{\x}_\iterP)}{\xiterP-\xiter}$ by the convexity of $r$. We can regard \(-\eta\|\mathcal{G}(\xiter)\|\) as the descent term, but for the {\color{blue} blue term$\leq \eta\|g_r(\xiterP)-g_r(\widetilde{\x}_\iterP)\|$}  it is difficult to upper-bound it when \(r\) is merely convex. 
 \end{itemize}
 
Therefore, based on the above discussion, we did not adopt these two seemingly natural ways because they have analytical limitations even in deterministic settings. In fact, we are considering the following update:
\begin{align}\label{eq:trust_region_update}
    \xiterP \in \argmin_{\x\in\ball{\xiter}{\eta_\iter}}\{r(\x)+\inner{\nabla f(\xiter)}{\x}\}\notag
\end{align}
It is easy to observe that the above update is equivalent to replacing the quadratic term $\eta_\iter\|\x-\xiter\|^2/2$ in the proximal operator with a explicit constraint, i.e., $\|\x-\xiter\|\leq \eta_\iter$, and it reduces to \algname{NGD} when $r$ is constant. We point out that the above update has actually been proposed by~\cite{understand_grad_ortho_trust_region_opt} for designing first-order methods for stochastic composite minimization problems, which they call \textit{trust-region gradient methods} (\algname{TRGM}). However,~\cite{understand_grad_ortho_trust_region_opt} did not focus on heavy-tailed noise; their main motivation was to understand orthogonal gradient methods serving matrix optimization, i.e., where $\xiter$ is a $2$D parameter\footnote{Clearly, if we only consider composite minimization in a vector space under light-tailed noise, the stochastic proximal gradient method is already sufficient to solve the problem.}. {Consequently, they did not establish a connection between \algname{TRGM} and \algname{NGD}. For \algname{TRGM}, its equivalence to \algname{NGD} is obvious when $r$ is constant. Actually, both proximal normalized gradient and normalized gradient mapping can reduce to \algname{NGD} when $r$ is constant. Here we emphasize that we are more interested in the case where $r$ is a general nondifferentiable convex function.} More importantly, they did not discuss the case where the optimal solution set for $\min_{\x\in\ball{\xiter}{\eta_\iter}}\{r(\x)+\inner{\nabla f(\xiter)}{\x}\}$ is a singleton. The following theorem answers these two questions.
\begin{theorem}\label[theorem]{thm:NC-SC/TR-SGDAM/trust-region-proximal_is_gradient_normalized}
Supposed that $r:\real{d}\mapsto (-\infty,+\infty]$ is a proper, closed and convex function. Given any $\z\in\dom{r}$, $\eta>0$ and $\bv\in\real{d}$, then {$\mathcal{X}^\dag_\star \eqdef\argmin_{\x\in\ball{\z}{\eta}} \dkh{\inner{\bv}{\x}+r(\x)}$ must be a nonempty, closed and convex set}. Furthermore,
if $\bm{0}\notin \bv+\partial r(\x^\dag)$, then
    \begin{align}
        \|\x^\dag-\z\|=\eta\; and \;\exists g_r(\x^\dag)\in\partial r(\x^\dag),\; \x^\dag = \z-\eta\frac{\bv+g_r(\x^\dag)}{\|\bv+g_r(\x^\dag)\|}.
    \end{align}
    If $\forall \x^\dag \in \mathcal{X}^\dag_\star$, we have $\bm{0}\notin \bv+\partial r(\x^\dag)$, then $\mathcal{X}^\dag_\star$ must be a singleton.
\end{theorem}
The proof of~\Cref{thm:NC-SC/TR-SGDAM/trust-region-proximal_is_gradient_normalized} can be found in Appendix~\ref{secs:app/proof_trust-region-proximal_is_gradient_normalized}. Let $\bv = \nabla f(\z)$, where $f$ is $L$-smooth, if $\exists \x^\dag \in \mathcal{X}^\dag_\star$, such that $\bm{0}\in\nabla f(\z)+\partial r(\x^\dag)$, then inequality $\dist{0,\nabla f(\x^\dag)+\partial r(\x^\dag)} \leq \|\nabla f(\x^\dag)-\nabla f(\z)\|\leq L\eta$ holds, which means that $\x^\dag$ is a $L\eta$-stationary point of $f+r$. Therefore, \Cref{thm:NC-SC/TR-SGDAM/trust-region-proximal_is_gradient_normalized} tells us that \algname{TRGM} can basically be seen as a computable approximation of $\xiterP = \xiter - \nicefrac{\eta_\iter (\nabla f(\xiter)+g_r(\xiterP))}{\|\nabla f(\xiter)+g_r(\xiterP)\|}$.
Inspired by \algname{TRGD}, we can design a \textit{trust-region gradient descent ascent} (\algname{TRGDA}) method for regularized minimax optimization, i.e.,
\begin{align}
    \begin{cases}
        \xiterP \in \argmin_{\x\in\ball{\xiter}{\etax}}\{\inner{\pgradX f(\xiter,\yiter)}{\x}+r(\x)\}; \\
        \yiterP \in \argmin_{\y\in\ball{\yiter}{\etay}}\{\inner{-\pgradY f(\xiter,\yiter)}{\y}+h(\y)\}.
    \end{cases}\notag
\end{align}
If we replace the partial gradients with the stochastic partial gradients, then \algname{TRGDA} becomes \textit{stochastic trust-region gradient descent ascent} (\algname{Stoc-TRGDA}). \textbf{To the best of our knowledge, \algname{(Stoc-)TRGDA} is novel.} We present \algname{Stoc-TRGDA} as a basic stochastic optimization framework, while our actual analysis focuses on its variants in the next section, such as replacing the stochastic gradient with \textit{Polyak momentum}.
\begin{remark}
    We note some recent work with ideas similar to \algname{(Stoc-)TRGD}.~\cite{broximal_alignment_for_global_non_convex_optimization,the_ball-proximal_point_method_new_algorithm_convergence_theory_and_applications,non-euclidean_broximal_point_method_a_blueprint_for_geometry_aware_optimization} developed \textit{Broximal Point Methods} (\algname{BPM}) for noncomposite optimization. \algname{BPM} updates $\xiterP \in \argmin_{\x\in\ball{\xiter}{\eta_\iter}} f(\x)$, and \algname{(Stoc-)TRGD} can be viewed as its linear (stochastic) approximation when $r$ is constant.~\cite{stabilized_proximal_point_method_via_trust_region_control} proposed a \textit{Trust-Region Proximal Point Method} (\algname{TRPPM}): ${\xiterP = \argmin_{\x\in\ball{\xiter}{\eta_\iter}} \{f(\x) + \lambda_\iter/2\|\x-\xiter\|^2\}}$, which is equivalent to adding an extra quadratic term to \algname{BPM} to guarantee the uniqueness of $\xiterP$ when $f$ is convex. It is worth noting that~\cite[Lemma 3.1]{the_ball-proximal_point_method_new_algorithm_convergence_theory_and_applications} proves that when $f$ is a differentiable convex function, \algname{BPM} reduces to $\xiterP = \xiter-\eta_\iter\nabla f(\xiterP)/\|\nabla f(\xiterP) \|$\footnote{Differently, \algname{TRGD} reduces to $\xiterP = \xiter-\eta_\iter\nabla f(\xiter)/\|\nabla f(\xiter) \|$ when $r=0$}. 
\end{remark}

\subsection{Two variants of \algname{Stoc-TRGDA}}
\begin{algorithm2e}[t]
	\caption{\algname{Stoc-TRGDA} with Momentumn (\algname{Stoc-TRGDAM})}
	\label{alg:TR-SGDAM}
	\DontPrintSemicolon
	\SetKwInOut{Input}{Input}
	\Input{Starting point $(\x_1,\y_1)\in\dom{r}\times\dom{h}$ , step-sizes $\etax,\etay >0$ and momentumn parameter $\beta\in(0,1)$}
	\For{$\iter = 1, 2, \,\hdots, T$}{
        Sample $\xi_\iter\sim\distri_\xi$ to get the stochastic gradients $\nabla f(\xiter,\yiter,\xi_\iter)$. 
        Compute
        \begin{align}
            \twoDvec{\momx{\iter}}{\momy{\iter}} = \beta\twoDvec{\momx{\iterM}}{\momy{\iterM}} + (1-\beta)\twoDvec{\pgradX f(\xiter,\yiter,\xi_\iter)}{\pgradY f(\xiter,\yiter,\xi_\iter)}
        \end{align}
        Choose any solution of $\min_{\x} \dkh{\inner{\momx{\iter}}{\x}+r(\x)\mid \x\in\ballxiter}$ as $\xiterP$.\;
        \, Choose any solution of $\min_{\y}\dkh{\inner{-\momy{\iter}}{\y}+h(\y)\mid \y\in\ballyiter}$ as $\yiterP$.
	}
\end{algorithm2e}

\begin{algorithm2e}[t]
	\caption{Stochastic Trust-Region Gradient Descent and Ascent with Max-oracle (\strgdamax)}
	\label{alg:TR-SGDAmax}
	\DontPrintSemicolon
	\SetKwInOut{Input}{Input}
	\Input{Starting point $(\x_1,\y_1)\in\dom{r}\times\dom{h}$ , step-sizes $\etax,\etay >0$ and batch size $B\in\N_+$}
    We set $\y_0 = \y_1$\;
	\For{$\iter = 1, 2, \,\hdots, T$}{
        $\yiterk{1} = \yiterM$.\;
        \For{$k=1,2,\,\hdots,K$}{
             Sample $\widehat{\xi}_{\iter,k}\sim\distri_\xi$ to get the stochastic gradients $\pgradY f(\xiter,\yiterk{k},\widehat{\xi}_{\iter,k})$.\;
            $\yiterk{k+1} = \prox{\etay h}{\yiterk{k}+\etay\pgradY f(\xiter,\yiterk{k},\widehat{\xi}_{\iter,k})}$
        }
        ${\yiter = \yiterk{K}}$\;
        Independently sample $\xi_{\iter,1},\dots, \xi_{\iter, B} \sim\distri_\xi$ to get the stochastic gradients $\{\pgradX f(\xiter,\yiter,\xi_{\iter,i})\}_{i=1}^{B}$. \; 
        Compute $ g_{\x,\iter} =\frac{1}{B} \sum_{i=1}^{B}\pgradX f(\xiter,\yiter,\xi_{\iter,i})$.\; 
        Choose any solution of $\min_{\x} \dkh{\inner{g_{\x,\iter}}{\x}+r(\x)\mid \x\in\ball{\xiter}{\etax}}$ as $\xiterP$.
	}
\end{algorithm2e}
In this section, we introduce our first algorithm: Algorithm~\ref{alg:TR-SGDAM}, i.e., \algname{Stoc-TRGDAM}. Compared to \algname{Stoc-TRGDA}, \algname{Stoc-TRGDAM} simply replaces the stochastic gradient with Polyak momentum. Note that when considering general closed convex regularizers $r,h$. The conclusions in Lemma 3.1, 3.2 still hold, that is, $\ystar(\x)$ is a $\condnum$-smooth function \wrt $\x$ and the primal function $\pf$ is a $2\condnum L$-smooth function \wrt x. Here, we redefine $ \# \textsc{Grad}(\eps)$ as the minimum number of stochastic gradient queries required to reach $\E\zkh{(1/T)\sumIter \dist{0,\partial r(\xiter)+\nabla \pf(\xiter)}\leq \eps}$.
The following theorem demonstrates the upper-bound of gradient complexity of \strgdam. 
\begin{theorem}\label[theorem]{thm:NC-SC/TR-SGDAM/thm_Complexity_TR-SGDAM}
   Let Assumption~\ref{assum: smoothness_heavy-tailed_noise} (\textbf{$L$-smoothness + heavy-tailed noise}) and~\ref{assum: strongly_concave_for_y} (\textbf{NC-SC Setting}) hold. We let $\partial_1$ be $\mathrm{dist}(0,\nabla\pf(\x_1)+\partial r(\x_1))+5\condnum \mathrm{dist}(\grady{\x_1}{\y_1},\partial h(\y_1))$. For Algorithm~\ref{alg:TR-SGDAM}, we have following results:  
    \begin{itemize}
        \item If we konw $\Lambda_1$, $L$, $\sigma$ and $p$, we let
        \begin{align}
           \frac{\etay}{5\condnum} = \etax = \frac{\sqrt{\Lambda_1(1-\beta)/L}}{\condnum\sqrt{T}}\;\text{and}\; \beta = 1-\min\dkh{1,\max\dkh{\frac{1}{T^{\frac{p}{2p-1}}},\xkh{\frac{\Lambda_1L}{\sigma^2T}}^\frac{p}{3p-2}}},\notag
        \end{align}
        then we have
        \begin{align}
            \# \textsc{Grad}(\eps)\asymleq {\frac{\partial_1}{\eps}+\frac{\condnum^2\Lambda_1L}{\eps^2}+\frac{(\condnum\sigma)^{\frac{2p-1}{p-1}}}{\eps^\frac{2p-1}{p-1}} + \frac{\condnum^\frac{3p-2}{p-1}\Lambda_1 L\sigma^\frac{p}{p-1}}{\eps^\frac{3p-2}{p-1}}}
        \end{align}
        \item If we don't konw $\Lambda_1$, $L$, $\sigma$ and $p$, we let $\etay/5\condnum=\etax = \frac{1}{\condnum T^{0.75}}$ and $\beta = 1-\frac{1}{\sqrt{T}}$, then we have
        \begin{align}
             \# \textsc{Grad}(\eps)\asymleq    {\frac{\partial_1}{\eps} + \frac{(\condnum\Lambda_1+\condnum L)^4}{\eps^4}
    + \xkh{\frac{\condnum\sigma}{\eps}}^2 + \xkh{\frac{\condnum\sigma}{\eps}}^{\frac{2p}{p-1}} 
    }
    \end{align}
    \end{itemize}
\end{theorem}
The formal proof of~\Cref{thm:NC-SC/TR-SGDAM/thm_Complexity_TR-SGDAM} can be found in Appendix~\ref{sec:app/convergence_analysis_strgdam}. Knowing $\Lambda_1$, $L$, $\sigma$ and $p$, the upper bound of the complexity of \strgdam is $\mathcal{O}(\eps^{-(3p-2)/(p-1)})$, which matches the lower bound $\Omega(\eps^{-(3p-2)/(p-1)})$ for general first-order methods in Theorem~\ref{thm:lower_bound}. \textbf{Not only that, the complexity is also compact with respect to $\Lambda_1$, $L$, and $\sigma$.} To our knowledge, under smooth-strongly-concave setting setting, \textbf{ \strgdam is the first single-loop method with {\color{blue} optimal complexity} for stochastic {\color{blue} regularized} min-max optimization under heavy-tailed gradient noise in NC-SC setting. Specially, \strgdam does not require any techniques that rely on clipping gradients.} Since that we have proven that \sgda can also converge under heavy-tailed noise, let’s compare the two. 
\begin{itemize}
 \item\textbf{\algname{Stoc-TRGDAM}'s advantages over \sgda.} First, from the perspective of $\eps$-complexity, the advantage is obvious, since $\mathcal{O}(\eps^{-(3p-2)/(p-1)})$ is significantly smaller than $\mathcal{O}(\eps^{-2p/(p-1)})$ when $1<p<2$, and \sgda requires a mini-batch to reach $\mathcal{O}(\eps^{-2p/(p-1)})$. Secondly, we find that when the parameters $\Lambda_1$, $L$, $\sigma$ and $p$ are unknown, which is common in practice, \strgdam can reach $\mathcal{O}(\eps^{-2p/(p-1)})$, consistent with \sgda. However,~\Cref{thm:NC-SC/SGDA/thm_Complexity_SGDA} shows that \sgda needs to adjust the batch size through $p$ to achieve this complexity. 
\item \textbf{\algname{Stoc-TRGDAM}'s disadvantages over \sgda.} Although \strgdam has advantages over \sgda, it is not perfect. Regarding the dependence of the complexity on the condition number $\kappa$, $\mathcal{O}(\kappa^{-(3p-2)/(p-1)})$ of \strgdam is smaller than $\mathcal{O}(\kappa^{-3p/(2(p-1))})$ of \sgda when $p\in(1,4/3)$, whereas the opposite relationship holds for $p\in[4/3,2]$. We point out that this is because the trust region control on $\y$ prevents us from fully exploiting the strong concavity of $f(\x,\cdot)$ \wrt $\y$ in our analysis. The key challenge lies in establishing a contraction bound for the dual error $\|\yiter-\yiterstar\|^2$, analogous to that achieved in \sgda. 
\end{itemize}

Next, to improve the gradient complexity \wrt $\kappa$, we developed Algorithm~\ref{alg:TR-SGDAmax} (\strgdamax), inspired by the nested loop method \algname{SGDAmax}. For updating the dual variable $\y$, we use the stochastic proximal gradient method to return an approximate solution $\yiter$ of $\yiterstar$ through $K$ iterations like \algname{SGDAmax}. Differently, we use trust-region gradient descent for updating the primal variable $\x$. The following theorem demonstrates the upper bound of the gradient complexity of \strgdamax.
\begin{theorem}\label[theorem]{thm:NC-SC/TR-SGDAmax/thm_Complexity_TR-SGDAmax}
    Let Assumption~\ref{assum: smoothness_heavy-tailed_noise} (\textbf{$L$-smoothness + heavy-tailed noise}) and~\ref{assum: strongly_concave_for_y} (\textbf{NC-SC Setting}) hold. For Algorithm~\ref{alg:TR-SGDAmax}, we let $\etax \asymeq \sqrt{\frac{\Delta_1}{(T-1)\condnum L}}$. Moreover, given a $\eps>0$, we let $\etay \asymeq \min\dkh{{\frac{\eps^\frac{p}{p-1}}{\mu(\condnum\sigma)^\frac{p}{p-1} }}, \frac{1}{\mu\condnum^\frac{p}{2(p-1)}}}$ and $B = \max\dkh{\upround{\xkh{\frac{4\sigma}{\eps}}^\frac{p}{p-1}},1}$.
    \begin{itemize}
        \item If $\eps<2L\delta_1$, we let
        \begin{align}
            K \asymeq \ln\xkh{\frac{L\delta_1}{\eps}}\cdot\max\dkh{\frac{\sigma^\frac{p}{p-1}\condnum^\frac{p}{p-1}}{\eps^\frac{p}{p-1}}, {\condnum^\frac{p}{2(p-1)}}},\quad T \asymeq \max\dkh{\frac{\Delta_1\condnum L}{\eps^2}, \frac{L\delta_1}{\eps}},
        \end{align}
        then we have
        \begin{align}
           \#\textsc{Grad}(\eps)\asymleq{L\ln\xkh{\frac{L\delta_1}{\eps}}\cdot\max\dkh{
            \frac{\delta_1\condnum^\frac{p}{2p-2}}{\eps} , \frac{\Delta_1\condnum^\frac{3p-2}{2p-2}}{\eps^2}, \frac{\delta_1(\condnum\sigma)^\frac{p}{p-1}}{\eps^\frac{2p-1}{p-1}} , \frac{\Delta_1 \condnum^\frac{2p-1}{p-1}\sigma^\frac{p}{p-1}}{\eps^\frac{3p-2}{p-1}}
            }};
        \end{align}
        \item If $\eps\geq 2L\delta_1$, we let
        \begin{align}
             K \asymeq \max\dkh{\frac{\sigma^\frac{p}{p-1}\condnum^\frac{p}{p-1}}{\eps^\frac{p}{p-1}}, {\condnum^\frac{p}{2(p-1)}}},\quad T = \frac{\Delta_1\condnum L}{\eps^2},
        \end{align}
        then we have
        \begin{align}
           \#\textsc{Grad}(\eps)\asymleq {\max\dkh{\frac{\Delta_1L\condnum^\frac{3p-2}{2p-2}}{\eps^2},\frac{\Delta_1 L\condnum^\frac{2p-1}{p-1}\sigma^\frac{p}{p-1}}{\eps^\frac{3p-2}{p-1}}}}.
        \end{align}
    \end{itemize}
\end{theorem}
The formal proof of~\Cref{thm:NC-SC/TR-SGDAmax/thm_Complexity_TR-SGDAmax} can be found in Appendix~\ref{sec:app/convergence_analysis_strgdamax}. For $p=2$ and $\eps<2L\delta_1$, our result reduces to $\widetilde{\mathcal{O}}(\condnum^3\eps^{-4})$, which matches the light-tailed result~\cite{on_gradient_descent_ascent_for_nonconvex_concave_minmax_problems} for \algname{SGDAmax}. However, the difference is that~\cite{on_gradient_descent_ascent_for_nonconvex_concave_minmax_problems} requires a batch size of $\mathcal{O}(\condnum\sigma^2\eps^{-2})$ for updating $\xiter$ and $\mathcal{O}(k\sigma^2\eps^{-2})$ for updating $\yiter$, whereas \strgdamax only requires a batch size of $\mathcal{O}(\sigma^2\eps^{-2})$ for updating $\xiter$ and $\mathcal{O}(1)$ for updating $\yiter$. Note that $\mathcal{O}(\sigma^2\eps^{-2})$ is smaller than $\mathcal{O}(k\sigma^2\eps^{-2})$ since $\condnum\geq 1$. When $p \leq 2$, \strgdamax achieves a tighter dependence on $\condnum$ compared to \strgdam. However, this improvement comes at the cost of increased space complexity. Specifically, a mini-batch is required for updating $\xiter$, and we can only obtain an upper bound without logarithmic factors when the initial point $\y_1$ is sufficiently accurate such that $\eps \geq 2L\delta_1$.

If we replace the update of $\xiter$ in \strgdamax with the update of $\xiter$ in \strgdam. we do not need mini-batch to guarantee convergence. However, this increases the number of outer iterations $T$ to $\mathcal{O}(\eps^\frac{3p-2}{p-1})$, which is consistent with the iteration complexity of \algname{NSGD}, while the number of inner iterations $K$ remains unchanged, resulting in a worse complexity.
\begin{table}
    \centering
    \small
    \begin{tabular}{|c|c|c|c|c|c|}
        \hline 
       \textbf{ Methods} & \textbf{Settings} & \textbf{Complexity} & \textbf{Regularized} & \textbf{Single loop}& \textbf{No Clip} \\ \hline \hline
        \algname{SGDA}$_{\text{\footnotesize Clip}}$~\cite{high-probability_convergence_guarantees_of_stochastic_gradient_descent_ascent_in_structured_nonconvex_min-max_games} & \small NC-PL & $\mathcal{O}(\kappa^{\frac{3p-2}{p-1}}\eps^{-\frac{3p-2}{p-1}})$ & \no & \yes & (\no , \no) \\ \hline
        {\small \sf N$^2$SBA}~\cite{stochastic_bilevel_optimization_with_heavy-tailed_noise} &\small NC-SC &$\widetilde{\mathcal{O}}(\kappa^{\frac{2p-1}{p-1}}\eps^{-\frac{3p-2}{p-1}})$ & \no & \no & (\yes , \no)  \\ \hline
        {\small \sf Fed-NSGDA-M}~\cite{federated_stochastic_minmax_optimization_under_heavy_tailed_noises} & \small NC-PL & $\mathcal{O}(\kappa^{\frac{2p}{p-1}}\eps^{-\frac{2p}{p-1}}) $& \no & \yes & (\yes , \yes)  \\ \hline
        {\small \sf Stoc-TRGDAM} (\textbf{Ours}) & NC-SC & $\mathcal{O}(\kappa^{\frac{3p-2}{p-1}}\eps^{-\frac{3p-2}{p-1}})$ &\yes&\yes&(\yes , \yes)\\
        \hline
        {\small\sf Stoc-TRGDAmax} (\textbf{Ours}) & NC-SC & $\widetilde{\mathcal{O}}(\kappa^{\frac{2p-1}{p-1}}\eps^{-\frac{3p-2}{p-1}})$ &\yes&\no&(\yes , \yes)\\
        \hline
    \end{tabular}
    \caption{\small \textbf{Summary of our method and previous methods on stochastic minimax optimization under heavy-tailed noise in NC-SC or NC-PL setting.} The notation $(\cdot,\cdot)$ in the ``No clip'' column indicates whether the corresponding method employs gradient clipping for $\x$ and $\y$, respectively.
}\label{tab:comparasion_NCSCorPL_heavy-tails}
\end{table}
\paragraph{Comparison with related work} We compare our method and its corresponding results with several recent methods for stochastic minimax problems under heavy-tailed noise, as summarized in Table~\ref{tab:comparasion_NCSCorPL_heavy-tails}. First, and importantly, all these previous works focus on nonregularized min-max problems. Second, for the single-loop methods \algname{SGDA}$_{\text{\footnotesize Clip}}$~\cite{high-probability_convergence_guarantees_of_stochastic_gradient_descent_ascent_in_structured_nonconvex_min-max_games} and \algname{Fed-NSGDA-M}~\cite{federated_stochastic_minmax_optimization_under_heavy_tailed_noises}, although they do not require the objective function to be strongly concave with respect to $\y$ and only assume the weaker PL condition, \algname{SGDA}$_{\text{\footnotesize Clip}}$ still relies on gradient clipping for stochastic partial gradients. Moreover, the complexity $\mathcal{O}(\condnum^\frac{2p}{p-1}\eps^{-\frac{2p}{p-1}})$ of \algname{Fed-NSGDA-M} is worse than the complexity $\mathcal{O}(\kappa^{\frac{3p-2}{p-1}}\eps^{-\frac{3p-2}{p-1}})$ of \strgdamax in terms of both $\condnum$ and $\eps$. For the nested-loop method {\small \sf N$^2$SBA}~\cite{stochastic_bilevel_optimization_with_heavy-tailed_noise}, its complexity $\widetilde{\mathcal{O}}(\kappa^{\frac{2p-1}{p-1}}\eps^{-\frac{3p-2}{p-1}})$ matches that of \strgdamax. Moreover, {\small \sf N$^2$SBA} requires the same batch size $\mathcal{O}(\sigma^\frac{p}{p-1}\eps^{-\frac{p}{p-1}})$ as \strgdamax. However, {\small \sf N$^2$SBA} still relies on clipping $\pgradY f(\x,\y,\xi)$. Furthermore, we have shown that \strgdamax can achieve an upper complexity bound without the logarithmic factor when the initial point $\y_1$ is sufficiently accurate.
\subsection{Proof sketches of Theorems~\ref{thm:NC-SC/TR-SGDAM/thm_Complexity_TR-SGDAM} and~\ref{thm:NC-SC/TR-SGDAmax/thm_Complexity_TR-SGDAmax}}
\paragraph{Proof sketch of~\Cref{thm:NC-SC/TR-SGDAM/thm_Complexity_TR-SGDAM}.} To prove~\Cref{thm:NC-SC/TR-SGDAM/thm_Complexity_TR-SGDAM}, we did not actually use the two-stage analysis used in proving~\Cref{thm:NC-SC/SGDA/thm_Complexity_SGDA}, i.e., the convergence of \algname{SGDA} under heavy-tailed noise, but instead utilized a potential function that we constructed as follows:
\begin{align}
    \mathrm{V}_\iter\eqdef 2\Psi(\xiter)-F(\xiter,\yiter)+\etax\dist{0,\nabla\pf(\xiter)+\partial r(\xiter)}+\etay\dist{\grady {\xiter}{\yiter},\partial h(\yiter)}
\end{align}
With such a potential function, we can have following descent inequality holds almost surely with $\etay/\etax \asymeq \condnum$:
\begin{align}\label{eq:strgdam/descent_inequality}
    \mathrm{V}_\iterP - \mathrm{V}_\iter \asymleq
     -\etax\dist{0,\nabla\pf(\xiterP)+\partial r(\xiterP)}+ \condnum\etax\norm{\erroryiter}+\etax\norm{\errorxiter} + \condnum^2 L\etax^2 + \condnum L\etax^2,
\end{align}
where $\errorxiter \eqdef \momxiter - \pgradX f(\xiter,\yiter)$ and $\erroryiter \eqdef \momyiter -\pgradY f(\xiter,\yiter)$. Here we point out that the validity of~\eqref{eq:strgdam/descent_inequality} does not actually require the estimator to be unbiased; $\errorxiter$ and $\erroryiter$ simply defines the deviation between the true gradient and the gradient estimator (momentum in our case) we choose. Next, we sum the descent inequality and take the expectation. Then we only need to control the summations $\sumIter \E[\|\errorxiter\|]$ and $\sumIter \E[\|\erroryiter\|]$. The analysis of this part is analogous to the analysis of \algname{NSGD} in~\cite{Zijian_Liu_nsgd,Suntao_nsgd,From_Gradient_Clipping_to_Normalization}.
\paragraph{Proof skectch of~~\Cref{thm:NC-SC/TR-SGDAmax/thm_Complexity_TR-SGDAmax}.} If Assumptions~\ref{assum: smoothness_heavy-tailed_noise} and~\ref{assum: strongly_concave_for_y} hold, we can get following inequality:
    \begin{align}
           \etax\Exp{\dist{0,{\nabla\pf(\xiter)+\partial r(\xiter)}}} \asymleq & \E\zkh{ \pf(\xiterM) + r(\xiterM)}-\E\zkh{ \pf(\xiter) + r(\xiter)} \nnl 
          &+ L\etax\Exp{\norm{\yiterM-\ystar_\iterM}} + \frac{\etax\sigma}{B^{(p-1)/p}} + \condnum L\etax^2.
    \end{align}
Now, we only need to control ${\norm{\yiterM-\ystar_\iterM}}$. If Assumptions~\ref{assum: smoothness_heavy-tailed_noise} and~\ref{assum: strongly_concave_for_y} hold and $\etay = \min\dkh{\frac{2}{72^{p/(p-1)}\mu}\xkh{\frac{\eps}{\condnum\sigma }}^\frac{p}{p-1}, \frac{2^\frac{1}{p-1}}{\mu\condnum^\frac{p}{2(p-1)}}}$, then we get following results: 
\begin{itemize}
    \item If $\eps < 2L\delta_1$, $\etax\asymleq \frac{\delta_1}{\condnum}$, and $K\asymeq \ln\xkh{\frac{L\delta_1}{\eps}}\cdot\max\dkh{\frac{\sigma^\frac{p}{p-1}\condnum^\frac{p}{p-1}}{\eps^\frac{p}{p-1}}, \condnum^\frac{p}{2(p-1)}}$,  then for any $t=2,\hdots, T$, we have $\E[\|\yiter-\yiterstar\|]\asymleq \eps/L$.
    \item If $\delta_1\leq \eps/2L$, $\etax\asymleq \frac{\eps}{\condnum L}$, and $K-1\asymeq \max\dkh{
    \frac{(\sigma\condnum)^\frac{p}{p-1}}{\eps^\frac{p}{p-1}}, {\condnum^\frac{p}{2(p-1)}}
    }$, then for any $t=1,\hdots,T$, we have $\E[\|\yiter-\yiterstar\|]\asymleq \eps/L$.
\end{itemize}
The formal formulation of the above conclusion is~\Cref{lem:app/TR-SGDAmax/tech_lems/dual_error}.


\section{Conclusion and Future work}
In this work, we studied stochastic min-max optimization under heavy-tailed noise and characterized the convergence and complexity of gradient-based methods in both the nonconvex-strongly-concave (NC-SC) and nonconvex-concave (NC-C) settings. We established convergence guarantees for vanilla \algname{SGDA} without modifying stochastic gradients and showed that, in the NC-C setting, its complexity matches that of clipping-based methods. In the NC-SC setting, however, vanilla \algname{SGDA} is suboptimal, which motivates the development of new algorithms. For regularized min-max problems, our proposed methods achieve the optimal dependence on the target accuracy without introducing clipping thresholds, despite the nontrivial interaction between gradient normalization and proximal structures. Finally, the work represented in this paper does have limitations: although this paper does not primarily focus on lower-bound theory, because our proof technique follows the techniques previously used for the bounded-variance setting, the dependence on $\condnum$ in the lower bound of Theorem~\ref{thm:lower_bound} cannot match that of the upper bound, leaving us uncertain whether the resulting complexity is loose with respect to $\condnum$. Moreover, this paper indeed does not provide a single-loop method that does not require a batch size and can still achieve $\mathcal{O}(\condnum^\frac{2p-1}{p-1}\sigma^\frac{p}{p-1}\eps^{-\frac{3p-2}{p-1}})$ or sharper complexity under heavy-tailed and NC-SC setting, leaving room for the subsequent development of more efficient methods.

For the future, an interesting direction is to investigate whether the techniques developed in this work can be extended to \textit{stochastic bilevel optimization under heavy-tailed noise} since stochastic min-max optimization encompasses bilevel optimization as a special case. Another important direction is to incorporate \textit{adaptive step-size strategies} commonly used in deep learning, such as those based on moving averages of stochastic gradients, and to understand their convergence and complexity under heavy-tailed noise. Finally, relaxing the structural assumptions considered in this work, such as moving from NC-(S)C problems to more general \textit{nonconvex-nonconcave settings}, such as nonconvex-weakly-concave setting,  may provide a broader understanding of gradient-based methods for heavy-tailed stochastic optimization.
\clearpage

\bibliography{myRef}

@article{Ineq_for_r_abs_mom,
ISSN = {00034851},
URL = {http://www.jstor.org/stable/2238095},
author = {Bengt von Bahr and Carl-Gustav Esseen},
journal = {The Annals of Mathematical Statistics},
number = {1},
pages = {299--303},
publisher = {Institute of Mathematical Statistics},
title = {Inequalities for the rth Absolute Moment of a Sum of Random Variables, $1\leq r\leq 2$},
urldate = {2026-06-15},
volume = {36},
year = {1965}
}

@article{best_poss_bounds_of_vonBahr,
author = {Iosif Pinelis},
title = {{Best possible bounds of the von Bahr--Esseen type}},
volume = {6},
journal = {Annals of Functional Analysis},
number = {4},
publisher = {Tusi Mathematical Research Group},
pages = {1 -- 29},
year = {2015},
doi = {10.15352/afa/06-4-1},
URL = {https://doi.org/10.15352/afa/06-4-1}
}

@inproceedings{
From_Gradient_Clipping_to_Normalization,
title={From Gradient Clipping to Normalization for Heavy Tailed {SGD}},
author={Florian H{\"u}bler and Ilyas Fatkhullin and Niao He},
booktitle={OPT 2024: Optimization for Machine Learning},
year={2024},
url={https://openreview.net/forum?id=FvTCCfRHcY}
}

@book{rockafellar_convex_analysis,
  title={Convex analysis},
  author={Rockafellar, R Tyrrell},
  volume={28},
  year={1997},
  publisher={Princeton university press}
}

@book{bertsekas_convex_theory,
  title={Convex Optimization Theory},
  author={Bertsekas, D.},
  isbn={9781886529311},
  series={Athena Scientific optimization and computation series},
  url={https://books.google.com/books?id=0H1iQwAACAAJ},
  year={2009},
  publisher={Athena Scientific}
}

@misc{understand_grad_ortho_trust_region_opt,
      title={Understanding Gradient Orthogonalization for Deep Learning via Non-Euclidean Trust-Region Optimization}, 
      author={Dmitry Kovalev},
      year={2025},
      eprint={2503.12645},
      archivePrefix={arXiv},
      primaryClass={cs.LG},
      url={https://arxiv.org/abs/2503.12645}
}

@inproceedings{
Zijian_Liu_nsgd,
title={Nonconvex Stochastic Optimization under Heavy-Tailed Noises: Optimal Convergence without Gradient Clipping},
author={Zijian Liu and Zhengyuan Zhou},
booktitle={The Thirteenth International Conference on Learning Representations},
year={2025},
url={https://openreview.net/forum?id=NKotdPUc3L}
}

@article{Suntao_nsgd,
 title = {Revisiting Gradient Normalization and Clipping for Nonconvex SGD under Heavy-Tailed Noise: Necessity, Sufficiency, and Acceleration},
 author = {Tao Sun and Xinwang Liu and Kun Yuan},
 year = {2025},
 url = {https://jmlr.org/papers/v26/24-1991.html},
 researchr = {https://researchr.org/publication/00050025},
 cites = {0},
 citedby = {0},
 journal = {Journal of Machine Learning Research},
 volume = {26}
}

@article{Solving_SVI_without_BV_assumption,
  title={Solving Stochastic Variational Inequalities without the Bounded Variance Assumption},
  author={Alacaoglu, Ahmet and Kim, Jun-Hyun},
  journal={arXiv preprint arXiv:2602.05531},
  year={2026}
}

@misc{sgd-weakly-convex-functions-damek_davis-2018,
      title={Stochastic subgradient method converges at the rate $O(k^{-1/4})$ on weakly convex functions}, 
      author={Damek Davis and Dmitriy Drusvyatskiy},
      year={2018},
      eprint={1802.02988},
      archivePrefix={arXiv},
      primaryClass={math.OC},
      url={https://arxiv.org/abs/1802.02988}
}

@misc{Can-sgd-handle-heavy-tailed-noise_Ilyas_Fatkhullin,
      title={Can SGD Handle Heavy-Tailed Noise?}, 
      author={Ilyas Fatkhullin and Florian Hübler and Guanghui Lan},
      year={2025},
      eprint={2508.04860},
      archivePrefix={arXiv},
      primaryClass={math.OC},
      url={https://arxiv.org/abs/2508.04860}
}

@misc{in-expectation-convergence-sgd-heavy-tailed_Zijian-Liu,
      title={In-Expectation Convergence of Stochastic Gradient Methods under Heavy-Tailed Noise}, 
      author={Zijian Liu},
      year={2026},
      eprint={2606.00520},
      archivePrefix={arXiv},
      primaryClass={math.OC},
      url={https://arxiv.org/abs/2606.00520}
}

@inproceedings{sgd-weakly-convex-heavy-tailed-noises,
title={Stochastic Gradient Methods under Heavy-Tailed Noises in Weakly Convex Optimization},
author={Tianxi Zhu and Yi Xu and Qi Wang and Xiangyang Ji},
booktitle={Forty-third International Conference on Machine Learning},
year={2026},
url={https://openreview.net/forum?id=Q15yX0AZdr}
}

@misc{GAN1,
      title={Generative Adversarial Networks}, 
      author={Ian J. Goodfellow and Jean Pouget-Abadie and Mehdi Mirza and Bing Xu and David Warde-Farley and Sherjil Ozair and Aaron Courville and Yoshua Bengio},
      year={2014},
      eprint={1406.2661},
      archivePrefix={arXiv},
      primaryClass={stat.ML},
      url={https://arxiv.org/abs/1406.2661}
}

@inproceedings{Accelerated-Zeroth-order-Method-for_Kornilov,
 author = {Kornilov, Nikita and Shamir, Ohad and Lobanov, Aleksandr and Dvinskikh, Darina and Gasnikov, Alexander and Shibaev, Innokentiy and Gorbunov, Eduard and Horv\'{a}th, Samuel},
 booktitle = {Advances in Neural Information Processing Systems},
 editor = {A. Oh and T. Naumann and A. Globerson and K. Saenko and M. Hardt and S. Levine},
 pages = {64083--64102},
 publisher = {Curran Associates, Inc.},
 title = {Accelerated Zeroth-order Method for Non-Smooth Stochastic Convex Optimization Problem with Infinite Variance},
 url = {https://proceedings.neurips.cc/paper_files/paper/2023/file/ca24eb48806df3af49e5ac59d8a46f67-Paper-Conference.pdf},
 volume = {36},
 year = {2023}
}

@inproceedings{faster_single_loop_algo_mminmax_without_concavity-SAGDA,
  title={Faster single-loop algorithms for minimax optimization without strong concavity},
  author={Yang, Junchi and Orvieto, Antonio and Lucchi, Aurelien and He, Niao},
  booktitle={International conference on artificial intelligence and statistics},
  pages={5485--5517},
  year={2022},
  organization={PMLR}
}

@article{Two-timescale_gradient_descent_ascent_algorithms_for_nonconvex_minimax_optimization,
  title={Two-timescale gradient descent ascent algorithms for nonconvex minimax optimization},
  author={Lin, Tianyi and Jin, Chi and Jordan, Michael I},
  journal={Journal of Machine Learning Research},
  volume={26},
  number={11},
  pages={1--45},
  year={2025}
}

@inproceedings{delving_into_the_convergence_of_generalized_smooth_minmax_optimiation,
  title={Delving into the convergence of generalized smooth minimax optimization},
  author={Xian, Wenhan and Chen, Ziyi and Huang, Heng},
  booktitle={Forty-first International Conference on Machine Learning},
  year={2024}
}

@article{tight_analysis_of_extra-gradient_and_optimistic_gradient_methods_for_nonconvex_minimax_problems,
  title={Tight analysis of extra-gradient and optimistic gradient methods for nonconvex minimax problems},
  author={Mahdavinia, Pouria and Deng, Yuyang and Li, Haochuan and Mahdavi, Mehrdad},
  journal={Advances in Neural Information Processing Systems},
  volume={35},
  pages={31213--31225},
  year={2022}
}

@article{efficient_mirror_descent_ascent_methods_for_nonsmooth_minimax_problems,
  title={Efficient mirror descent ascent methods for nonsmooth minimax problems},
  author={Huang, Feihu and Wu, Xidong and Huang, Heng},
  journal={Advances in Neural Information Processing Systems},
  volume={34},
  pages={10431--10443},
  year={2021}
}

@article{unified-convergennce-analysis-for-adaptive-optimization-with-moving-average-estimator,
  title={Unified convergence analysis for adaptive optimization with moving average estimator},
  author={Guo, Zhishuai and Xu, Yi and Yin, Wotao and Jin, Rong and Yang, Tianbao},
  journal={Machine Learning},
  volume={114},
  number={4},
  pages={86},
  year={2025},
  publisher={Springer}
}

@article{sapd+,
  title={Sapd+: An accelerated stochastic method for nonconvex-concave minimax problems},
  author={Zhang, Xuan and Aybat, Necdet Serhat and Gurbuzbalaban, Mert},
  journal={Advances in Neural Information Processing Systems},
  volume={35},
  pages={21668--21681},
  year={2022}
}

@article{alternating_proximal-gradient_steps_for_stochastic_nonconvex_concave_minmax_problems,
  title={Alternating proximal-gradient steps for (stochastic) nonconvex-concave minimax problems},
  author={Bo{\c{t}}, Radu Ioan and B{\"o}hm, Axel},
  journal={SIAM Journal on Optimization},
  volume={33},
  number={3},
  pages={1884--1913},
  year={2023},
  publisher={SIAM}
}

@article{adaptive_algorithms_with_sharp_convergence_rates_for_stochastic_hierarchical_optimizatio,
  title={Adaptive algorithms with sharp convergence rates for stochastic hierarchical optimization},
  author={Gong, Xiaochuan and Hao, Jie and Liu, Mingrui},
  journal={Advances in Neural Information Processing Systems},
  volume={38},
  pages={57161--57210},
  year={2026}
}

@article{tiada,
  title={Tiada: A time-scale adaptive algorithm for nonconvex minimax optimization},
  author={Li, Xiang and Yang, Junchi and He, Niao},
  journal={arXiv preprint arXiv:2210.17478},
  year={2022}
}

@article{normalized_stochastic_proximal_approximation_methods_for_nonsmmoth_composite_optimization,
  title={Normalized stochastic proximal approximation methods for nonsmooth composite optimization under heavy-tailed noise},
  author={Han, Chunhao and Wang, Xiao and Xu, Pengxiang and Zhang, Jin},
  year={2026}
}

@article{the_ball-proximal_point_method_new_algorithm_convergence_theory_and_applications,
  title={The Ball-Proximal (=" Broximal") Point Method: a new algorithm, convergence theory, and applications},
  author={Gruntkowska, Kaja and Li, Hanmin and Rane, Aadi and Richt{\'a}rik, Peter},
  journal={arXiv preprint arXiv:2502.02002},
  year={2025}
}

@article{non-euclidean_broximal_point_method_a_blueprint_for_geometry_aware_optimization,
  title={Non-Euclidean Broximal Point Method: A Blueprint for Geometry-Aware Optimization},
  author={Gruntkowska, Kaja and Richt{\'a}rik, Peter},
  journal={arXiv preprint arXiv:2510.00823},
  year={2025}
}

@article{broximal_alignment_for_global_non_convex_optimization,
  title={Broximal Alignment for Global Non-Convex Optimization},
  author={Gruntkowska, Kaja and Li, Hanmin and Qian, Xun and Richt{\'a}rik, Peter},
  journal={arXiv preprint arXiv:2604.13483},
  year={2026}
}

@article{stabilized_proximal_point_method_via_trust_region_control,
  title={Stabilized Proximal Point Method via Trust Region Control},
  author={Li, Hanmin and Gruntkowska, Kaja and Richt{\'a}rik, Peter},
  journal={arXiv preprint arXiv:2604.02943},
  year={2026}
}

@article{high-probability_convergence_guarantees_of_stochastic_gradient_descent_ascent_in_structured_nonconvex_min-max_games,
  title={High Probability Convergence Guarantees of Stochastic Gradient Descent Ascent in Structured Nonconvex Min-Max Games},
  author={Ha, Junsoo},
  journal={Proceedings of Machine Learning Research vol},
  volume={336},
  pages={1--71},
  year={2026}
}

@article{stochastic_bilevel_optimization_with_heavy-tailed_noise,
  title={Stochastic bilevel optimization with heavy-tailed noise},
  author={Liu, Zhuanghua and Luo, Luo},
  journal={arXiv preprint arXiv:2509.14952},
  year={2025}
}

@article{federated_stochastic_minmax_optimization_under_heavy_tailed_noises,
  title={Federated Stochastic Minimax Optimization under Heavy-Tailed Noises},
  author={Zhang, Xinwen and Gao, Hongchang},
  journal={arXiv preprint arXiv:2511.04456},
  year={2025}
}

@inproceedings{on_gradient_descent_ascent_for_nonconvex_concave_minmax_problems,
  title={On gradient descent ascent for nonconvex-concave minimax problems},
  author={Lin, Tianyi and Jin, Chi and Jordan, Michael},
  booktitle={International conference on machine learning},
  pages={6083--6093},
  year={2020},
  organization={PMLR}
}

@article{nonconvex_decentralized_stochastic_bilevel_optimization_uner_heavy_tailed_noises,
  title={Nonconvex Decentralized Stochastic Bilevel Optimization under Heavy-Tailed Noises},
  author={Zhang, Xinwen and Zhang, Yihan and Gao, Hongchang},
  journal={arXiv preprint arXiv:2509.15543},
  year={2025}
}

@misc{high_probability_convergence_of_clip-sgd_under_heavy-tailed_noise,
      title={High Probability Convergence of Clipped-SGD Under Heavy-tailed Noise}, 
      author={Ta Duy Nguyen and Thien Hang Nguyen and Alina Ene and Huy Le Nguyen},
      year={2023},
      eprint={2302.05437},
      archivePrefix={arXiv},
      primaryClass={math.OC},
      url={https://arxiv.org/abs/2302.05437}
}

@inproceedings{why_are_adaptive_methods_good_for_attention_models,
 author = {Zhang, Jingzhao and Karimireddy, Sai Praneeth and Veit, Andreas and Kim, Seungyeon and Reddi, Sashank and Kumar, Sanjiv and Sra, Suvrit},
 booktitle = {Advances in Neural Information Processing Systems},
 editor = {H. Larochelle and M. Ranzato and R. Hadsell and M.F. Balcan and H. Lin},
 pages = {15383--15393},
 publisher = {Curran Associates, Inc.},
 title = {Why are Adaptive Methods Good for Attention Models?},
 url = {https://proceedings.neurips.cc/paper_files/paper/2020/file/b05b57f6add810d3b7490866d74c0053-Paper.pdf},
 volume = {33},
 year = {2020}
}

@misc{proximal_gradient_descent_ascent_variable_under_KL_geometry,
      title={Proximal Gradient Descent-Ascent: Variable Convergence under K{\L} Geometry}, 
      author={Ziyi Chen and Yi Zhou and Tengyu Xu and Yingbin Liang},
      year={2021},
      eprint={2102.04653},
      archivePrefix={arXiv},
      primaryClass={math.OC},
      url={https://arxiv.org/abs/2102.04653}
}

@article{smoothness_parameter_of_power_of_euclidean_norm,
   title={Smoothness Parameter of Power of Euclidean Norm},
   volume={185},
   ISSN={1573-2878},
   url={http://dx.doi.org/10.1007/s10957-020-01653-6},
   DOI={10.1007/s10957-020-01653-6},
   number={2},
   journal={Journal of Optimization Theory and Applications},
   publisher={Springer Science and Business Media LLC},
   author={Rodomanov, Anton and Nesterov, Yurii},
   year={2020},
   month=Mar, pages={303–326} 
}

@article{high_probability_convergence_for_composite_and_distributed_stochastic_minmization,
  title={High-probability convergence for composite and distributed stochastic minimization and variational inequalities with heavy-tailed noise},
  author={Gorbunov, Eduard and Sadiev, Abdurakhmon and Danilova, Marina and Horv{\'a}th, Samuel and Gidel, Gauthier and Dvurechensky, Pavel and Gasnikov, Alexander and Richt{\'a}rik, Peter},
  journal={arXiv preprint arXiv:2310.01860},
  year={2023}
}

@inproceedings{clipped_gradient_methods_for_nonsmooth_convex_optim,
  title={Clipped gradient methods for nonsmooth convex optimization under heavy-tailed noise: A refined analysis},
  author={Liu, Zijian},
  booktitle={International Conference on Learning Representations},
  volume={2026},
  pages={116560--116608},
  year={2026}
}

@article{zeroth-order_proximal_clipped_gradient_method_with_shifts_for_distributed_stochastic_composite_optimization,
  title={Zeroth-order Proximal Clipped Gradient Method with Shifts for Distributed Stochastic Composite Optimization Problems with Infinite Variance: Z.-P. Yang et al.},
  author={Yang, Zhen-Ping and Chen, Pin-Bo and Zhao, Yong and Chen, Lin},
  journal={Journal of Scientific Computing},
  volume={105},
  number={2},
  pages={36},
  year={2025},
  publisher={Springer}
}

@article{stochastic_compositional_optimization_via_hybrid_momentum_frank_wolfe,
  title={Stochastic Compositional Optimization via Hybrid Momentum Frank--Wolfe},
  author={Chayti, El Mahdi},
  journal={arXiv preprint arXiv:2605.15350},
  year={2026}
}

@article{zeroth-order_methods_for_non-smooth_stochastic_problems_under_heavy_tailed_noise,
  title={Zeroth-order methods for non-smooth stochastic problems under heavy-tailed noise},
  author={Bashirov, Nail and Gasnikov, Alexander and Lobanov, Aleksandr},
  journal={Optimization Methods and Software},
  pages={1--26},
  year={2026},
  publisher={Taylor \& Francis}
}

@inproceedings{linear_attention_is_all_you_need,
 author = {Ahn, Kwangjun and Cheng, Xiang and Song, Minhak and Yun, Chulhee and Jadbabaie, Ali and Sra, Suvrit},
 booktitle = {International Conference on Learning Representations},
 editor = {B. Kim and Y. Yue and S. Chaudhuri and K. Fragkiadaki and M. Khan and Y. Sun},
 pages = {16193--16205},
 title = {Linear attention is (maybe) all you need (to understand Transformer optimization)},
 volume = {2024},
 year = {2024}
}

@InProceedings{revisiting_the_noise_medel_of_stochastic_gradient_descent,
  title = 	 {Revisiting the Noise Model of Stochastic Gradient Descent},
  author =       {Battash, Barak and Wolf, Lior and Lindenbaum, Ofir},
  booktitle = 	 {Proceedings of The 27th International Conference on Artificial Intelligence and Statistics},
  pages = 	 {4780--4788},
  year = 	 {2024},
  editor = 	 {Dasgupta, Sanjoy and Mandt, Stephan and Li, Yingzhen},
  volume = 	 {238},
  series = 	 {Proceedings of Machine Learning Research},
  month = 	 {02--04 May},
  publisher =    {PMLR}
}

@InProceedings{on_proximal_policy_optimization_heavy-tailed-gradients,
  title = 	 {On Proximal Policy Optimization’s Heavy-tailed Gradients},
  author =  {Garg, Saurabh and Zhanson, Joshua and Parisotto, Emilio and Prasad, Adarsh and Kolter, Zico and Lipton, Zachary and Balakrishnan, Sivaraman and Salakhutdinov, Ruslan and Ravikumar, Pradeep},
  booktitle = 	 {Proceedings of the 38th International Conference on Machine Learning},
  pages = 	 {3610--3619},
  year = 	 {2021},
  editor = 	 {Meila, Marina and Zhang, Tong},
  volume = 	 {139},
  series = 	 {Proceedings of Machine Learning Research},
  month = 	 {18--24 Jul},
  publisher =    {PMLR}
 }

@inproceedings{
MARL1,
title={Self-Play Preference Optimization for Language Model Alignment},
author={Yue Wu and Zhiqing Sun and Huizhuo Yuan and Kaixuan Ji and Yiming Yang and Quanquan Gu},
booktitle={The Thirteenth International Conference on Learning Representations},
year={2025},
url={https://openreview.net/forum?id=a3PmRgAB5T}
}

@inproceedings{MARL2,
author = {Li, Shihui and Wu, Yi and Cui, Xinyue and Dong, Honghua and Fang, Fei and Russell, Stuart},
title = {Robust multi-agent reinforcement learning via minimax deep deterministic policy gradient},
year = {2019},
isbn = {978-1-57735-809-1},
publisher = {AAAI Press},
url = {https://doi.org/10.1609/aaai.v33i01.33014213},
doi = {10.1609/aaai.v33i01.33014213},
articleno = {517},
numpages = {8},
location = {Honolulu, Hawaii, USA},
series = {AAAI'19/IAAI'19/EAAI'19}
}

@inproceedings{
AUC1,
title={Stochastic AUC Maximization with Deep Neural Networks},
author={Mingrui Liu and Zhuoning Yuan and Yiming Ying and Tianbao Yang},
booktitle={International Conference on Learning Representations},
year={2020},
url={https://openreview.net/forum?id=HJepXaVYDr}
}

@misc{AUC2,
      title={Large-scale Robust Deep AUC Maximization: A New Surrogate Loss and Empirical Studies on Medical Image Classification}, 
      author={Zhuoning Yuan and Yan Yan and Milan Sonka and Tianbao Yang},
      year={2021},
      eprint={2012.03173},
      archivePrefix={arXiv},
      primaryClass={cs.LG},
      url={https://arxiv.org/abs/2012.03173}
}

@InProceedings{high_probability_bounds_for_stochastic_optimization_and_variational_inequalities,
  title = 	 {High-Probability Bounds for Stochastic Optimization and Variational Inequalities: the Case of Unbounded Variance},
  author =       {Sadiev, Abdurakhmon and Danilova, Marina and Gorbunov, Eduard and Horv\'{a}th, Samuel and Gidel, Gauthier and Dvurechensky, Pavel and Gasnikov, Alexander and Richt\'{a}rik, Peter},
  booktitle = 	 {Proceedings of the 40th International Conference on Machine Learning},
  pages = 	 {29563--29648},
  year = 	 {2023},
  editor = 	 {Krause, Andreas and Brunskill, Emma and Cho, Kyunghyun and Engelhardt, Barbara and Sabato, Sivan and Scarlett, Jonathan},
  volume = 	 {202},
  series = 	 {Proceedings of Machine Learning Research},
  month = 	 {23--29 Jul},
  publisher =    {PMLR},
  url = 	 {https://proceedings.mlr.press/v202/sadiev23a.html}
}

@misc{mirror_descent_strikes_again_optimal_stochastic_convex_optimization_under_infinite_noise_variance,
      title={Mirror Descent Strikes Again: Optimal Stochastic Convex Optimization under Infinite Noise Variance}, 
      author={Nuri Mert Vural and Lu Yu and Krishnakumar Balasubramanian and Stanislav Volgushev and Murat A. Erdogdu},
      year={2022},
      eprint={2202.11632},
      archivePrefix={arXiv},
      primaryClass={stat.ML},
      url={https://arxiv.org/abs/2202.11632}
}

@misc{revisiting_the_last_iterate_convergence_of_stochastic_gradient_methods,
      title={Revisiting the Last-Iterate Convergence of Stochastic Gradient Methods}, 
      author={Zijian Liu and Zhengyuan Zhou},
      year={2026},
      eprint={2312.08531},
      archivePrefix={arXiv},
      primaryClass={cs.LG},
      url={https://arxiv.org/abs/2312.08531}
}

@misc{online_convex_optimization_with_heavy_tails,
      title={Online Convex Optimization with Heavy Tails: Old Algorithms, New Regrets, and Applications}, 
      author={Zijian Liu},
      year={2026},
      eprint={2508.07473},
      archivePrefix={arXiv},
      primaryClass={cs.LG},
      url={https://arxiv.org/abs/2508.07473}
}

@article{vanilla_sgd_with_momentum_survives_heavy-tailed_noise,
  title={Vanilla SGD with Momentum Survives Heavy-Tailed Noise: Convergence Analysis without Gradient Clipping or Normalization},
  author={Yamada, Ryusei and Sato, Naoki and Iiduka, Hideaki},
  journal={arXiv preprint arXiv:2607.08104},
  year={2026}
}

@misc{O2NC,
      title={Optimal Stochastic Non-smooth Non-convex Optimization through Online-to-Non-convex Conversion}, 
      author={Ashok Cutkosky and Harsh Mehta and Francesco Orabona},
      year={2025},
      eprint={2302.03775},
      archivePrefix={arXiv},
      primaryClass={cs.LG},
      url={https://arxiv.org/abs/2302.03775}
}

@article{accelerated_stochastic_first-order_method_for_convex_optimizatio_under_heavy-tailed_noise,
  title={Accelerated stochastic first-order method for convex optimization under heavy-tailed noise},
  author={He, Chuan and Lu, Zhaosong},
  journal={arXiv preprint arXiv:2510.11676},
  year={2025}
}

@article{zeroth-order_methods_for_non-smooth_stochastic_problems_under_heavy-tailed_noise,
  title={Zeroth-order methods for non-smooth stochastic problems under heavy-tailed noise},
  author={Bashirov, Nail and Gasnikov, Alexander and Lobanov, Aleksandr},
  journal={Optimization Methods and Software},
  pages={1--26},
  year={2026},
  publisher={Taylor \& Francis}
}

@inproceedings{avoid_overclaims_summary_of_complexity_bounds_minimax_optimization,
  author = {Zhang, Siqi and Hu, Yifan},
  title = {Avoid Overclaims - Summary of Complexity Bounds for Algorithms in Minimization and Minimax Optimization},
  booktitle = {ICLR Blogposts 2025},
  year = {2025},
  date = {April 28, 2025},
  note = {https://iclr-blogposts.github.io/2025/blog/opt-summary/},
  url  = {https://iclr-blogposts.github.io/2025/blog/opt-summary/}
}

@misc{complexity_lower_bounds_for_nonconvex_strongly_concave_min_max_optimization,
      title={Complexity Lower Bounds for Nonconvex-Strongly-Concave Min-Max Optimization}, 
      author={Haochuan Li and Yi Tian and Jingzhao Zhang and Ali Jadbabaie},
      year={2021},
      eprint={2104.08708},
      archivePrefix={arXiv},
      primaryClass={math.OC},
      url={https://arxiv.org/abs/2104.08708}
}

@misc{lower_bounds_for_non-convex_stochastic_optimization,
      title={Lower Bounds for Non-Convex Stochastic Optimization}, 
      author={Yossi Arjevani and Yair Carmon and John C. Duchi and Dylan J. Foster and Nathan Srebro and Blake Woodworth},
      year={2022},
      eprint={1912.02365},
      archivePrefix={arXiv},
      primaryClass={math.OC},
      url={https://arxiv.org/abs/1912.02365}
}
\clearpage

\begin{spacing}{0.5}
\setcounter{tocdepth}{2}
\tableofcontents    
\end{spacing}

\appendix
\section{Missing proofs in section~\ref{sec:sgda_under_heavy-tails}}
\subsection{Proof of~\Cref{thm:NC-SC/SGDA/thm_Complexity_SGDA} }\label{sec:app-proof-of-sgda}
First, we prove the following key lemmas

\begin{lemma}\label[lemma]{lem:app/Analysis_SGDA_v2/primal_error_SGDA}
     Let Assumptions~\ref{assum: smoothness_heavy-tailed_noise} and~\ref{assum: strongly_concave_for_y} hold. For \sgda with {\color{blue} Case \romanOne}, if $\etax\leq 1/(8\condnum L)$, we have 
     \begin{align}\label{eq:app/Analysis_SGDA_v2/dual_error_SGDA}
    \Exp{\xkh{\sumIter {\etax\norm{\nabla\pf(\xiter)}^2 }}^\frac{p}{2}} \leq& 4\Exp{ \xkh{\sumIter  L^2\etax \norm{\yiter-\ystar_\iter}^2}^{\nicefrac{p}{2}}}\nnl
    &+3\Delta_1^{\nicefrac{p}{2}} + 8(\condnum L)^{\nicefrac{p}{2}} \etax^p T \var^p + 6\var^pT^{1-\frac{p}{2}} \etax^\frac{p}{2}
     \end{align}
      where $\Delta_1 \eqdef \pf(\x_1)-\min_\x\pf(\x)$. 
\end{lemma}
\begin{proofof}
According to Lemma~\ref{lem:prelims/smoothness_primal_function}, we know that the dual function $\pf(\x)\eqdef \max_\y f(\x,\y)$ is a $2\condnum L$-smooth function. Therefore, the following inequality holds.
\begin{align}
    &\pf(\xiterP) - \pf(\xiter) \leq \inner{\nabla \pf(\xiter)}{\xiterM-\xiter} + \condnum \norm{\xiterP-\xiter}^2 = -\etax\inner{\nabla \pf(\xiter)}{\gxiter} + \condnum L\etax^2 \norm{\gxiter}^2 &\nnl 
    &= -\etax\inner{\nabla \pf(\xiter)}{\xerror{\iter}} -\etax\inner{\nabla \pf(\xiter)}{\pgradX f(\xiter,\yiter)-\nabla\pf(\xiter)} - \etax\norm{\nabla \pf(\xiter)}^2 &\nnl 
    &+ \condnum L\etax^2 \norm{\pgradX f(\xiter,\yiter)}^2 + 2\condnum L\etax^2\inner{\pgradX f(\xiter,\yiter)}{\xerror{\iter}} + \condnum L\etax^2 \norm{\xerror{\iter}}^2 &\nnl
    &= (-\etax+2\condnum L\etax^2) \norm{\nabla \pf(\xiter)}^2 + (-\etax+2\condnum L\etax^2)\inner{\nabla\pf(\xiter)}{\xerror{\iter}} + \condnum L\etax^2 \norm{\xerror{\iter}}^2 +  \condnum L\etax^2 \norm{\pgradX f(\xiter,\yiter)-\nabla\pf(\xiter)}^2&\nnl
    &+ (-\etax+2\condnum L\etax^2)\inner{\nabla\pf(\xiter)}{\pgradX f(\xiter,\yiter)-\nabla \pf(\xiter)}+  2\condnum L\etax^2\inner{\nabla \pf(\xiter)}{\pgradX f(\xiter,\yiter)-\nabla\pf(\xiter)}
\end{align}
   Applying Cauchy-Schwarz inequality, the $L$-Lipschitz continuity of $\pgradX f(\xiter,\yiter)$ and Young's inequality , we can obtain
\begin{align}
& {\condnum L \etax^2 \norm{\pgradX f(\xiter,\yiter)-\nabla\pf(\xiter)}^2+} & \nnl
& (-\etax+2\condnum L\etax^2)\inner{\nabla\pf(\xiter)}{\pgradX f(\xiter,\yiter)-\nabla \pf(\xiter)}+  2\condnum L\etax^2\inner{\nabla \pf(\xiter)}{\pgradX f(\xiter,\yiter)-\nabla\pf(\xiter)} &\nnl
    &\leq L\etax\abs{2\condnum L\etax-1}\norm{\nabla\pf(\xiter)}\norm{\yiter-\ystar_\iter} + 2\condnum\etax^2 L^2\norm{\yiter-\ystar_\iter}\norm{\xerror{\iter}} + \condnum L^3\etax^2\norm{\yiter-\ystar_\iter}^2& \nnl 
    &\overset{(a)}{\leq} {\frac{\etax\norm{\nabla\pf(\xiter)}^2}{4} + \etax L^2\norm{\yiter-\ystar_\iter}^2} + 2\condnum\etax^2 L^2\xkh{\frac{L\norm{\yiter-\ystar_\iter}^2}{4}+\frac{\norm{\xerror{\iter}}^2}{L}} + \condnum L^3\etax^2\norm{\yiter-\ystar_\iter}^2&\nnl 
    &= \frac{\etax}{4} \norm{\nabla\pf(\xiter)}^2 + \xkh{\frac{3\condnum L^3\etax^2}{2} +L^2\etax} \norm{\yiter-\ystar_\iter}^2 + 2\condnum L\etax^2 \norm{\xerror{\iter}}^2 \nnl 
    &\overset{(b)}{\leq} \frac{\etax}{4} \norm{\nabla\pf(\xiter)}^2 + \frac{19L^2\etax\norm{\yiter-\ystar_\iter}^2}{16}+ 2\condnum L\etax^2 \norm{\xerror{\iter}}^2,
\end{align}
where $(a)$ holds beacause $2\condnum L\etax \leq 1$, and $(b)$ holds because $3\condnum L^3\etax^2/2 = 3L^2\etax/2\cdot \condnum L\etax\leq 3L^2\etax/16$. Then, we can obtain
\begin{align}
     \pf(\xiterP) - \pf(\xiter) \leq& \etax\xkh{2\condnum L\etax - 3/4} \norm{\nabla\pf(\xiter)}^2  + 2\etax(2\condnum L\etax-1)\inner{\nabla\pf(\xiter)}{\xerror{\iter}} + 3\condnum L\etax^2 \norm{\xerror{\iter}}^2\nnl 
     &+ 3\condnum L\etax^2 \norm{\xerror{\iter}}^2+ \frac{19L^2\etax\norm{\yiter-\ystar_\iter}^2}{16}\nnl 
     \stackAlign{(c)}{\leq} -\frac{\etax \norm{\nabla\pf(\xiter)}^2}{2} + \etax(2\condnum L\etax-1)\inner{\nabla\pf(\xiter)}{\xerror{\iter}} + 3\condnum L\etax^2 \norm{\xerror{\iter}}^2+ \frac{19L^2\etax\norm{\yiter-\ystar_\iter}^2}{16},
\end{align}
where $(c)$ holds due to the condition $2\condnum L\etax\leq1/4$. Multiplying both sides of the above inequality by $2$ and summing the result from $t=1$ to $T$, we get
\begin{align}
    \sumIter \etax\norm{\nabla\pf(\xiter)}^2 \leq& \abs{\sumIter 2\etax(2\condnum L\etax-1)\inner{\nabla\pf(\xiter)} {\xerror{\iter}}} \nnl 
    &+2\Delta_1 + + \sumIter 6 \condnum L\etax^2 \norm{\xerror{\iter}}^2 +\sumIter 3L^2\etax \norm{\yiter-\ystar_\iter}^2,
\end{align}
where $\Delta_1\eqdef \pf(\x_i)-\min_\x\pf(\x)$. 
Taking both sides of the above inequality to the power of $p/2$ and then taking the expectation to get
\begin{align}\label{eq:SECS/appdenix/Analysis_SGDA_v2/UB1}
    \Exp{\xkh{\sumIter \etax\norm{\nabla\pf(\xiter)}^2}^{\nicefrac{p}{2}}} \leq& \Exp{\xkh{\abs{\sumIter 2\etax(2\condnum L\etax-1)\inner{\nabla\pf(\xiter)} {\xerror{\iter}}}}^{\nicefrac{p}{2}}} \nnl 
    +& 2\Delta_1^{\nicefrac{p}{2}} + 6(\condnum L)^{\nicefrac{p}{2}} \etax^p \sumIter \Exp{\norm{\xerror{\iter}}^p} + 3\Exp{ \xkh{\sumIter L^2\etax \norm{\yiter-\ystar_\iter}^2}^{\nicefrac{p}{2}}}\nnl 
    \stackAlign{(d)}{\leq} \Exp{\xkh{\abs{\sumIter 2\etax(2\condnum L\etax-1)\inner{\nabla\pf(\xiter)} {\xerror{\iter}}}}^{\nicefrac{p}{2}}} \nnl 
    +& 2\Delta_1^{\nicefrac{p}{2}} + 6(\condnum L)^{\nicefrac{p}{2}} \etax^p T \var^p + 3\Exp{ \xkh{\sumIter L^2\etax \norm{\yiter-\ystar_\iter}^2}^{\nicefrac{p}{2}}},
\end{align}
where $(d)$ holds due to $\E[\norm{\xerror{\iter}}^p] = \E[\E_\iter[\norm{\xerror{\iter}}^p]]\leq\var^p$. Furthermore, we have
\begin{align}\label{eq:SECS/appdenix/Analysis_SGDA_v2/UB2}
    & \Exp{\xkh{\abs{\sumIter 2\etax(2\condnum L\etax-1)\inner{\nabla\pf(\xiter)} {\xerror{\iter}}}}^{\nicefrac{p}{2}}}& \nnl &\overset{(c)}{\leq} \sqrt{\Exp{\xkh{\abs{\sumIter 2\etax(2\condnum L\etax-1)\inner{\nabla\pf(\xiter)} {\xerror{\iter}}}}^p}}= \sqrt{2^p\etax^p \abs{2\condnum L\etax-1}^p  \Exp{\xkh{\abs{\sumIter \inner{\nabla\pf(\xiter)} {\xerror{\iter}}}}^p}}&\nnl 
    &\overset{(d)}{\leq}   \sqrt{2^p\etax^p \Exp{\xkh{\abs{\sumIter \inner{\nabla\pf(\xiter)} {\xerror{\iter}}}}^p}} \overset{\text{\tiny \Cref{lem:app/Supp_lems/p_moment_sum_martingales}}}{\leq} \sqrt{4\etax^p {\sumIter \Exp{\abs{\inner{\nabla\pf(\xiter)}{\xerror{\iter}}}^p} }}& \nnl 
    &\overset{(e)}{\leq} \sqrt{4\etax^{\nicefrac{p}{2}} \var^p\Exp{\sumIter {\etax^{\nicefrac{p}{2}}\norm{\nabla\pf(\xiter)}^p }}}\overset{(f)}{\leq} \sqrt{4\var^p\xkh{\sumIter \etax^\frac{p}{2-p}}^\frac{2-p}{2} \Exp{\xkh{\sumIter {\etax\norm{\nabla\pf(\xiter)}^2 }}^\frac{p}{2}}}&\nnl 
    &\overset{(g)}{\leq} \frac{1}{4} \Exp{\xkh{\sumIter {\etax\norm{\nabla\pf(\xiter)}^2 }}^\frac{p}{2}} + 4\var^pT^\frac{2-p}{2} \etax^\frac{p}{2},
\end{align}
where $(c)$ holds due to the Jensen's inequality, i.e., $q_1=q_2=2$ and $Y\eqdef\bm{1}$ in~\eqref{eq: app/Supp_lems/Holder_Ineq_exp}; $(d)$ holds due to the condition $|2\condnum L\etax-1|\leq1$; $(e)$ holds true because we used Cauchy-Schwarz's inequality and the law of total expectation, i.e., $\Exp{\|\nabla \pf(\xiter)\|\|\xerror{\iter}\|} = \Exp{\|\nabla \pf(\xiter)\|\Epiter{\|\xerror{\iter}\|}} \leq \var^p\Exp{\|\nabla \pf(\xiter)\|}$; $(f)$ holds because we used Hölder's inequality (\Cref{lem:app/Supp_lems/Holder_Ineq}), i.e., let $a_t$, $b_t$, $q_1$ and $q_2$ in~\eqref{eq: app/Supp_lems/Holder_Ineq_sum} be $\|\nabla \pf(\xiter)\|^p$, $\etax^{\nicefrac{p}{2}}$, $2/p$ and $2/(2-p)$. Note that we can apply~\eqref{eq:app/Supp_lems/p_moment_sum_martingales} in Lemma~\ref{lem:app/Supp_lems/p_moment_sum_martingales} because $\{\inner{\nabla\pf(\xiter)} {\xerror{\iter}}\}_{t=1}^T$ is indeed a martingale difference sequence related to filteration $\mathcal{F}_\iterM\eqdef \sigma(\x_1,\y_1,\dots,\xiter,\yiter)$.

Substituting~\eqref{eq:SECS/appdenix/Analysis_SGDA_v2/UB2} into~\eqref{eq:SECS/appdenix/Analysis_SGDA_v2/UB1}, we get
\begin{align}
    \Exp{\xkh{\sumIter {\etax\norm{\nabla\pf(\xiter)}^2 }}^\frac{p}{2}} \leq& (8/3)\Delta_1^{\nicefrac{p}{2}} + 8(\condnum L)^{\nicefrac{p}{2}} \etax^p T \var^p + (16/3)\var^pT^\frac{2-p}{p} \etax^\frac{p}{2}\nnl 
    &+ 4\Exp{ \xkh{\sumIter  L^2\etax \norm{\yiter-\ystar_\iter}^2}^{\nicefrac{p}{2}}}
\end{align}
By appropriately increasing the constants, we proved the lemma.
\end{proofof}

\begin{lemma}\label[lemma]{lem:app/Analysis_SGDA_v2/dual_error(One_Step)_SGDA}
    Suppose that Assumptions~\ref{assum: smoothness_heavy-tailed_noise}\ref{assum:smoothness} and~\ref{assum: strongly_concave_for_y} hold. For \sgda with {\color{blue} Case \romanOne}, if $\etax\leq\min\{1/(2\condnum L),\etay/(12\condnum^2)\}$ and $\etay\leq1/(2L)$, then after one iteration of \sgda, we have
    \begin{align}\label{eq:app/Analysis_SGDA_v2/dual_error(One_Step)_SGDA}
         L^2\etax\norm{\yiter-\ystar_\iter}^2 &\leq \xkh{1-\frac{\mu\etay}{2}}L^2\etax\norm{\yiterM-\ystar_\iterM}^2+ {2\condnum L\etax^2\norm{\nabla \pf(\xiterM)}^2} \nnl &\quad+\frac{3\condnum L\etax^2\norm{\xerror{\iterM}}^2}{2}+ 4L^2\etax\etay^2 \norm{\yerror{\iterM}}^2\nnl 
        &\quad + 2L^2\etax\etay(1+2\condnum L\etax)\inner{\yerror{\iterM}}{\yiterM-\ystar_\iterM}\nnl 
        &\quad + \xkh{2(\condnum L\etax)^2\etax+{\condnum L\etax^2}}\inner{\xerror{\iterM}}{\nabla\pf(\xiterM)}
    \end{align}
\end{lemma}
\begin{proofof}
First, we have
\begin{align}\label{eq:app/SGDA/Proof_Lemma2/UB1}
    & \|\yiter-\ystar_\iterM\|^2 = \norm{\yiter-\yiterM +\yiterM-\ystar_\iterM}^2 &\nnl
    & = \norm{\yiter-\yiterM}^2 + 2\inner{\yiter-\yiterM}{\yiterM-\ystar_\iterM} + \norm{\yiterM-\ystar_\iterM}^2 &\nnl 
    & = \norm{\yiterM-\ystar_\iterM}^2 + 2\etay\inner{\gyiterM}{\yiterM-\ystar_\iterM} + \etay^2 \norm{\gyiterM}^2 &\nnl 
    & \leq \norm{\yiterM-\ystar_\iterM}^2 + 2\etay\inner{\pgradY \fiterM}{\yiterM-\ystar_\iterM}+ 2\etay^2\norm{\pgradY\fiterM}^2 &\nnl 
    &\quad + 2\etay\inner{\yerror{\iterM}}{\yiterM-\ystar_\iterM} + 2\etay^2\norm{\yerror{\iterM}}^2 &\nnl 
    & \overset{(a)}{\leq} (1-\mu\etay)\norm{\yiterM-\ystar_\iterM}^2 + 2\etay\inner{\yerror{\iterM}}{\yiterM-\ystar_\iterM} + 2\etay^2\norm{\yerror{\iterM}}^2 &\nnl 
    &\quad+ 2\etay(2L\etay-1)\xkh{f(\xiterM,\ystar_\iterM)-\fiterM} & \nnl 
    &\overset{(b)}{\leq} (1-\mu\etay)\norm{\yiterM-\ystar_\iterM}^2 + 2\etay\inner{\yerror{\iterM}}{\yiterM-\ystar_\iterM} + 2\etay^2\norm{\yerror{\iterM}}^2,
\end{align}
where $(a)$ holds due to the $\mu$-strognly convavity of $f(\x,\cdot)$ with respect to $\y$, i.e., $f(\xiterM,\ystar_\iterM)-\fiterM \leq \inner{\pgradY \fiterM}{\ystar_\iterM-\yiterM} - (\mu/2)\|\yiter-\ystar_\iter\|^2$, and Lemma~\ref{lem:app/Supp_lems/smoothFunction_NormGrad_upperBound_functionValue}; $(b)$ holds due to condition $\etay\leq 1/(2L)$. 

Second, for $\|\ystar_\iter-\ystar_\iterM\|^2$, we have
\begin{align}\label{eq:app/SGDA/Proof_Lemma2/UB2}
   & \|\ystar_\iter-\ystar_\iterM\|^2 \overset{\text{\tiny Lemma~\ref{lem:prelims/smoothness_dual_optimizers}}}{\leq} \condnum^2\norm{\xiter-\xiterM}^2 = \condnum^2\etax^2\norm{\gxiterM}^2 &\nnl 
    & = \condnum^2\etax^2\norm{\pgradX \fiterM}^2 +2\condnum^2\etax^2\inner{\xerror{\iterM}}{\pgradX \fiterM} + \condnum^2\etax^2\norm{\xerror{\iterM}}^2 & \nnl 
    & \leq  2\condnum^2\etax^2\norm{\nabla\pf(\xiterM)}^2 +2\condnum^2\etax^2\norm{\pgradX \fiterM - \nabla\pf(\xiterM)}^2 + \condnum^2\etax^2\norm{\xerror{\iterM}}^2 &\nnl 
    &\quad + 2\condnum^2\etax^2\inner{\xerror{\iterM}}{\nabla \pf(\xiterM)}+2\condnum^2\etax^2\inner{\xerror{\iterM}}{\pgradX \fiterM-\nabla\pf(\xiterM)}&\nnl 
    &\overset{(a)}{\leq} 4(\condnum L\etax)^2 \norm{\yiterM-\ystar_\iterM}^2 + 2(\condnum\etax)^2\xkh{\norm{\nabla \pf(\xiterM)}^2 + \inner{\xerror{\iterM}}{\nabla \pf(\xiterM)}} + \frac{3(\condnum \etax)^2\norm{\xerror{\iterM}}^2}{2},
\end{align}
where $(a)$ holds because
\begin{align}
    \inner{\xerror{\iterM}}{\pgradX \fiterM-\nabla\pf(\xiterM)} \leq& \norm{\xerror{\iterM}}\norm{\pgradX \fiterM-\nabla\pf(\xiterM)}\nnl 
    \leq& L\norm{\xerror{\iterM}}\norm{\yiterM-\ystar_\iterM} \nnl 
    \leq& L^2\norm{\yiterM-\ystar_\iterM}^2 + \frac{\norm{\xerror{\iterM}}^2}{4}.
\end{align}

 Then, by the Young's inequality, we can obtain 
 \begin{align}\label{eq:app/SGDA/Proof_Lemma2/UB3}
     &\norm{\yiter-\ystar_\iter}^2 \leq (1+2\condnum L\etax)\norm{\yiter-\ystar_\iterM}^2 + \xkh{1+\frac{1}{2\condnum L\etax}}\norm{\ystar_\iter-\ystar_\iterM}^2 &\nnl 
     &\overset{(\ref{eq:app/SGDA/Proof_Lemma2/UB1},\ref{eq:app/SGDA/Proof_Lemma2/UB2})}{\leq} \zkh{(1-\mu\etay)(1+2\condnum L\etax)+4(\condnum L\etax)^2+2\condnum L\etax}\norm{\yiterM-\ystar_\iterM}^2 + 2\etay^2(1+2\condnum L\etax)\norm{\yerror{\iterM}}^2 & \nnl 
     &\quad\quad\quad+ \xkh{1+\frac{1}{2\condnum L\etax}}\xkh{2(\condnum\etax)^2\norm{\nabla \pf(\xiterM)}^2 + \frac{3(\condnum\etax)^2\norm{\xerror{\iterM}}^2}{2}}&\nnl 
     &\quad\quad\quad+ 2\etay(1+2\condnum L\etax)\inner{\yerror{\iterM}}{\yiterM-\ystar_\iterM} + \xkh{2(\condnum\etax)^2+\frac{\condnum\etax}{L}}\inner{\xerror{\iterM}}{\nabla\pf(\xiterM)}&\nnl 
     &\quad \overset{(b)}{\leq} \xkh{1-\frac{\mu\etay}{2}}\norm{\yiterM-\ystar_\iterM}^2+ \frac{2\condnum \etax\norm{\nabla \pf(\xiterM)}^2}{L} +\frac{3\condnum\etax\norm{\xerror{\iterM}}^2}{2L}+ 4\etay^2\norm{\yerror{\iterM}}^2& \nnl 
     &\quad \quad\;+ 2\etay(1+2\condnum L\etax)\inner{\yerror{\iterM}}{\yiterM-\ystar_\iterM} + \xkh{2(\condnum\etax)^2+\frac{\condnum\etax}{L}}\inner{\xerror{\iterM}}{\nabla\pf(\xiterM)},
 \end{align}
 where $(b)$ holds because
 \begin{align}
     (1-\mu\etay)(1+2\condnum L\etax)+4(\condnum L\etax)^2+2\condnum L\etax\leq& 1-\mu\etay+4\condnum L\etax + (2\condnum L\etax)^2\nnl 
     \leq& 1-\mu\etay+6\condnum L\etax \tag{$\etax\leq\frac{1}{2\condnum L}$} \\ 
     \leq& 1-\frac{\mu\etay}{2}, \tag{$\etax\leq\frac{\etay}{12\condnum^2}$} 
 \end{align}
 and 
 \begin{align}
     \begin{cases}
         2(\condnum\etax)^2+\frac{\condnum\etax}{L} \leq \frac{2\condnum\etax}{L},\\ 
         \frac{3(\condnum \etax)^2}{2} + \frac{3\condnum\etax}{4L}\leq \frac{3\condnum \etax}{2L},\\ 
         2\etay^2(1+2\condnum L\etax)\leq 4\etay^2.
     \end{cases}\tag{$\etax\leq\frac{1}{2\condnum L}$}
 \end{align}
 We now finish the proof by multiplying both sides of~\eqref{eq:app/SGDA/Proof_Lemma2/UB3} by $L^2\etax$. 
 \end{proofof}

\begin{lemma}\label[lemma]{lem:app/Analysis_SGDA_v2/Exp_dual_error_SGDA}
 Let Assumptions~\ref{assum: smoothness_heavy-tailed_noise} and~\ref{assum: strongly_concave_for_y} hold. For \sgda with {\color{blue} Case \romanOne}, if $\etax\leq \min\dkh{1/(2\condnum L),\etay/2(17^{\nicefrac{2}{p}}\condnum^2)}$ and $\etay\leq 1/(2L)$, we have
    \begin{align}\label{eq:app/Analysis_SGDA_v2/Exp_dual_error_SGDA}
        4\Exp{ \xkh{\sumIter \condnum L^2\etax \norm{\yiter-\ystar_\iter}^2}^{\nicefrac{p}{2}}} \leq& \frac{1}{2} \Exp{\xkh{\sumIter \etax\norm{\nabla \pf(\xiter)}^2}^{\nicefrac{p}{2}}}\nnl 
        &+9\xkh{\frac{\delta_1^2}{2\etay}}^{\nicefrac{p}{2}} + 18 (\condnum\etax)^{\nicefrac{p}{2}}T\var^p + 2405 T^{1-\nicefrac{p}{2}}(\condnum\var)^p\etax^{\nicefrac{p}{2}},
    \end{align}
    where $\delta_1\eqdef \norm{\y_1-\ystar_1}$.
\end{lemma}
\begin{proofof}
To simplify the calculation, we first introduce the following notation
\begin{align}\label{eq:app/Analysis_SGDA_v2/Exp_dual_error_SGDA/notaions}
& \beta\eqdef 1-\mu\etay/2,\;\alpha\eqdef 2L^2\etax\etay(1+2\condnum L\etax),\;\theta \eqdef 2(\condnum L\etax)^2\etax +\condnum L\etax^2 & \nnl 
    & D_\iter \eqdef L^2\etax\norm{\yiter-\ystar_\iter}^2,\;B_\iterM \eqdef 2\condnum L\etax^2\norm{\nabla \pf(\xiterM)}^2+\frac{3\condnum L\etax^2\norm{\xerror{\iterM}}^2}{2}+ 4L^2\etax\etay^2\norm{\yerror{\iterM}}^2 & \nnl 
    & X_\iterM \eqdef \inner{\nabla \pf(\xiterM)}{\xerror{\iterM}},\; Y_\iterM \eqdef \inner{\yerror{\iterM}}{\yiterM-\ystar_\iterM},
\end{align}
then by~\Cref{lem:app/Analysis_SGDA_v2/dual_error(One_Step)_SGDA}, we have
\begin{align}
    D_\iter \leq \beta D_\iterM + B_\iterM + \theta X_\iterM + \alpha Y_\iterM .
\end{align}
Note that since we are asking for $\etay\leq1/2L$, then we have $0<\beta<1$ due to $\mu\etay/2\leq 1/4\condnum<1$. Continuing recursively, we obtain the following:
\begin{align}
    D_\iter \leq& \beta^{\iterM} D_1 + \sum_{s=1}^{\iterM} \beta^{\iterM-s}(B_s+\theta X_s + \alpha Y_s)
\end{align}
We sum both sides of above inequality from $t=1$ to $T$ to obtain
\begin{align}
    \sumIter D_\iter =& \frac{D_1}{1-\beta} + \sumIterM \sum_{s=1}^{\iterM} \beta^{\iterM-s}(B_s+\theta X_s + \alpha Y_s)\nnl 
    =& \frac{D_1}{1-\beta} + \sumIterM \xkh{\sum_{s=0}^{T-1-t}\beta^s}(B_t+\theta X_t + \alpha Y_t)\nnl 
    \leq& \frac{D_1}{1-\beta} + \sumIter \xkh{\sum_{s=0}^{T-1}\beta^s} B_\iter + \abs{\sumIterM\xkh{\sum_{s=0}^{T-1-t}\beta^s} \theta X_t} + \abs{\sumIterM\xkh{\sum_{s=0}^{T-1-t}\beta^s}\alpha Y_t}\nnl 
    \leq& \frac{D_1+\sumIter B_\iter}{1-\beta} + {\abs{\sumIterM\xkh{\sum_{s=0}^{T-1-t}\beta^s} \theta X_t} + \abs{\sumIterM\xkh{\sum_{s=0}^{T-1-t}\beta^s}\alpha Y_t}}.
\end{align}
By raising both of the above inequalities to the power of $p/2$ and then taking the expectation, we obtain
\begin{align}\label{eq:app/SGDA/Proof_Lemma_3/UB1}
    \Exp{\xkh{\sumIter D_\iter}^{\nicefrac{p}{2}}} &\leq \xkh{\frac{2D_1}{\mu\etay}}^{\nicefrac{p}{2}} + \Exp{\xkh{\frac{\sumIter B_\iter}{1-\beta}}^{\nicefrac{p}{2}}}\nnl 
   &\quad + \underbrace{\Exp{\abs{\sumIterM\xkh{\sum_{s=0}^{T-1-t}\beta^s} \theta X_t}^{\nicefrac{p}{2}}} + \Exp{\abs{\sumIterM\xkh{\sum_{s=0}^{T-1-t}\beta^s}\alpha Y_t}^{\nicefrac{p}{2}}}}_{\texttt{A}}
\end{align}
Then for \texttt{A}, we have
\begin{flalign}\label{eq:app/SGDA/Proof_Lemma_3/UB2}
    \texttt{A} \leq& \sqrt{{\Exp{\abs{\sumIterM\xkh{\sum_{s=0}^{T-1-t}\beta^s} \theta X_t}^p}}} + \sqrt{\Exp{\abs{\sumIterM\xkh{\sum_{s=0}^{T-1-t}\beta^s}\alpha Y_t}^p}}\nnl 
     \stackAlign{\text{\tiny \Cref{lem:app/Supp_lems/p_moment_sum_martingales}}}{\leq}  \sqrt{2^{2-p}\sumIterM \abs{\sum_{s=0}^{T-1-t}\beta^s}^p\theta^p\Exp{\abs{X_\iter}^p}} + \sqrt{2^{2-p}\sumIterM \abs{\sum_{s=0}^{T-1-t}\beta^s}^p\alpha^p\Exp{\abs{Y_\iter}^p}}\nnl 
     \leq& {\sqrt{2^{2-p}\widehat{\theta}^p\sumIter \Exp{\abs{X_\iter}^p}}} + {\sqrt{2^{2-p}\widehat{\alpha}^p\sumIter \Exp{\abs{Y_\iter}^p}}}\nnl 
     \leq& \sqrt{2^{2-p}\widehat{\theta}^p \sumIter\Exp{\norm{\nabla\pf(x_\iter)}^p\norm{\xerror{\iter}}^p}} + \sqrt{2^{2-p}\widehat{\alpha}^p \sumIter\Exp{\norm{\yiter-\ystar_\iter}^p\norm{\yerror{\iter}}^p}}\nnl 
     \leq&  \sqrt{2^{2-p}\var^p\etax^{\nicefrac{p}{2}} \Exp{\sumIter \xkh{\frac{\widehat{\theta}}{\sqrt{\etax}}}^p{\norm{\nabla\pf(x_\iter)}^p}}} + \sqrt{2^{2-p}\var^p(L^2\etax)^{\nicefrac{p}{2}} \Exp{\sumIter \xkh{\frac{\widehat{\alpha}}{L\sqrt{\etax}}}^p{\norm{\yiter-\ystar_\iter}^p}}}\nnl 
     \stackAlign{\text{\tiny \Cref{lem:app/Supp_lems/Holder_Ineq}}}{\leq}\sqrt{2^{2-p}\var^p\xkh{\sumIter \xkh{\frac{\widehat{\theta}}{\sqrt{\etax}}}^{\frac{2p}{2-p}}}^{\frac{2-p}{2}}\Exp{\xkh{\sumIter \etax\norm{\nabla \pf(\xiterM)}^2}^{\nicefrac{p}{2}}}}\nnl 
     &+ \sqrt{2^{2-p}\var^p\xkh{\sumIter \xkh{\frac{\widehat{\alpha}}{L\sqrt{\etax} }}^{\frac{2p}{2-p}}}^{\frac{2-p}{2}}\Exp{\xkh{\sumIter D_\iter}^{\nicefrac{p}{2}}}}\nnl 
     \leq& \frac{\Exp{\xkh{\sumIter \etax\norm{\nabla \pf(\xiter)}^2}^{\nicefrac{p}{2}}}}{17} + \frac{\Exp{\xkh{\sumIter D_\iter}^{\nicefrac{p}{2}}}}{17} + \frac{17\zkh{\xkh{\nicefrac{\widehat{\theta}}{\sqrt{\etax}}}^p +\xkh{\nicefrac{\widehat{\alpha}}{L\sqrt{\etax} }}^p}\var^pT^{1-\nicefrac{p}{2}}}{2}
\end{flalign}
where $\widehat{\theta}\eqdef \theta/(1-\beta)$ and $\widehat{\alpha}\eqdef \alpha/(1-\beta)$. We also can get
\begin{align}\label{eq:app/SGDA/Proof_Lemma_3/UB3}
    \Exp{\xkh{\frac{\sumIter B_\iter}{1-\beta}}^{\nicefrac{p}{2}}} \leq& \Exp{\xkh{\sumIter \frac{2\condnum L\etax^2\norm{\nabla \pf(\xiter)}^2}{\mu\etay}}^{\nicefrac{p}{2}}} + \zkh{\xkh{\frac{3\condnum L \etax^2}{\mu\etay}}^{\nicefrac{p}{2}} + \xkh{\frac{8 L^2\etax\etay}{\mu}}^{\nicefrac{p}{2}} }T\var^p\nnl 
    \leq& \xkh{\frac{2\condnum^2\etax}{\etay}}^{\nicefrac{p}{2}} \Exp{\xkh{\sumIter \etax\norm{\nabla \pf(\xiter)}^2}^{\nicefrac{p}{2}}}+ \zkh{2\xkh{\frac{2\condnum^2\etax}{\etay}}^{\nicefrac{p}{2}}\etax^{\nicefrac{p}{2}} + \xkh{8\condnum L\etax\etay}^{\nicefrac{p}{2}}}T\var^p\nnl 
    \stackAlign{(a)}{\leq} \frac{\Exp{\xkh{\sumIter \etax\norm{\nabla \pf(\xiter)}^2}^{\nicefrac{p}{2}}}}{17} +\xkh{\frac{2\etax^{\nicefrac{p}{2}}}{17}+4(\condnum\etax)^{\nicefrac{p}{2}}}T\var^p \nnl 
     \stackAlign{\condnum\geq1}{\leq} \frac{\Exp{\xkh{\sumIter \etax\norm{\nabla \pf(\xiter)}^2}^{\nicefrac{p}{2}}}}{17} +\xkh{\frac{2}{17}+4}T(\condnum\etax)^{\nicefrac{p}{2}}\var^p,
\end{align}
where $(a)$ holds due to $(2\condnum^2\etax/\etay)^{\nicefrac{p}{2}}\leq 1/17$ and $\etay\leq 1/(2L)$.
We now  substitute~\eqref{eq:app/SGDA/Proof_Lemma_3/UB2} and~\eqref{eq:app/SGDA/Proof_Lemma_3/UB3} into~\eqref{eq:app/SGDA/Proof_Lemma_3/UB1}, then multiply both sides of the inequaity we obtained by $17/4$ to get
\begin{align}
    4\Exp{\xkh{\sumIter D_\iter}^{\nicefrac{p}{2}}}\leq& \frac{1}{2} \Exp{\xkh{\sumIter \etax\norm{\nabla \pf(\xiter)}^2}^{\nicefrac{p}{2}}} \nnl  
    &+9\xkh{\frac{D_1}{\mu\etay}}^{\nicefrac{p}{2}} +  18 (\condnum\etax)^{\nicefrac{p}{2}}T\var^p +37 T^{1-\nicefrac{p}{2}}\var^p\zkh{\xkh{\frac{\widehat{\theta}}{\sqrt{\etax}}}^p+\xkh{\frac{\widehat{\alpha}}{L\sqrt{\etax} }}^p}\nnl 
    \stackAlign{(b)}{\leq} \frac{1}{2} \Exp{\xkh{\sumIter \etax\norm{\nabla \pf(\xiter)}^2}^{\nicefrac{p}{2}}} \nnl  
    &+9\xkh{\frac{\delta_1^2}{2\etay}}^{\nicefrac{p}{2}} + 18 (\condnum\etax)^{\nicefrac{p}{2}}T\var^p + 2405 T^{1-\nicefrac{p}{2}}(\condnum\var)^p\etax^{\nicefrac{p}{2}},
\end{align}
where $(b)$ holds because
\begin{align}
    & \frac{\widehat{\theta}}{\sqrt{\etax}} = \frac{2\condnum^2\etax}{\etay}(2L\etax+1/2)\sqrt{\etax}\leq \frac{(1/\condnum+1/2)}{17^{\nicefrac{2}{p}}}\sqrt{\etax} \leq \sqrt{\etax}\tag{$\xkh{\frac{2\condnum^2\etax}{\etay}}^{\nicefrac{p}{2}}\leq 1/17,\; \condnum L\etax\leq\frac{1}{2}$}\\ 
    &\frac{\widehat{\alpha}}{L\sqrt{\etax}} = 4\condnum(1+2\condnum L\etax)\sqrt{\etax}\leq 8 \condnum \sqrt{\etax},\quad \frac{D_1}{\mu\etay} =  \frac{\condnum L\etax}{\etay} \norm{\y_1-\y^\star_1}^2\leq \frac{\norm{\y_1-\y^\star_1}^2}{2\etay} \tag{$\condnum L\etax\leq\frac{1}{2}$}
\end{align}
Now, we finish the proof.
\end{proofof}
\subsubsection{Prove the complexity}
\begin{proofof}
If we let $\etax= 1/(1156 \condnum^2 L)$ and $\etay=1/(2L)$, then we can easily verify that $\etax\leq \min\dkh{1/(8\condnum L), \etay/2(17^{\nicefrac{2}{p}}\condnum^2)}$ holds true. Therefore, we can substitute~\eqref{eq:app/Analysis_SGDA_v2/Exp_dual_error_SGDA} into~\eqref{eq:app/Analysis_SGDA_v2/dual_error_SGDA} to get
\begin{align}
    &\ \frac{1}{2}\Exp{\xkh{\sumIter {\etax\norm{\nabla\pf(\xiter)}^2 }}^\frac{p}{2}} &\nnl 
    &\ \leq3\Delta_1^{\nicefrac{p}{2}} + 8(\condnum L)^{\nicefrac{p}{2}} (\etax\var)^p T + 6\var^p T^{1-\nicefrac{p}{2}}\etax^{\nicefrac{p}{2}} + 9 \xkh{\frac{\delta_1^2}{2\etay}}^{\nicefrac{p}{2}} + 18(\condnum\etax)^{\nicefrac{p}{2}}T\var^p + 2405T^{1-\nicefrac{p}{2}}(\condnum\var)^p\etax^{\nicefrac{p}{2}} &\nnl 
     &\ \leq 3\Delta_1^{\nicefrac{p}{2}} + 9L^{\nicefrac{p}{2}}\delta_1^p + 8(\condnum L)^{\nicefrac{p}{2}} (\etax\var)^p T + 18(\condnum\etax)^{\nicefrac{p}{2}}T\var^p + 2411T^{1-\nicefrac{p}{2}}(\condnum\var)^p\etax^{\nicefrac{p}{2}} 
\end{align}
By multiplying both sides of the above equation by $(\etax T)^{-\nicefrac{p}{2}}$, we get
\begin{align}\label{eq:app/SGDA/Proof_Theorem/UB1}
   &\  \frac{1}{2}\Exp{\xkh{\frac{1}{T}\sumIter {\norm{\nabla\pf(\xiter)}^2 }}^\frac{p}{2}} &\nnl 
   &\ \leq 3\xkh{\frac{\Delta_1}{T\etax}}^{\nicefrac{p}{2}} + 9 \xkh{\frac{L\delta_1^2}{T\etax}}^{\nicefrac{p}{2}} + 8(\condnum L)^{\nicefrac{p}{2}}\etax^{\nicefrac{p}{2}}\var^pT^{1-\nicefrac{p}{2}} + 18\condnum^{\nicefrac{p}{2}} T^{1-\nicefrac{p}{2}}\var^p + \frac{2411(\condnum\var)^p}{T^{p-1}} & \nnl 
   &\ \leq 3468\xkh{\condnum\sqrt{\frac{\Delta_1L}{T}}}^p + 10404\xkh{\frac{L\condnum\delta_1}{\sqrt{T}}}^p + 19\condnum^{\nicefrac{p}{2}} T^{1-\nicefrac{p}{2}}\var^p + \frac{2411(\condnum\var)^p}{T^{p-1}}\nnl 
   &\ \leq 3468\xkh{\condnum\sqrt{\frac{\Delta_1L}{T}}}^p + 10404\xkh{\frac{L\condnum\delta_1}{\sqrt{T}}}^p + \frac{38\condnum^{\nicefrac{p}{2}}\sigma^pT^{1-\nicefrac{p}{2}}}{B^{p-1}} + \frac{4822(\condnum\sigma)^p}{(BT)^{p-1}}.
\end{align}
Note that the last inequality holds because we have
\begin{align}\label{eq:app/SGDA/Proof_Theorem/UB2}
    \Epiter{\|\zeta_{\x,t}\|^p} = & \frac{\Epiter{\norm{\sum_{s=1}^B {\pgradX f(\xiter,\yiter,\xi_{t,s})-\pgradX f(\xiter,\yiter)}}^p }}{B^{p}} \nnl \stackAlign{\text{\tiny \Cref{lem: app/Supp_lems/p_momment_batch_martingales}}} {\leq} 2B^{-p}\sum_{s=1}^B \Epiter{\|\pgradX f(\xiter,\yiter,\xi_{t,s})-\pgradX  f(\xiter,\yiter)\|^p}
    \leq \frac{2\sigma^p}{B^{p-1}}
\end{align}
From the law of total expectation, we get $\Exp{\|\zeta_{\x,t}\|^p} = \Exp{\Epiter{\|\zeta_{\x,t}\|^p}}\leq \nicefrac{2\sigma^p}{B^{p-1}}$. The same analysis applies to $\Exp{\|\zeta_{\y,t}\|^p}$, and we can conclude that $\var^p= \nicefrac{2\sigma^p}{B^{p-1}}$ holds.

Now, if we let
\begin{align}
  &\  T \asymeq \max\dkh{\frac{\condnum^2\Delta_1 L}{\eps^2},\frac{(L\condnum\delta_1)^2} {\eps^2}}, & \nnl 
  &\ B\asymeq \max\dkh{\upround{\frac{\condnum^{\frac{p}{2(p-1)}}\sigma^{\frac{p}{p-1}}T^\frac{2-p}{2(p-1)}}{\eps^\frac{p}{p-1}}},1}\asymeq \max\dkh{\upround{\frac{\condnum^\frac{4-p}{2(p-1)}(\sqrt{\Delta_1L})^\frac{2-p}{p-1}\sigma^\frac{p}{p-1}}{\eps^\frac{2}{p-1}}},\upround{\frac{\condnum^\frac{4-p}{2(p-1)}(L\delta_1)^\frac{2-p}{p-1}\sigma^{\frac{p}{p-1}}}{\eps^\frac{2}{p-1}}},1}, & 
\end{align}
then we can conlude that 
\begin{align}
   \xkh{\condnum\sqrt{\frac{\Delta_1L}{T}}}^p + \xkh{\frac{L\condnum\delta_1}{\sqrt{T}}}^p + \frac{\condnum^{\nicefrac{p}{2}}\sigma^pT^{1-\nicefrac{p}{2}}}{B^{p-1}} \asymleq \eps^p
\end{align}
If, in addition, $\eps$ is small enough that $\eps<\sqrt{\condnum}\min\dkh{\sqrt{\Delta_1L},L\delta_1}$ holds true, then we can obtain
\begin{align}
    \frac{(\sigma\condnum)^p}{(BT)^{p-1}}\asymgeq (\sigma\condnum)^p\cdot\xkh{\frac{\eps^2}{T\condnum\sigma^2}}^{\nicefrac{p}{2}}\asymeq&  (\sigma\condnum)^p\cdot\min\dkh{\xkh{\frac{\eps^4}{\condnum^3\Delta_1L\sigma^2}}^{\nicefrac{p}{2}}, \xkh{\frac{\eps^4}{\condnum^3(L\delta_1\sigma)^2}}^{\nicefrac{p}{2}}}\nnl &= \xkh{\frac{\eps}{\sqrt{\condnum}\min\dkh{\sqrt{\Delta_1L},L\delta_1}}}^p \eps^p < \eps^p 
\end{align}
Hence, to achieve $\xkh{T^{-1}\Exp{{\sumIter {\norm{\nabla\pf(\xiter)} }}}}^p\leq\Exp{\xkh{T^{-1}\sumIter {\norm{\nabla\pf(\xiter)}^2 }}^\frac{p}{2}}\leq\eps^p$, the total gradient complexity is upper bounded by 
\begin{align}
    B\cdot T \asymleq& \max\dkh{{\frac{\condnum^{\frac{p}{2(p-1)}}\sigma^{\frac{p}{p-1}}T^\frac{p}{2(p-1)}}{\eps^\frac{p}{p-1}}},T}\nnl 
    \asymleq& \max\dkh{\frac{\condnum^\frac{3p}{2(p-1)}(\sigma\cdot\max\dkh{\sqrt{\Delta_1L},(L\delta_1)})^\frac{p}{p-1}}{\eps^\frac{2p}{p-1}}, \frac{\condnum^2\Delta_1 L}{\eps^2},\frac{(L\condnum\delta_1)^2} {\eps^2} }
\end{align}
\end{proofof}
\clearpage

\subsection{Proof of Theorem~\ref{thm:UpperBound_SPGDA_NC-C}}\label{sec:proof_of_sgda_nc-c}
First, we prove the following key lemmas
\begin{lemma}\label{lem:app/SGDA_NC-C/primal_descent}
    Let Assumptions~\ref{assum: smoothness_heavy-tailed_noise} (\textbf{$L$-smoothness and heavy-tailed noise}) and~\ref{assum:prelims/nonconvex-concave-setting} (\textbf{NC-C Setting}) hold. For \sgda with {\color{blue} Case \romanTwo}, we have 
    \begin{align}
       \frac{1}{8}\Exp{\xkh{\sumIter \etax\|\nabla\pf_{1/2L}(\xiter)\|^2}^{\nicefrac{p}{2}}} \leq& 2(L\etax)^{\nicefrac{p}{2}}\Exp{\xkh{\sumIter (\pf(\xiter) -\whatf(\xiter,\yiter))}^{\nicefrac{p}{2}}} \nnl 
       &+\Delta_1^{\nicefrac{p}{2}} + 2(LT)^{\nicefrac{p}{2}}\etax^p(G^p+\sigma^p) + 4\sigma^pT^{1-\nicefrac{p}{2}}\etax^{\nicefrac{p}{2}},
    \end{align}
    where $\Delta_1\eqdef \pf(\x_1)-\min_\x\pf(\x)$ and $\whatf(\x,\y)\eqdef f(\x,\y)-h(\y)$. 
\end{lemma}
\begin{proofof}
By Lemma~\ref{lem:prelims/weakly-convexity-primal-function}, we know $\pf(\x)$ is a $L$-weakly convex function on $\mathcal{X}$. Let $\xiter^+ \eqdef \argmin_{\x\in\mathcal{X}} \dkh{\pf(\x)+L\|\x-\xiter\|^2}$, and we have
\begin{align}\label{eq:app/SGDA_NC-C/UB1}
    &\ \pf_{1/2L}(\xiterP)  = \pf(\xiterP^+) + L\norm{\xiterP^+ - \xiterP}^2 & \nnl 
    &\ \leq \pf(\xiter^+) + L\norm{\xiter^+ - \xiterP}^2 & \nnl 
    &\ = \pf(\xiter^+) + L\norm{\Pi_{\mathcal{X}}(\xiter^+) - \Pi_{\mathcal{X}}(\xiter-\etax\pgradX f(\xiter,\yiter,\xi_\iter))}^2 & \nnl 
    &\ \leq \pf(\xiter^+) + L\norm{\xiter^+ - \xiter+\etax\pgradX f(\xiter,\yiter,\xi_\iter)}^2 & \nnl
    &\ = \pf_{1/2L}(\xiter) + 2L\etax\inner{\xiter^+-\xiter}{\pgradX f(\xiter,\yiter)} -  \etax\inner{\nabla \pf_{1/2L}(\xiter)}{\errorxiter} + L\etax^2\norm{\pgradX f(\xiter,\yiter) + \errorxiter}^2 & \nnl 
    &\ \leq  \pf_{1/2L}(\xiter) + 2L\etax\inner{\xiter^+-\xiter}{\pgradX f(\xiter,\yiter)} -  \etax\inner{\nabla \pf_{1/2L}(\xiter)}{\errorxiter} + 2L\etax^2G^2 + 2L\etax^2\norm{\errorxiter}^2,
\end{align}
where we let $\errorxiter \eqdef \pgradX f(\xiter,\yiter,\xi_\iter)-\pgradX f(\xiter,\yiter)$, and for the upper bound of $2L\etax\inner{\xiter^+-\xiter}{\pgradX f(\xiter,\yiter)}$, by the $L$-smoothness of $f(\cdot,\yiter)$ \wrt $\x$, we can get 
\begin{align}\label{eq:app/SGDA_NC-C/UB2}
    \inner{\xiter^+-\xiter}{\pgradX f(\xiter,\yiter)} \leq& f(\xiter^+,\yiter) -f(\xiter,\yiter) + \frac{L}{2}\|\xiter^+-\xiter\|^2 \nnl
    =& \whatf(\xiter^+,\yiter) -\whatf(\xiter,\yiter) + \frac{L}{2}\|\xiter^+-\xiter\|^2 \nnl
    \leq& \pf(\xiter^+) -\whatf(\xiter,\yiter) + \frac{L}{2}\|\xiter^+-\xiter\|^2 \nnl 
    =& \pf_{1/2L}(\xiter^+) -\whatf(\xiter,\yiter) - \frac{\|\nabla \pf_{1/2L}(\xiter)\|^2}{8L} \nnl 
    \leq& \pf(\xiter) -\whatf(\xiter,\yiter) - \frac{\|\nabla \pf_{1/2L}(\xiter)\|^2}{8L}.
\end{align} 
Substitute~\eqref{eq:app/SGDA_NC-C/UB2} into~\eqref{eq:app/SGDA_NC-C/UB1}, then sum both sides of the resulting inequality from $t=1$ to $T$ to obtain
\begin{align}
    &\ \sumIter \frac{\etax\|\nabla \pf_{1/2L}(\xiter)\|^2}{4}&\nnl &\ \leq \Delta_1 + 2L\etax\sumIter (\pf(\xiter) -\whatf(\xiter,\yiter)) + \abs{\sumIter \etax\inner{\nabla \pf_{1/2L}(\xiter)}{\errorxiter}} + 2TL\etax^2(G^2+\|\errorxiter\|^2). \notag  
\end{align}
Note that $\pf_{1/2L}(\x_1)-\pf_{1/2L}(\x_{T+1}) \leq \pf(\x_1)-\min_\x\pf(\x)\eqdef \Delta_1$. Next, raise both sides of the above inequality to the power of $p/2$ to obtain
\begin{align}
    &\ \frac{1}{4}\xkh{\sumIter{\etax\|\nabla \pf_{1/2L}(\xiter)\|^2}}^{\nicefrac{p}{2}} &\nnl &\ \leq \Delta_1^{\nicefrac{p}{2}} + 2(L\etax)^{\nicefrac{p}{2}}\xkh{\sumIter (\pf(\xiter) -\whatf(\xiter,\yiter))}^{\nicefrac{p}{2}} + \xkh{\abs{\sumIter \etax\inner{\nabla \pf_{1/2L}(\xiter)}{\errorxiter}}}^{\nicefrac{p}{2}} + 2(TL)^{\nicefrac{p}{2}}\etax^p(G^p+\|\errorxiter\|^p),\notag
\end{align}
where we used the fact that $1/4^{\nicefrac{p}{2}}\geq 1/4$ and $(a+b)^{\nicefrac{p}{2}}\leq a^{\nicefrac{p}{2}} + b^{\nicefrac{p}{2}}$ for $a,b\geq 0$. Then, taking the expectation on both sides of the above inequality yields
\begin{align}\label{eq:app/SGDA_NC-C/UB3}
     &\ \frac{1}{4}\Exp{\xkh{\sumIter{\etax\|\nabla \pf_{1/2L}(\xiter)\|^2}}^{\nicefrac{p}{2}}} &\nnl &\ \leq \Delta_1^{\nicefrac{p}{2}} + 2(L\etax)^{\nicefrac{p}{2}}\Exp{\xkh{\sumIter (\pf(\xiter) -\whatf(\xiter,\yiter))}^{\nicefrac{p}{2}}} + \Exp{\xkh{\abs{\sumIter \etax\inner{\nabla \pf_{1/2L}(\xiter)}{\errorxiter}}}^{\nicefrac{p}{2}}}& \nnl 
     &\ + 2(TL)^{\nicefrac{p}{2}}\etax^p(G^p+\sigma^p)
\end{align}
We can observe that $\{\inner{\etax\nabla\pf_{1/2L}(\xiter)}{\errorxiter}\}_{t\geq 1}$ is a martingale difference sequence \wrt $\{\mathcal{F}_t\}_{t\geq 0}$, where $\mathcal{F}_t\eqdef \sigma(\xi_1,\hdots,\xi_t)$ and $\mathcal{F}_0\eqdef \{\varnothing\}$. Therefore, we can get 
\begin{align}\label{eq:app/SGDA_NC-C/UB4}
    \Exp{\xkh{\abs{\sumIter \etax\inner{\nabla \pf_{1/2L}(\xiter)}{\errorxiter}}}^{\nicefrac{p}{2}}} \leq& \sqrt{\Exp{\xkh{\abs{\sumIter \etax\inner{\nabla \pf_{1/2L}(\xiter)}{\errorxiter}}}^{p}}} \nnl 
    \stackAlign{\text{\Cref{lem:app/Supp_lems/p_moment_sum_martingales}}}{\leq} \sqrt{2^{2-p}\sumIter \Exp{\abs{ \etax\inner{\nabla \pf_{1/2L}(\xiter)}{\errorxiter}}^p}} \nnl 
    \leq& \sqrt{2^{2-p}\sigma^p \sumIter \Exp{\etax^p\|\nabla\pf_{1/2L}(\xiter)\|^p}}\nnl 
    =& \sqrt{ 2^{2-p}\sigma^p \Exp{\sumIter\etax^{\nicefrac{p}{2}}\cdot \etax^{\nicefrac{p}{2}}\|\nabla\pf_{1/2L}(\xiter)\|^p} 
    }\nnl 
    \leq& \sqrt{ 2^{2-p}\sigma^p \Exp{\xkh{\sumIter\etax^{\frac{p}{2-p}}}^{\frac{2-p}{2}} \cdot \xkh{\sumIter \etax\|\nabla\pf_{1/2L}(\xiter)\|^2}^\frac{p}{2}}} \nnl 
    =& \sqrt{ 2^{2-p}\sigma^p T^{1-\nicefrac{p}{2}}\etax^{\frac{p}{2}}\Exp{\xkh{\sumIter \etax\|\nabla\pf_{1/2L}(\xiter)\|^2}^\frac{p}{2}}} \nnl 
    \leq& \frac{1}{8}\Exp{\xkh{\sumIter \etax\|\nabla\pf_{1/2L}(\xiter)\|^2}^\frac{p}{2}} + 4\sigma^p T^{1-\nicefrac{p}{2}}\etax^{\frac{p}{2}} 
\end{align}
After substituting~\eqref{eq:app/SGDA_NC-C/UB3} into~\eqref{eq:app/SGDA_NC-C/UB4}, we finish the proof
\end{proofof}
\begin{lemma}\label{lem:app/SGDA_NC-C/dual_ascent}
    Let Assumptions~\ref{assum: smoothness_heavy-tailed_noise} (\textbf{$L$-smoothness and heavy-tailed noise}) and~\ref{assum:prelims/nonconvex-concave-setting} (\textbf{NC-C Setting}) hold. For \sgda with {\color{blue} Case \romanTwo}, we have  
    \begin{align}
       \sumIter \Exp{\pf(\xiter)-\whatf(\xiter,\yiter)} \leq \whatdelta_1 + T\zkh{\diaY\sqrt{\frac{2G(G+\sigma)\etax}{\etay}} + 2G(\sigma+G)\etax + \diaY^{2-p}\etay^{p-1}\sigma^p},
    \end{align}
    where $\whatdelta_1\eqdef \pf(\x_1)-\whatf(\x_1,\y_1)$. 
\end{lemma}
\begin{proofof}
On the one hand, since $\yiter = \prox{\etay h}{\yiterM +\etay \pgradY f(\xiterM,\yiterM,\xi_\iterM)}$, we have \[0\in\partial h(\yiter) + \etay^{-1}(\yiter-\yiterM-\etay\pgradY f(\xiterM,\yiterM,\xi_\iterM) )\] by the optimality condition; on the other hand, due to the convexity of $h(\y)$, for any $\y\in\dom{h}$, we have
    \begin{align}\label{eq:app/SGDA_NC-C/UB5}
        &\ -h(\y) \leq -h(\yiter) + \frac{\inner{\yiter-\yiterM-\etay\pgradY f(\xiterM,\yiterM,\xi_\iterM) }{\y-\yiter}}{\etay} & \nnl
        &\ = -h(\yiter) + \frac{\inner{\yiter-\yiterM}{\y-\yiter}}{\etay} - \inner{\pgradY f(\xiterM,\yiterM,\xi_\iterM)}{\y-\yiterM}&\nnl &\ + \inner{\pgradY f(\xiterM,\yiterM)}{\yiter-\yiterM} + \inner{\errorxiter}{\yiter-\yiterM}\nnl 
        &\ = -h(\yiter) + \frac{\|\yiterM-\y\|^2-\|\yiter-\y\|^2}{2\etay} - \frac{\|\yiter-\yiterM\|^2}{2\etay}- \inner{\pgradY f(\xiterM,\yiterM,\xi_\iterM)}{\y-\yiterM}&\nnl &\ + \inner{\pgradY f(\xiterM,\yiterM)}{\yiter-\yiterM} + \inner{\errorxiter}{\yiter-\yiterM} &\nnl
        & \leq -h(\yiter) + \frac{\|\yiterM-\y\|^2-\|\yiter-\y\|^2}{2\etay}+ f(\xiterM,\yiter) -f(\xiterM,\yiterM) + \xkh{\frac{L}{2}-\frac{1}{2\etay}}\|\yiter-\yiterM\|^2&\nnl &\ - \inner{\pgradY f(\xiterM,\yiterM,\xi_\iterM)}{\y-\yiterM} + \inner{\errorxiter}{\yiter-\yiterM},
    \end{align}
  where the last inequality holds becasue $\inner{\pgradY f(\xiterM,\yiterM)}{\yiter-\yiterM} \leq f(\xiterM,\yiter)-f(\xiterM,\yiterM)+(L/2)\|\yiter-\yiterM\|^2$ due to the $L$-smoothness of $f(\xiter,\cdot)$ \wrt $\y$. In addtion, we notice that 
\begin{align}
    \inner{\yerror{\iterM}}{\yiter-\yiterM} \leq& \norm{\yerror{\iterM}}\norm{\yiter-\yiterM}\nnl
    =& \norm{\yerror{\iterM}} \norm{\yiter-\yiterM}^{1-\frac{2}{\beta}}\norm{\yiter-\yiterM}^{\frac{2}{\beta}}\nnl
    =& \xkh{\frac{4\etay}{\beta}}^{\frac{1}{\beta}}\norm{\yerror{\iterM}}\norm{\yiter-\yiterM}^{1-\frac{2}{\beta}}\xkh{\frac{\beta\norm{\yiter-\yiterM}^2}{4\etay}}^{\frac{1}{\beta}}\nnl
    =& \zkh{\xkh{\frac{4\etay}{\beta}}^{p/\beta}\norm{\yerror{\iterM}}^p\norm{\yiter-\yiterM}^{p-\frac{2p}{\beta}}}^{\frac{1}{p}}\zkh{\xkh{\frac{\beta\norm{\yiter-\yiterM}^2}{4\etay}}^{\frac{p}{\beta(p-1)}}}^{\frac{p-1}{p}}\nnl 
    \stackAlign{\text{Young's Inequality}}{\leq} \frac{1}{p}\xkh{\frac{4\etay}{\beta}}^{\frac{p}{\beta}}\norm{\yerror{\iterM}}^p\norm{\yiter-\yiterM}^{p-\frac{2p}{\beta}} + \frac{p-1}{p}\xkh{\frac{\beta\norm{\yiter-\yiterM}^2}{4\etay}}^{\frac{p}{\beta(p-1)}}\notag
\end{align}
We let $\beta = \nicefrac{(p-1)}{p}$ to get
\begin{align}\label{eq:app/SGDA_NC-C/UB6}
    \inner{\yerror{\iterM}}{\yiter-\yiterM} \leq& \frac{(4p-4)^{p-1}}{p^p}\|\yiter-\yiterM\|^{2-p}\etay^{p-1}\|\yerror{\iterM}\|^p + \frac{\norm{\yiter-\yiterM}^2}{4\etay} \nnl \leq& \diaY^{2-p}\etay^{p-1}\|\yerror{\iterM}\|^p + \frac{\norm{\yiter-\yiterM}^2}{4\etay},
\end{align}
where the last inequality holds due to $\nicefrac{(4p-4)^{p-1}}{p^p}\leq 1$ for $p\in(1,2]$ and $\{\yiter\}_{t\geq 1}$ in the bounded $\dom{h}$. Substitute~\eqref{eq:app/SGDA_NC-C/UB6} into~\eqref{eq:app/SGDA_NC-C/UB5}, then taking the expectation\footnote{Suppose that $\y$ is $\sigma(\xi_{t-2},\hdots,\xi_1)$-measurable, we first take the condtional expectation \wrt $\sigma(\xi_{t-2},\hdots,\xi_1)$, then we take the full expectation.} on both sides of the resulting inequality, we can get
\begin{align}
   &\ \Exp{f(\xiterM,\y)-f(\xiterM,\yiterM)}&\nnl &\ \leq \Exp{\inner{\pgradY f(\xiterM,\yiterM)}{\y-\yiterM}}&\nnl 
   &\ \leq \Exp{h(\y)-h(\yiter)+\frac{\|\yiterM-\y\|^2-\|\yiter-\y\|^2}{2\etay}+ f(\xiterM,\yiter) -f(\xiterM,\yiterM)} &\nnl
   &+ \Exp{\xkh{\frac{L}{2}-\frac{1}{4\etay}}\|\yiter-\yiterM\|^2}+\diaY^{2-p}\etay^{p-1}\sigma^p &\nnl 
   &\ \leq \Exp{h(\y)-h(\yiter)+\frac{\|\yiterM-\y\|^2-\|\yiter-\y\|^2}{2\etay}+ f(\xiterM,\yiter) -f(\xiterM,\yiterM)}+\diaY^{2-p}\etay^{p-1}\sigma^p,\notag
\end{align}
where the first inequality holds due to the concavity of $f(\xiterM,\cdot)$ \wrt $\y$, and the last inequality holds because we let $\etay\leq 1/2L$. The above inequality is equivently to 
\begin{align}
    \Exp{\whatf(\xiterM,\y)-\whatf(\xiterM,\yiter)} \leq \Exp{\frac{\|\yiterM-\y\|^2-\|\yiter-\y\|^2}{2\etay}} + \diaY^{2-p}\etay^{p-1}\sigma^p.
\end{align}
To upper bound $\Exp{\pf(\xiterM)-\whatf f(\xiterM,\yiterM)}$, we consider folloiwng decomposition:
\begin{align}
    \pf(\xiterM)-\whatf(\xiterM,\yiterM) =& \pf(\xiterM) -\whatf(\xiterM,\y)+ \whatf(\xiterM,\y) - \whatf(\xiterM,\yiter)\nnl 
    +&\whatf(\xiterM,\yiter) -\whatf(\xiter,\yiter)+\whatf(\xiter,\yiter) -\whatf(\xiterM,\yiterM).\notag
\end{align}
By the $G$-Lipschitz continuity of $f(\cdot,\yiter)$ \wrt $\x$, we have
\begin{align}
    \whatf(\xiterM,\yiter) -\whatf(\xiter,\yiter) = f(\xiterM,\yiter)- f(\xiter,\yiter) \leq G\|\xiterM-\xiter\| \leq G\etax(\|\errorxiterM\| + G).\notag
\end{align}
Therefore, we can obtatin
\begin{align}\label{eq:app/SGDA_NC-C/UB7}
    &\ \Exp{ \pf(\xiterM)-\whatf(\xiterM,\yiterM)}& \nnl  
    &\ \leq \Exp{\pf(\xiterM) -\whatf(\xiterM,\y)} + \frac{\Exp{\|\yiterM-\y\|^2-\|\yiter-\y\|^2}}{2\etay}+ \Exp{\whatf(\xiter,\yiter) -\whatf(\xiterM,\yiterM)}\nnl
    &\ + \diaY^{2-p}\etay^{p-1}\sigma^p + G(\sigma+G)\etax
\end{align}

In the subsequent analysis, we will denote $\pf(\xiter)-\whatf(\xiter,\yiter)$ as $\whatdelta_\iter$.
Without loss of generality, we assume that $T$ is large enough to be divided into $N$ windows of width $M$, then $\sumIter \whatdelta_t= \sum_{n=0}^N \sum_{j=1}^M \whatdelta_{nM+j}$. For $\whatdelta_t$ in the $n+1$-th block, i.e., $t=nM+j,\,j=1,\hdots,M$, we substitute $\y=\ystar_{nM+1}\eqdef\argmax_\y\dkh{f(\x_{nM+1},\y)-h(\y)}$ into~\eqref{eq:app/SGDA_NC-C/UB7} to get
\begin{align}
&\ \Exp{\whatdelta_{nM+j}} \leq \Exp{\pf(\x_{nM+j})-\whatf(\x_{nM+j},\ystar_{nM+1})} + \frac{\Exp{\|\y_{nM+j}-\ystar_{nM+1}\|^2-\|\y_{nM+j+1}-\ystar_{nM+1}\|^2}}{2\etay} & \nnl
&\ + \Exp{\whatf(\x_{nM+j+1},\y_{nM+j+1}) -\whatf(\x_{nM+j},\y_{nM+j})}+ \diaY^{2-p}\etay^{p-1}\sigma^p + G(\sigma+G)\etax\notag
\end{align}
We sum both sides of abonve inequality from $j=1$ to $M$ to obtain
\begin{align}\label{eq:app/SGDA_NC-C/UB8}
    &\ \sum_{j=1}^M\Exp{\whatdelta_{nM+j}} \leq \sum_{j=1}^M\Exp{\pf(\x_{nM+j})-\whatf(\x_{nM+j},\ystar_{nM+1})} + \frac{\Exp{\delta_{nM+1}^2}}{2\etay} & \nnl
&\ + \Exp{\whatf(\x_{(n+1)M+1},\y_{(n+1)M+1}) -\whatf(\x_{nM+1},\y_{nM+1})}+ M\zkh{\diaY^{2-p}\etay^{p-1}\sigma^p + G(\sigma+G)\etax}.
\end{align}
Note that 
\begin{align}
    &\ \Exp{\pf(\x_{nM+j})-\whatf(\x_{nM+j},\ystar_{nM+1})} &\nnl
    &\ = \Exp{\whatf(\x_{nM+j},\ystar_{nM+j})-\whatf(\x_{nM+j},\ystar_{nM+1})} & \nnl 
    &\ = \Exp{f(\x_{nM+j},\ystar_{nM+j})- f(\x_{nM+1},\ystar_{nM+j}) + \whatf (\x_{nM+1},\ystar_{nM+j})  -\whatf(\x_{nM+j},\ystar_{nM+1})} & \nnl
    &\ \leq \Exp{f(\x_{nM+j},\ystar_{nM+j})- f(\x_{nM+1},\ystar_{nM+j}) + \whatf (\x_{nM+1},\ystar_{nM+1})  -\whatf(\x_{nM+j},\ystar_{nM+1})}&\nnl 
    &\ = \Exp{f(\x_{nM+j},\ystar_{nM+j})- f(\x_{nM+1},\ystar_{nM+j}) + f (\x_{nM+1},\ystar_{nM+1})  -f(\x_{nM+j},\ystar_{nM+1})}&\nnl
    &\ \leq 2G\Exp{\norm{\x_{nM+j}-\x_{nM+1}}} =\etax\Exp{\left\|\sum_{s=2}^j \pgradX f(\x_{nM+s-1},\y_{nM+s-1},\xi_{nM+s-1})\right\|} & \nnl
    &\ \leq 2G(G+\sigma)(j-1)\etax\notag,
\end{align}
which means 
\begin{align}\label{eq:app/SGDA_NC-C/UB9}
    \sum_{j=1}^M \Exp{\pf(\x_{nM+j})-\whatf(\x_{nM+j},\ystar_{nM+1})} \leq 2G(G+\sigma)\etax\sum_{j=1}^M(j-1)=M(M-1)G(G+\sigma)
\end{align}
Substitute~\eqref{eq:app/SGDA_NC-C/UB9} into~\eqref{eq:app/SGDA_NC-C/UB8} and sum both sides of the resulting inequality from $n=0$ to $N$ to get
\begin{align}\label{eq:app/SGDA_NC-C/UB10}
  \sumIter \Exp{\whatdelta_t} \leq MTG(\sigma+G)\etax + T\diaY^{2-p}\etay^{p-1}\sigma^p + \frac{T\diaY^2}{2M\etay} + \Exp{\whatf(\x_{T+1},\y_{T+1})-\whatf(\x_1,\y_1)}.
\end{align}
Note that $\delta_{nM+1}\leq \diaY$ because $\{\yiterstar\}_{t\geq1}$ also in the bounded $\dom{h}$. In addtion, we have
\begin{align}\label{eq:app/SGDA_NC-C/UB11}
    \Exp{\whatf(\x_{T+1},\y_{T+1})-\whatf(\x_1,\y_1)} =& \Exp{f(\x_{T+1},\y_{T+1})-f(\x_1,\y_{T+1})+\whatf(\x_1,\y_{T+1}) - \whatf(\x_1,\y_1)}\nnl 
    \leq& G\Exp{\|\x_{T+1}-\x_1} + \pf(\x_1) - \whatf(\x_1,\y_1)
    \leq TG(G+\sigma)\etax + \whatdelta_1.
\end{align}
Now, we substitute~\eqref{eq:app/SGDA_NC-C/UB11} into~\eqref{eq:app/SGDA_NC-C/UB10} to get
\begin{align}
    \sumIter \Exp{\whatdelta_t} \leq \whatdelta_1+ (M+1)TG(\sigma+G)\etax + T\diaY^{2-p}\etay^{p-1}\sigma^p + \frac{T\diaY^2}{2M\etay}.\notag
\end{align}
We let $M = \upround{\frac{\diaY}{\sqrt{2G(G+\sigma)\etax\etay}}}$, then we can have
\begin{align}
    \sumIter \Exp{\whatdelta_t} \leq& \whatdelta_1+ \xkh{\frac{\diaY}{\sqrt{2G(G+\sigma)\etax\etay}}+2}TG(\sigma+G)\etax + T\diaY\sqrt{\frac{G(G+\sigma)\etax}{2\etay}} + T\diaY^{2-p}\etay^{p-1}\sigma^p\nnl 
    =& \whatdelta_1 + T\diaY\sqrt{\frac{2G(G+\sigma)\etax}{\etay}} + 2TG(\sigma+G)\etax + T\diaY^{2-p}\etay^{p-1}\sigma^p\notag 
\end{align}
\end{proofof}

\subsubsection{Prove the complexity}
\begin{proofof}
    We assume that that $\etay\leq 1/2L$, then by Lemma~\ref{lem:app/SGDA_NC-C/primal_descent} and Lemma~\ref{lem:app/SGDA_NC-C/dual_ascent}, we can get
    \begin{align}
      &\  \Exp{\xkh{\sumIter \etax\|\nabla\pf_{1/2L}(\xiter)\|^2}^{\nicefrac{p}{2}}} & \nnl 
      &\ \asymleq (L\etax)^{\nicefrac{p}{2}}\Exp{\xkh{\sumIter (\pf(\xiter) -\whatf(\xiter,\yiter))}^{\nicefrac{p}{2}}} +\Delta_1^{\nicefrac{p}{2}} + (LT)^{\nicefrac{p}{2}}\etax^p(G^p+\sigma^p) + \sigma^pT^{1-\nicefrac{p}{2}}\etax^{\nicefrac{p}{2}} & \nnl 
     &\ \asymleq \xkh{L\etax\Exp{{\sumIter (\pf(\xiter) -\whatf(\xiter,\yiter))}}}^{\nicefrac{p}{2}} +\Delta_1^{\nicefrac{p}{2}} + \etax^p(LT(G^2\mmax\sigma^2))^{\nicefrac{p}{2}} + \sigma^pT^{1-\nicefrac{p}{2}}\etax^{\nicefrac{p}{2}}& \nnl 
     &\ \asymleq \xkh{L\etax\whatdelta_1 + TL\etax\xkh{\diaY\sqrt{\frac{G(G\mmax\sigma)\etax}{\etay}} + G(\sigma\mmax G)\etax + \diaY^{2-p}\etay^{p-1}\sigma^p}}^{\nicefrac{p}{2}} &\nnl 
     &\quad +\Delta_1^{\nicefrac{p}{2}} + \etax^p(LT(G^2\mmax\sigma^2))^{\nicefrac{p}{2}} + \sigma^pT^{1-\nicefrac{p}{2}}\etax^{\nicefrac{p}{2}} \notag 
    \end{align}
    Next, we multiply both sides of the above inequality by $(T\etax)^{-p/2}$ to obtain the right-hand side as
    \begin{align}
         \mathrm{L.H.S.} \asymleq& \xkh{\frac{\Delta_1}{T\etax}}^{\nicefrac{p}{2}} + \xkh{(G^2\mmax\sigma^2)L\etax}^{\nicefrac{p}{2}}+\xkh{LG(\sigma\mmax G)\etax}^{\nicefrac{p}{2}} \nnl 
         +& \xkh{L\diaY\sqrt{\frac{G(G\mmax\sigma)\etax}{\etay}} + L\diaY^{2-p}\etay^{p-1}\sigma^p}^{\nicefrac{p}{2}}+ \xkh{\frac{L\whatdelta_1}{T}}^{\nicefrac{p}{2}} + \frac{\sigma^p}{T^{p-1}}\notag
    \end{align}
    We denote $\max\{a,b\}$ and $\min\{a,b\}$ as $a\mmax b$ and $a\mmin b$, respectively. Let $\etax = C_{\x,1}T^{-\frac{2p-1}{3p-2}}\mmin C_{\x,2}T^{-1/2}$ and $\etay = 1/2L\mmin C_\y T^{-\frac{1}{3p-2}}$, where $C_{\x,1},\,C_{\x,2}$ and $C_\y$ are all positive constants, and substitute them into the above equation to obtain the right hand side as
    \begin{align}\label{eq:app/SGDA_NC-C/UB12}
   &\ \mathrm{L.H.S.}& \nnl 
    &\ \asymleq \xkh{\frac{\Delta_1}{T}\xkh{\frac{T^\frac{2p-1}{3p-2}}{C_{\x,1}}\mmax \frac{\sqrt{T}}{C_{\x,2}}}}^{\nicefrac{p}{2}} + \xkh{\frac{C_{\x,2}(G^2\mmax\sigma^2)L}{\sqrt{T}}}^{\nicefrac{p}{2}} +\xkh{\frac{C_{\x,2}LG(\sigma\mmax G)}{\sqrt{T}}}^{\nicefrac{p}{2}}+\xkh{\frac{L\whatdelta_1}{T}}^{\nicefrac{p}{2}} + \frac{\sigma^p}{T^{p-1}} & \nnl 
    &\ + \xkh{L\diaY\sqrt{\frac{G(G\mmax\sigma)C_{\x,1}}{T^\frac{2p-1}{3p-2}}\xkh{2L\mmax \frac{T^\frac{1}{3p-1}}{C_\y}}}+\frac{C_\y^{p-1} L\diaY^{2-p}\sigma^p}{T^\frac{p-1}{3p-2}}}^{\nicefrac{p}{2}}&\nnl 
    &\ \asymleq \xkh{\frac{\Delta_1/C_{\x,1}}{T^\frac{p-1}{3p-2}}}^{\nicefrac{p}{2}} + \xkh{\frac{\Delta_1/C_{\x,2}}{\sqrt{T}}}^{\nicefrac{p}{2}} + \xkh{\frac{LC_{\x,2}(G^2\mmax\sigma^2\mmax G\sigma)}{\sqrt{T}}}^{\nicefrac{p}{2}}+\xkh{\frac{L\whatdelta_1}{T}}^{\nicefrac{p}{2}} + \frac{\sigma^p}{T^{p-1}} & \nnl
    &\ + \xkh{\frac{L\diaY\sqrt{L(G^2\mmax G\sigma)C_{\x,1}}}{T^\frac{2p-1}{6p-4}} }^{\nicefrac{p}{2}} + \xkh{\frac{L\diaY\sqrt{(G^2\mmax G\sigma)C_{\x,1}/C_\y}}{T^\frac{p-1}{3p-2}}+\frac{C_\y^{p-1} L\diaY^{2-p}\sigma^p}{T^\frac{p-1}{3p-2}}}^{\nicefrac{p}{2}} 
    \end{align}
    If we let 
    \begin{align}
        & C_{\x,1}  = \frac{(\Delta_1/L)^\frac{2p-1}{3p-2}}{(\diaY\sigma)^\frac{p}{3p-2}(G\sigma\mmax G^2)^\frac{p-1}{3p-2}},\quad C_{\x,2}=\sqrt{\frac{\Delta_1}{L(G^2\mmax\sigma^2\mmax G\sigma)}},\nnl 
        &C_\y = C_\y = \frac{\diaY^\frac{3p-4}{3p-2}}{\sigma^\frac{3p}{3p-2}}\xkh{\frac{\Delta_1(G\sigma\mmax G^2)}{L}}^\frac{1}{3p-2}.\notag
    \end{align}
    Then, we get
    \begin{align}\label{eq:app/SGDA_NC-C/UB13}
       & \Delta_1/C_{\x,1} = L\diaY\sqrt{(G\sigma\mmax G^2)C_{\x,1}/C_\y} = L\diaY^{2-p}C_\y^{p-1}\sigma^p = (\Delta_1(G\sigma\mmax G^2))^\frac{p-1}{3p-2}(\diaY\sigma)^\frac{p}{3p-2}L^\frac{2p-1}{3p-2}\nnl 
       & \Delta_1/C_{\x,2} = LC_{\x,2}(G^2\mmax\sigma^2\mmax G\sigma) = \sqrt{\Delta_1L(G^2\mmax\sigma^2\mmax G\sigma)}\nnl 
       & L\diaY\sqrt{L(G^2\mmax G\sigma)C_{\x,1}} = \frac{L^\frac{7p-5}{6p-4}\diaY^\frac{5p-4}{6p-4}(\Delta_1(G^2\mmax G\sigma))^\frac{2p-1}{6p-4}}{\sigma^\frac{p}{6p-4}}
    \end{align}
    Now, we substitute~\eqref{eq:app/SGDA_NC-C/UB13} into~\eqref{eq:app/SGDA_NC-C/UB12} to get
    \begin{align}
       &\ \Exp{\xkh{\sumIter \etax\|\nabla\pf_{1/2L}(\xiter)\|^2}^{\nicefrac{p}{2}}}&\nnl 
        &\ \asymleq \xkh{\frac{(\Delta_1(G\sigma\mmax G^2))^\frac{p-1}{3p-2}(\diaY\sigma)^\frac{p}{3p-2}L^\frac{2p-1}{3p-2}}{T^\frac{p-1}{3p-2}}}^{\nicefrac{p}{2}} + \xkh{\frac{L^\frac{7p-5}{6p-4}\diaY^\frac{5p-4}{6p-4}}{\sigma^\frac{p}{6p-4}}\xkh{\frac{\Delta_1(G^2\mmax G\sigma)}{T}}^\frac{2p-1}{6p-4}}^{\nicefrac{p}{2}}&\nnl &\ + \xkh{\sqrt{\frac{\Delta_1L(G^2\mmax\sigma^2\mmax G\sigma)}{T}}}^{\nicefrac{p}{2}} + \xkh{\frac{L\whatdelta_1}{T}}^{\nicefrac{p}{2}} + \frac{\sigma^p}{T^{p-1}}.
    \end{align}
    If, we have $\sigma\geq G$\footnote{Note that we used $\sigma \geq G$, which gives $G^2\mmax\sigma G \leq G^2\mmax\sigma^2$. In fact, we retained the $G^2$ term, so this does not cause the degeneration of the dependence of the upper bound on $G$.}, then we can get
    \begin{align}
         &\ \Exp{\xkh{\sumIter \etax\|\nabla\pf_{1/2L}(\xiter)\|^2}^{\nicefrac{p}{2}}}&\nnl
          &\ \asymleq \xkh{\frac{\Delta_1^\frac{p-1}{3p-2}\sigma\diaY^\frac{p}{3p-2}L^\frac{2p-1}{3p-2}}{T^\frac{p-1}{3p-2}}\mmax \frac{(\Delta_1 G^2)^\frac{p-1}{3p-2}(\sigma\diaY)^\frac{p}{3p-2}L^\frac{2p-1}{3p-2}}{T^\frac{p-1}{3p-2}}}^{\nicefrac{p}{2}}+\xkh{\frac{\Delta_1^\frac{2p-1}{6p-4}\sqrt{\sigma}L^\frac{7p-5}{6p-4}\diaY^\frac{5p-4}{6p-4}}{T^\frac{2p-1}{6p-4}}}^{\nicefrac{p}{2}}&\nnl  &\ + \xkh{\sqrt{\frac{\Delta_1L(G^2\mmax\sigma^2)}{T}}}^{\nicefrac{p}{2}}+\xkh{\frac{L\whatdelta_1}{T}}^{\nicefrac{p}{2}} + \frac{\sigma^p}{T^{p-1}}
    \end{align}
    To achieve $\Exp{\xkh{\sumIter \etax\|\nabla\pf_{1/2L}(\xiter)\|^2}^{\nicefrac{p}{2}}} \leq \eps^p$, the total gradient complexity is upper bounded by
    \begin{align}
        \mathcal{O}\xkh{\frac{\Delta_1\diaY^\frac{p}{p-1}L^\frac{2p-1}{p-1}(\sigma^\frac{3p-2}{p-1}\mmax G^2\sigma^\frac{p}{p-1})}{\eps^\frac{6p-4}{p-1}} + \frac{\Delta_1{\sigma}^\frac{3p-2}{2p-1}L^\frac{7p-5}{2p-1}\diaY^\frac{5p-4}{2p-1}}{\eps^\frac{12p-8}{2p-1}} + \xkh{\frac{\sigma}{\eps}}^\frac{p}{p-1} + \frac{\Delta_1L(G^2\mmax\sigma^2)}{\eps^4} + \frac{L\whatdelta_1}{\eps^2}} \notag
    \end{align}
    If, in addition, we have $\sigma,L,\diaY\geq 1$ and $\eps\leq 1$, then we can get $\sigma^\frac{3p-2}{p-1}\geq\sigma^\frac{3p-2}{2p-1}$, $\diaY^\frac{p}{p-1}\geq \diaY^\frac{5p-4}{2p-1}$,  $L^\frac{2p-1}{p-1}\geq L^\frac{7p-5}{2p-1}$ and $\eps^{-\frac{6p-4}{p-1}}\geq \eps^{-\frac{12p-8}{2p-1}}$, which means the total gradient complexity is upper bounded by
        \begin{align}
        \mathcal{O}\xkh{\frac{\Delta_1\diaY^\frac{p}{p-1}L^\frac{2p-1}{p-1}(\sigma^\frac{3p-2}{p-1}\mmax G^2\sigma^\frac{p}{p-1})}{\eps^\frac{6p-4}{p-1}} + \xkh{\frac{\sigma}{\eps}}^\frac{p}{p-1} + \frac{\Delta_1L(G^2\mmax\sigma^2)}{\eps^4} + \frac{L\whatdelta_1}{\eps^2}} 
    \end{align}
\end{proofof}

\clearpage

\subsection{Proof of Theorem~\ref{thm:lower_bound}}\label{sec:proof_of_lower_bound}
\paragraph{Additional notations} For positive integers $T, a, b$, $[T]$ represents the set ${1,\hdots,T}$, and $[a:b]$ represents the set $\{a,a+1,\hdots,b\}$. Given a real vector $v$, we denote $v[i]$ as the $i$-th coordinate of $v$, $\mathrm{supp}(v)\eqdef \{i\mid v[i]\neq 0\}$. $\|v\|_2$ and $\|v\|_{\infty}$ are the $\ell_2$ and $\ell_\infty$ norm of $v$, respectively. For a differentiable function $f$, $\nabla_i f(\bv)$ represents $\nabla f(\bv)[i]$. Given a matrix $A$, we let $A_{i,j}$ be the $i,j$-th entry of $A$. $I_n\in\real{n\times n}$ represents the identity matrix, and $\mathbf{e}_i$ represents the standard unit vector with all zeros except for the $i$-th coordinate, which is $1$.
\paragraph{Necessary definitions} We introduce some definitions needed to prove the lower bound. They basically come from~\cite{lower_bounds_for_non-convex_stochastic_optimization} and~\cite{complexity_lower_bounds_for_nonconvex_strongly_concave_min_max_optimization} extend them to constrained optimization. A differentiable function $f:\mathcal{V}\subset\real{d}\mapsto \real{}$ is a \textit{(first-order) zero chain} over $\mathcal{V}$ if for any $\bv\in \mathcal{V}$ and $1\leq i\leq d$, $\forall j\in[i:d], \bv[j]=0$ implies that $\forall k\in[i+1:d]$, $\nabla_k f(\bv)=0$. Given a ojective $f:\mathcal{X}\mapsto \real{}$ and its associated first-order stochastic oracle $\nabla f(\x,\xi)$, if $(t+1)$-th iterate $\xiterP$ of a algorithm {\small \sf A} with $\nabla f(\x,\xi)$ awalys satisfies $\xiterP \in \dkh{\Pi_\mathcal{X}(\bv)\mid \mathrm{supp}(\bv)\subset \cup_{0\leq i\leq t}\xkh{\mathrm{supp}(\x_i)\cup\mathrm{supp}(\nabla f(\x_i,\xi))}}$, we say {\small \sf A} is a stochasitc \textit{first-order (zero-respecting) algorithm.} Moreover, if $\mathrm{supp}(\x)\subset [1:i-1]$ implies that $\mathbb{P}\xkh{\mathrm{supp}(\nabla f(\x,\xi))\nsubseteq [1:i-1]}\leq q$ and $\mathbb{P}\xkh{\mathrm{supp}(\nabla f(\x,\xi)) \subset [1:i]}=1$ hold, we say $f$ is a \textit{probability-$q$ zero-chain} with respect to the given oracle $\nabla f(\x,\xi)$.
\paragraph{The hard instance} The hard instance for min-max problems comes from~\cite{complexity_lower_bounds_for_nonconvex_strongly_concave_min_max_optimization}. Specifically, we decompose the minimization variable into $\x\in \real{T}$ and $\z\in\real{T-1}$, and the maximization variable into $\bar{\y}\eqdef [\y^{(1)},\hdots,\y^{(T-1)}]$, where $\y^{(i)}\in\real{n}$. The formal expression of the instance $\widehat{f}:\real{T}\times\real{T-1}\times\real{n(T-1)}\mapsto \real{}$ is following
\begin{align}
    \widehat{f}(\x,\z,\bar{\y}) \eqdef& -\Psi(1)\Phi(\x[1]) + \sum_{2}^T \zkh{\Psi(-\z[i])\Phi(-\x[i])-\Psi(\z[i])\Phi(\x[i])}\nnl  
   & + \sum_{i=1}^{T-1}h(\x[i],\z[i+1],\bar{\y}^{(i)}) + \sum_{i=1}^{T-1} \xkh{c_1\x[i]^2 + c_2\z[i+1]^2},
\end{align}
where 
\begin{align}
    &\ \Psi(u)\eqdef\begin{cases}
            0, & u\leq 1/2\\ \exp\xkh{1-\frac{1}{(2u-1)}^2}, & u>1/2
        \end{cases}, \quad \Phi(u)\eqdef \sqrt{e}\int_{-\infty}^u \exp\xkh{-\frac{t^2}{2}}dt, & \nnl 
        &\ h(x,z,\y) \eqdef \frac{C}{n} \zkh{-\frac{1}{2}\y^\mathrm{T}\xkh{\frac{1}{n^2}I_n+A}\y + \mathbf{b}^T_{x,z}\y},\quad \mathbf{b}^\mathrm{T}_{x,z}\eqdef x\mathbf{e_1} - \frac{1}{2}z \mathbf{e_n},& \nnl
        &\ A\eqdef \begin{cases}
            A_{1,1} = A_{n,n}=1\\ 
            A_{i,i}=2,& \forall i\in[2:n] \\
            A_{i,i+1} = A_{i+1,i}=-1& \forall i\in[1:n-1]\\
            0,&otherwise
        \end{cases}. \nonumber
\end{align}
Here we stipulate $\z[1]=1$. We arrange $\x$, $\z$, $\bar{\y}$ as follows
\begin{align}
    \mathsf{u}\eqdef \xkh{\x[1],\y^{(1)},\z[2],\x[2],\y^{(2)},\z[3],\dots, \x[i],\y^{(i)},\z[i+1],\dots,\z[T],\x[T]}\nonumber
\end{align}
We let $\mathcal{B}^d_R\eqdef \dkh{x\in\real{d}\mid \|\x\|_\infty \leq R}$.
According to Lemma 11, 12 in~\cite{complexity_lower_bounds_for_nonconvex_strongly_concave_min_max_optimization}, $\widehat{f}$ has following key properties for some numercial constants $c_1,c_2,C,R_1, R_2$ and engough large $n$:
\begin{enumerate}
    \item $\widehat{f}$ is a zero-chain with respect to $\mathsf{u}$
    \item $\widehat{f}$ is $\ell$-smooth \wrt $(\x,\z,\bar{\y})$ and is $1/n^3$-strongly convcave \wrt $\bar{\y}$. Moreover, $\forall (\x,\z,\bar{\y})\in \mathcal{B}^T_{R_1}\times\mathcal{B}^{T-1}_{R_1}\times \mathcal{B}^{n(T-1)}_{nR_2}, \|\nabla \widehat{f}(\x,\z,\bar{\y})\|_\infty \leq G$.
    \item Define $\widehat{\pf}(\x,\z)\eqdef \max_{\bar{\y}\in \mathcal{B}^{n(T-1)}_{nR_2}} \widehat{f}(\x,\z,\bar{\y})$, then $\widehat{\pf}(\x,\z)$ is $\ell_{\widehat{\pf}}$-smooth for any $(\x,\z)\in \mathcal{B}^T_{R_1}\times\mathcal{B}^{T-1}_{R_1}$, and $\widehat{\pf}(\bm{0},\bm{0}) - \inf_{(\x,\z)\in \mathcal{B}^T_{R_1}\times\mathcal{B}^{T-1}_{R_1}}\widehat{\pf}(\x,\z)\leq 12 T$. 
    \item If $|\z[i]| < 1$ for some $i\in[2:T]$, then $\ell_{\widehat{\pf}} \norm{\Pi_{\mathcal{B}^T_{R_1}\times\mathcal{B}^{T-1}_{R_1}} ((\x,\z)-\ell_{\widehat{\pf}}^{-1}\nabla\pf(\x,\z) ) - (\x,\z)}>1/3$. 
\end{enumerate}
Here, we emphasize that $\ell$, $G$, $\ell_{\widehat{\pf}}$ are all numerical constants.
\begin{proofof}
We state that the proof mostly follows the proof of Theorem 2 in~\cite{complexity_lower_bounds_for_nonconvex_strongly_concave_min_max_optimization}. The difference is that we need to modify the constructed stochastic oracle to make it match the heavy-tailed noise assumption. The instance $f:\mathcal{B}^T_{\lambda R_1}\times\mathcal{B}^{T-1}_{\lambda R_1}\times \mathcal{B}^{n(T-1)}_{\lambda nR_2} \mapsto \real{}$ on which the lower bound is established is actually the rescaled $\widehat{f}$, that is:
\begin{align}
    f(\x,\z,\bar{\y}) \eqdef  \frac{L\lambda^2}{\ell} \widehat{f}\xkh{\frac{\x}{\lambda},\frac{\z}{\lambda},\frac{\bar{\y}}{\lambda}},
\end{align}
where $\lambda$ is to be determined. By Property 1, $f$ is still a zero-chain with respect to $\mathsf{u}$ after rescaling. Let $\pf(\x,\z)\eqdef \max_{\bar{\y}\in \mathcal{B}^{n(T-1)}_{\lambda nR_2}} f(\x,\z,\bar{\y})$, then the analysis in~\cite{complexity_lower_bounds_for_nonconvex_strongly_concave_min_max_optimization} tells that $\pf(\x,\z) = (L\lambda^2/\ell)\widehat{\pf}(\x/\lambda,\z/\lambda)$, which means the smoothness parameter of $\pf$ is $\ell_\pf = \ell_{\widehat{\pf}}L/\ell$ by Property 3. Moreover, $n = \loround{(\condnum/\ell)^{1/3}}$ and $T = \loround{\frac{L\Delta_1}{432\ell\eps^2}}$ can guarantee that $f$ is $L$-smooth \wrt $(\x,\z,\bar{\y})$ and is $\mu$-stongly concave \wrt $\bar{\y}$, and $\pf(\bm{0},\bm{0}) - \inf_{(\x,\z)\in \mathcal{B}^T_{\lambda R_1}\times\mathcal{B}^{T-1}_{\lambda R_1}}\pf(\x,\z)\leq \Delta_1$. Therefore, $f$ belongs to the objective functions that conforms to the NC-SC setting with given parameters $L$, $\mu$ and $\Delta_1$. For convenience, we denote $\textsc{Gap}(\ell,f,\x,\z,\mathcal{C})$ as $\ell\left\|  \Pi_\mathcal{C}( (\x,\z) - \ell^{-1} \nabla f (\x,\z) ) - (\x,\z) \right\|_2$. If $\z[T] =0$ and $\lambda = 6\ell\eps/L$, we can get
\begin{align}
    \textsc{Gap}(\ell_\pf,f,\x,\z,\mathcal{B}^T_{\lambda R_1}\times\mathcal{B}^{T-1}_{\lambda R_1}) = \frac{L\lambda}{\ell} \cdot \textsc{Gap}(\ell_{\widehat{\pf}},\widehat{f},\x/\lambda,\z/\lambda,\mathcal{B}^T_{ R_1}\times\mathcal{B}^{T-1}_{R_1}) > \frac{L\lambda}{3\ell} = 2\eps,
\end{align}
where the inequality holds due to the Property 4. The construction of the stochastic oracle is in the same way as~\cite{lower_bounds_for_non-convex_stochastic_optimization} and~\cite{complexity_lower_bounds_for_nonconvex_strongly_concave_min_max_optimization}. Given $q\in(0,1]$, the oracle $\nabla f(\mathsf{u},\xi)$ is defined as
\begin{align}
\nabla f(\mathsf{u},\xi)[i] \eqdef \begin{cases}
        \frac{\xi}{q}\nabla_i f(\mathsf{u}), & i = \inf_i \dkh{i\mid \nabla_i f(\mathsf{u}) = 0}\\
        \nabla_i f(\mathsf{u}), & otherwise
    \end{cases},
\end{align}
where $\xi\sim\mathrm{Bernoulli}(q)$. Note that $\|\nabla f(\x,\z,\bar{\y})\|_\infty\leq GL\lambda/\ell$ since we have already know that $\|\nabla \widehat{f}(\x,\z,\bar{\y})\|_\infty\leq G$ by Property 2. Then, the $p$-the central moment of $\nabla f(\mathsf{u},\xi)$ is bounded by
\begin{align*}
     \E_\xi\zkh{\norm{\nabla f(\mathsf{u},\xi)-\nabla f(\mathsf{u})}^p } &= \E_\xi \zkh{\abs{\nabla f(\mathsf{u},\xi)[i^\star] - \nabla_{i^\star} f(\mathsf{u})}^p} \\
    &= \E_\xi\zkh{\abs{\xi/q-1}^p}\cdot\abs{\nabla_{i^\star} f(\mathsf{u})}^p\\ &= \xkh{1-q+\frac{(1-q)^p}{q^{p-1}}}\xkh{\frac{L\lambda}{\ell}} \abs{\nabla_{i^\star} \widehat{f}\xkh{\mathsf{u}/\lambda}}^p \\ 
    &\leq \xkh{1-q+\frac{(1-q)^p}{q^{p-1}}}\xkh{\frac{L\lambda}{\ell}} \xkh{\frac{GL\lambda}{\ell}}^p\leq \frac{1}{q^p}\xkh{\frac{GL\lambda}{\ell}}^p
\end{align*}
Let $q = \min\dkh{1,\xkh{\frac{GL\lambda}{\sigma\ell}}^\frac{p}{p-1}}$ to make the bound no greater than $\sigma^p$. Note that Lemma 3 in~\cite{lower_bounds_for_non-convex_stochastic_optimization} implies that $f(\mathsf{u})$ with $\nabla f(\mathsf{u},\xi)$ is a probability-$q$ zero-chain, then by Lemma 1 in~\cite{lower_bounds_for_non-convex_stochastic_optimization}, we have
\begin{align*}
    \forall t \leq \frac{w_1nT-\log_2(1/\delta)}{2q}\Rightarrow \mathbb{P}\xkh{\z_t[T] = 0} \geq 1-\delta \Rightarrow \mathbb{P}\xkh{\textsc{Gap}(\ell_\pf,f,\x_t,\z_t,\mathcal{B}^T_{\lambda R_1}\times\mathcal{B}^{T-1}_{\lambda R_1}) >2\eps}\geq 1-\delta  
\end{align*}
where $\{\x_t,\z_t,\bar{\y}_t\}_{t\geq1}$ is generated by a certatin first-order (zero-respecting) algorithm {\small \sf A} with $(\x_1,\z_1,\bar{\y}) = (\bm{0},\bm{0},\bm{0})$, and $w_1$ is some constant. Furthermore, according to Markov's inequality, we have
\begin{align*}
    \E\zkh{\textsc{Gap}(\ell_\pf,f,\x_t,\z_t,\mathcal{B}^T_{\lambda R_1}\times\mathcal{B}^{T-1}_{\lambda R_1})} \geq (1-\delta)\cdot 2\eps
\end{align*}
Choosing $\delta = 1/2$, then for any 
\begin{align*}
    t\leq \frac{w_2nT}{q} = w_2\cdot \loround{\xkh{\frac{\condnum}{\ell}}^{1/3}} \cdot\loround{\frac{L\Delta_1}{\ell\eps^2}}\cdot\max\dkh{1,\xkh{\frac{\sigma\ell}{GL\lambda}}^\frac{p}{p-1}} \geq w_3 \cdot {\frac{\Delta_1L\condnum^{1/3}\sigma^\frac{p}{p-1}}{\eps^\frac{3p-2}{p-1}}},
\end{align*}
where $w_2, w_3$ are some constants, $\x_t,\z_t$ cannot be a $\eps$-stationary point of $\pf$ over $\mathcal{B}^T_{\lambda R_1}\times\mathcal{B}^{T-1}_{\lambda R_1}$. Therefore, in order to find a $\eps$-stationary point, the total number of iterations of {\small\sf A} is at least $\Omega(\Delta_1L\condnum^{1/3}\sigma^\frac{p}{p-1}\eps^{-\frac{3p-2}{p-1}})$
\end{proofof}

\section{Missing proofs in section~\ref{sec:trgda_for_nc-sc}}
\subsection{Proof of~\Cref{thm:NC-SC/TR-SGDAM/trust-region-proximal_is_gradient_normalized}}\label{secs:app/proof_trust-region-proximal_is_gradient_normalized}
\paragraph{Additional notations} $\ri{\mathcal{S}}$ and $\cl{\mathcal{S}}$ represent the \textit{relative interior} and \textit{closure} of a given subset $\mathcal{S}$, respectively. For a function $f:\real{d}\mapsto (-\infty,+\infty]$, we let $\dom{f}\eqdef \dkh{\x\mid f(\x)<+\infty}$. If $\dom{f}\neq \varnothing$, we say that $f$ is \textit{proper}. For the proper function $f$, we can define its nonempty \textit{epigraph} $\mathrm{epi}(f)\eqdef \{(\x,t)\mid t\in\real{}, f(\x)\leq t\}$.

The following lemmas in convex analysis are needed for our proof.

\begin{lemma}{\cite[Lemma C.5]{the_ball-proximal_point_method_new_algorithm_convergence_theory_and_applications}}\label[lemma]{lem:app/Supp_lems/normal_cone_closed_ball} Given a closed ball $\ball{\z}{\eta}\eqdef \dkh{\x\mid\|\x-\z\|\leq\eta}$, then the normal cone $\norCone{\ball{\z}{\eta}}{\x}\eqdef \dkh{v\mid\inner{v}{\y-\x}\leq0,\;\forall\y\in\ball{\z}{\eta}}$ is 
\begin{align}
    \norCone{\ball{\z}{\eta}}{\x} = 
    \begin{cases}
        \dkh{v\mid \lambda(\x-\z),\;\forall \lambda\geq0} & \norm{\x-\z} = \eta,\\
        \dkh{\bm{0}} & \norm{\x-\z} < \eta, \\
        \varnothing & \norm{\x-\z} > \eta.
    \end{cases}
\end{align}
\end{lemma}

\begin{lemma}\label[lemma]{lem:app/Supp_lems/convex_analysis_1}
Let $f:\real{d}\mapsto (-\infty,+\infty]$ is proper and convex. Given a nonempty convex set $X$, where $\dom{f}\cap X\neq\varnothing$, then for $\x\in\dom{f}\cap X$, we have
\begin{align}
    \ri{X}\cap\ri{\dom{f}}\neq\varnothing\;and\;\partial(f+\delta_X)(\x)\neq\varnothing \Longrightarrow \partial f(\x)\neq\varnothing,
\end{align}
where $\delta_X$ is the indicator function about $X$.  
\end{lemma}
\begin{remark}
     The relative interior condition $\ri{X}\cap\ri{\dom{f}}\neq\varnothing$ is necessary. Assume that $f(x)=-\sqrt{x}$ for $x\geq 0$ and $f(x)\equiv +\infty$ for $x<0$, then $f$ is proper and convex. If $\X = \{0\}$, then $\partial (f+\delta_X)(0) = (-\infty,+\infty)$, while $\partial f(0)=\varnothing$.
\end{remark}
\begin{proofof}
    First, we denote $f_\delta$ as $f+\delta_X$, then $f_\delta$ is also a proper and convex function, and $\epi{f_\delta} = \epi{f} \cap (X\times\real{})$. Since $\ri{X}\cap\ri{\dom{f}}\neq\varnothing$, we know that $\ri{\epi{f}}\cap\ri{X\times\real{}}\neq \varnothing$ also holds true. Then by the noraml cone intersection rule~\cite[Corollary 23.8.1]{rockafellar_convex_analysis}, we get $\norCone{\epi{f_\delta}}{(\x,f(\x))} = \norCone{\epi{f}}{(\x,f(\x))}+\norCone{X\times\real{}}{(\x,f(\x))}$. Since $\partial(f_\delta)(\x)\neq\varnothing$, then for any $u\in\partial f_\delta (\x)$, we have $(u,-1)\in \norCone{\epi{f_\delta}}{(\x,f(\x))}$, and there exists $(v,\alpha)\in\norCone{\epi{f}}{(\x,f(\x))}$ and $(w,\beta)\in\norCone{X\times\real{}}{(\x,f(\x))}$ such that $(u,-1) = (v+w,\alpha+\beta)$. Let $s>0$, then by the definition of normal cone, we know that $\inner{(w,\beta)}{(\x,f(\x))-(\x,f(\x)+s)}\leq 0$ and $\inner{(w,\beta)}{(\x,f(\x))-(\x,f(\x)-s)}\leq 0$ are both ture. Hence, $\beta$ must be $0$, which means $\alpha=-1$ and $v\in\partial f(\x)$. 
\end{proofof}

\begin{lemma}\label[lemma]{lem:app/Supp_lems/convex_analysis_2}
Let $f:\real{d}\mapsto (-\infty,+\infty]$ is proper, closed and convex. Given any $\z\in\dom{f}$, then $\argmin_{\x\in\ball{\z}{\eta}}f(\x)$ ($\eta>0$) is nonempty, closed and convex. Moreover, for any $\x^\dag\in \argmin_{\x\in\ball{\z}{\eta}}f(\x)$, $\partial f(\x^\dag)\neq\varnothing$ and $\partial(f+\delta_{\ball{\z}{\eta}})(\x^\dag) = \partial f(\x^\dag) + \norCone{\ball{\z}{\eta}}{\x^\dag}$, where $\norCone{\ball{\z}{\eta}}{\x^\dag}\eqdef \dkh{v\mid\inner{v}{\x-\x^\dag}\leq0,\;\forall\x\in\ball{\z}{\eta}}$.
\end{lemma}
\begin{proofof}
    Since $f$ is a proper closed function, meaning $f$ is lower semi-continuous, and because $\dom{f}\cap\ball{\z}{\eta}\neq\varnothing$, then according to the Weierstrass extreme value theorem, $\argmin_{\x\in\ball{\z}{\eta}}f(\x)\neq\varnothing$ and $\argmin_{\x\in\ball{\z}{\eta}}f(\x)\subset\dom{f}\cap\ball{\z}{\eta}$. Furthermore, since $f$ is closed, we can conclude that for any $\gamma$, the level set $V_\gamma\eqdef\dkh{\x\mid f(\x)\leq\gamma}$ is also closed. Let $\gamma_\star = \min_{\x\in\ball{\z}{\eta}}f(\x)$, then $\argmin_{\x\in\ball{\z}{\eta}}f(\x) = V_{\gamma_\star}\cap\ball{\z}{\eta}$, so $\argmin_{\x\in\ball{\z}{\eta}}f(\x)$ is closed. let $\x_1,\x_2\in \argmin_{\x\in\ball{\z}{\eta}}f(\x)$, then for any $\lambda\in(0,1)$, $(1-\lambda)\x_1+\lambda\x_2\in\dom{f}\cap\ball{\z}{\eta}$ also minimizes $f$ on $\ball{\z}{\eta}$ due to the convexity of $f$. Therefore, $\argmin_{\x\in\ball{\z}{\eta}}f(\x)$ must be nonempty, closed and convex. 

    Now, we prepare to show $\ri{\dom{f}}\cap\ri{\ball{\z}{\eta}}\neq\varnothing$. If $\z\in\ri{\dom{f}}$, we surely have $\ri{\dom{f}}\cap\ri{\ball{\z}{\eta}}\neq\varnothing$. If $\z\notin\ri{\dom{f}}$, we still can get $\z\in\cl{\dom{f}} = \cl{\ri{\dom{f}}}$ by~\cite[Proposition 1.3.5]{bertsekas_convex_theory}, which means there exists a sequence $\dkh{\x_t}_{t\in\N_+}\in\ri{\dom{f}}$ such that $\lim_{t\rightarrow+\infty}\|\x_t-\z\|=0$. Therefore, there always exists a sufficiently large $t$ such that $\|\x_t-\z\|<\eta$, and thus $\ri{\dom{f}}\cap\ri{\ball{\z}{\eta}}\neq\varnothing$. Moreover, we know that $\bm{0}\in\partial(f+\delta_{\ball{\z}{\eta}})(\x^\dag)$ for any $\x^\dag\in \argmin_{\x\in\ball{\z}{\eta}}f(\x)$ by the optimality condtion, then we can conclude that $\partial f(\x^\dag)\neq\varnothing$ by~\Cref{lem:app/Supp_lems/convex_analysis_1}. Furthermore, by~\cite[Theorem 23.8]{rockafellar_convex_analysis}, we have $\partial\xkh{f+\delta_{\ball{\z}{\eta}}}(\x^\dag) = \partial f(\x^\dag)+\partial\delta_{\ball{\z}{\eta}}(\x^\dag)$ because  $\ri{\dom{f}}\cap\ri{\ball{\z}{\eta}}\neq\varnothing$, and it is easy to verify that $\partial \delta_{\ball{\z}{\eta}}(\x^\dag) = \norCone{\ball{\z}{\eta}}{\x^\dag}$ 
\end{proofof}

\begin{proofof}
 First, since $r(\x)$ is a proper, closed and convex function, we know that $\widehat{r}(\x)\eqdef\inner{\bv}{\x}+r(\x)$ is also a proper, closed and convex function for any $\bv\in\real{d}$, and $\dom{\widehat{r}}=\dom{r}$. Then by~\Cref{lem:app/Supp_lems/convex_analysis_2}, $\argmin_{x\in\ball{\z}{\eta}}\widehat{r}(\x)$ is a nonempty, closed and convex set, and for any $\x^\dag \in \argmin_{x\in\ball{\z}{\eta}}\widehat{r}(\x)$, we have $\partial \widehat{r}(\x^\dag)\neq \varnothing$, which also means $\partial r(\x^\dag)\neq\varnothing$. Moreover, we get $\bm{0}\in \partial(\widehat{r}+\delta_{\ball{\z}{\eta}})(\x^\dag)= \bv+\partial r(\x^\dag)+\norCone{\ball{\z}{\eta}}{\x^\dag}$ by the optimality condition and~\Cref{lem:app/Supp_lems/convex_analysis_2}. {If $\bm{0}\notin \bv+\partial r(\x^\dag)$} By~\Cref{lem:app/Supp_lems/normal_cone_closed_ball}, we know that $\|\x^\dag-\z\|=\eta$ and $\exists \lambda>0,g_r(\x^\dag)\in\partial r(\x^\dag)$ such that $\bm{0} = \bv+g_r(\x^\dag)+\lambda(\x^\dag-\z)$. That is to say $\|\x^\dag-\z\| = \lambda^{-1}\|\bv+g_r(\x^\dag)\| = \eta$, which means $\x^\dag = \z-\eta\frac{\bv+g_r(\x^\dag)}{\|\bv+g_r(\x^\dag)\|}$.

We then prove that in the assumption that $\forall \x^\dag\in\mathcal{X}^\dag_\star, \bm{0}\notin \bv+\partial r(\x^\dag)$, $\x^\dag$ must be unique. If $\x^\dag_1,\x^\dag_2\in\mathcal{X}^\dag_\star$ and $\x^\dag_1\neq\x^\dag_2$, we define $\x^\dag_\lambda \eqdef \lambda\x_1^\dag+(1-\lambda)\x_2^\dag$, where $\lambda\in(0,1)$. Note that $\|\x^\dag_\lambda-\z\|<\eta$ because $\|\x^\dag_1-\z\|=\eta$ and $\|\x^\dag_2-\z\|=\eta$. Due to the convexity of $\inner{\nabla f(\z)}{\x}+r(\x)$, we get $\x^\dag_\lambda\in \mathcal{X}^\dag_\star$. Consequently, we must have $\bm{0}\notin \bv+\partial r(\x^\dag_\lambda)$. However, $\bm{0}\in \nabla f(\z)+\partial r(\x^\dag_\lambda)+\norCone{\ball{\z}{\eta}}{\x^\dag_\lambda} = \nabla f(\z)+\partial r(\x^\dag_\lambda)$ by the optimality condition and~\Cref{lem:app/Supp_lems/normal_cone_closed_ball}. Such a contradiction indicates that $\mathcal{X}_\star^\dag$ must be a singleton 
\end{proofof}
\clearpage

\subsection{Proof of~\Cref{thm:NC-SC/TR-SGDAM/thm_Complexity_TR-SGDAM}}\label{sec:app/convergence_analysis_strgdam}
First, we prove the following lemmas
\begin{lemma}{\cite[Lemma 3]{understand_grad_ortho_trust_region_opt}}\label[lemma]{lem:NC-SC/TR-SGDAM/Key_descent_ineq}
    Suppose that $r:\real{d}\mapsto (-\infty,+\infty]$ is a proper, closed and convex function. Given any $\z\in\dom{r}$, $\eta>0$ and $\bv\in\real{d}$, we let $\mathcal{X}^\dag_\star$ be the solution set of $\min_{\x\in\ball{\z}{\eta}} \dkh{\inner{\bv}{\x}+r(\x)}$. Then $\forall \x^\dag\in\mathcal{X}_\star^\dag$, there always exists $g_r(\x^\dag)\in\partial r(\x^\dag)$, such that
    \begin{align}\label{eq:NC-SC/TR-SGDAM/Key_descent_ineq}
        r(\z) + \inner{\bv}{\z-\x^\dag}\geq r(\x^\dag) +\eta\norm{\bv+g_r(\x^\dag)}
    \end{align}
    holds true.
\end{lemma}
\begin{proofof}
First, we have $r(\z)+\inner{\bv}{\z-\x^\dag}\geq r(\x^\dag) + \inner{g_r(\x^\dag)+\bv}{\z-\x^\dag}$ holds due to the convexity of $r(\x)$. If $0\in \bv+\partial r(\x^\dag)$, i.e., $\exists g_r(\x^\dag)\in\partial r(\x^\dag),\bv=-g_r(\x^\dag)$, then~\eqref{eq:NC-SC/TR-SGDAM/Key_descent_ineq} is obviously holds. If $0\notin \bv+\partial r(\x^\dag)$, then accroding to the~\Cref{thm:NC-SC/TR-SGDAM/trust-region-proximal_is_gradient_normalized}, we know that $\exists g_r(\x^\dag)\in\partial r(\x^\dag),\; \x^\dag = \z-\eta\frac{\bv+g_r(\x^\dag)}{\|\bv+g_r(\x^\dag)\|}$, then we immediately obtained~\eqref{eq:NC-SC/TR-SGDAM/Key_descent_ineq}.
\end{proofof}
\begin{remark}
    Although~\Cref{lem:NC-SC/TR-SGDAM/Key_descent_ineq} has been proposed by~\cite{understand_grad_ortho_trust_region_opt}, their proof is more complex than ours. Once we understand through~\Cref{thm:NC-SC/TR-SGDAM/trust-region-proximal_is_gradient_normalized} that the trust-region gradient method is actually a generalization of gradient normalization to composite optimization problems, then the validity of~\eqref{eq:NC-SC/TR-SGDAM/Key_descent_ineq} is intuitive. 
\end{remark}
\begin{lemma}\label[lemma]{lem:app/TR-SGDAM/tech_lems/primal_descent_ineq}
Suppose that Assumption~\ref{assum: smoothness_heavy-tailed_noise}\ref{assum:smoothness} and Assumption~\ref{assum: strongly_concave_for_y} hold. After one iteration of \trsgdam, we have
\begin{align}\label{eq:app/TR-SGDAM/tech_lems/primal_descent_ineq}
    &\ \Psi(\xiterP)-\Psi(\xiter)&\nnl 
    &\ \leq  -\etax{\dist{0,{\nabla\pf(\xiterP)+\partial r(\xiterP)}}} + 2\condnum\etax\dist{\grady{\xiter}{\yiter},\partial h(\yiter)} + 2\etax\norm{\errorxiter} + 3\condnum L\etax^2,
\end{align}
where $\pf(\x)\eqdef \max_{\y} f(\x,\y)-h(\y)$, $\Psi(\x)\eqdef \pf(\x)+r(\x)$ and $\errorxiter\eqdef \momxiter - \gradx{\xiter}{\yiter}$.
\end{lemma}
\begin{proofof} 
Based on the $2\condnum L$-smoothness of $\pf(\cdot)$, we have
\begin{align}\label{eq:app/TR-SGDAM/Proof_Lemma1_UB1}
    \pf(\xiterP) + r(\xiterP)\leq& \pf(\xiter) + r(\xiterP)+\inner{\gradx{\xiter}{\ystar_\iter}}{\xiterP-\xiter} + L\condnum\etax^2\nnl 
    =& \pf(\xiter) + r(\xiterP)+\inner{\momxiter}{\xiterP-\xiter}+\inner{\gradx{\xiter}{\ystar_\iter}-\gradx{\xiter}{\yiter}}{\xiterP-\xiter}&\nnl 
    &+\inner{\gradx{\xiter}{\yiter}-\momxiter}{\xiterP-\xiter} + L\condnum\etax^2  &\nnl 
    \leq& \pf(\xiter) + r(\xiterP)+\inner{\momxiter}{\xiterP-\xiter} + L\etax\norm{\yiter-\yiterstar} + \etax\norm{\errorxiter} + L\condnum\etax^2 
\end{align}
holds for any $\xiterP\in \argmin_{\x\in\ball{\xiter}{\etax}}\dkh{\inner{\momxiter}{\x}+r(\x)}$, then by~\Cref{lem:NC-SC/TR-SGDAM/Key_descent_ineq}, there exists $\subgradxiter\in\partial r(\xiter)$, such that
\begin{align}\label{eq:app/TR-SGDAM/Proof_Lemma1_UB2}
   &\ r(\xiter)-r(\xiterP)+\inner{\momxiter}{\xiter-\xiterP} \geq   \etax\norm{\momxiter+\subgradxiterP} &\nnl 
   &\ = \etax\norm{\nabla\pf(\xiterP)+\subgradxiterP+\nabla\pf(\xiter)-\nabla\pf(\xiterP)+\momxiter-\gradx{\xiter}{\yiter}+\gradx{\xiter}{\yiter}-\nabla\pf(\xiter)}& \nnl 
   &\ \geq  \etax\xkh{\norm{\nabla\pf(\xiterP)+\subgradxiterP} - \norm{\nabla\pf(\xiter)-\nabla\pf(\xiterP)}-\norm{\errorxiter}-\norm{\gradx{\xiter}{\yiter}-\gradx{\xiter}{\ystar_\iter}}} &\nnl 
   &\ \geq \etax\xkh{\dist{0,{\nabla \pf(\xiterP)+\partial r(\xiterP)}} - 2\condnum L\etax-\norm{\errorxiter}-L\norm{\yiter-\ystar_\iter}} 
\end{align}
Substituting~\eqref{eq:app/TR-SGDAM/Proof_Lemma1_UB2} into~\eqref{eq:app/TR-SGDAM/Proof_Lemma1_UB1}, we get
\begin{align}\label{eq:app/TR-SGDAM/Proof_Lemma1_UB3}
  &\  \zkh{ \pf(\xiterP) + r(\xiterP)} - \zkh{ \pf(\xiter) + r(\xiter)} & \nnl 
  &\ \leq -\etax{\dist{0,{\nabla\pf(\xiterP)+\partial r(\xiterP)}}} + 2L\etax\norm{\yiter-\ystar_\iter} + 2\etax\norm{\errorxiter} + 3\condnum L\etax^2.
\end{align}
By our assumption, $f(\xiter,\cdot)$ is $\mu$-strongly concave with respect to $\y$, then $f(\xiter,\cdot)-h(\cdot)$ is also $\mu$-strongly concave since $h(\cdot)$ is convex. We also konw $-\grady{\xiter}{\y} + \partial h(\y)=\partial(-f(\xiter,\cdot)+h)(\y)$ for any $\y\in\dom{h}$. Since $\ystar_\iter= \argmin_\y\dkh{-f(\xiter,\y)+h(\y)}$, we have $0\in -\grady{\xiter}{\y} + \partial h(\y)$, and 
\begin{align}
    -f(\xiter,\yiter)+h(\yiter)+f(\xiter,\ystar_\iter) -h(\ystar_\iter) \geq\frac{\mu}{2}\norm{\yiter-\ystar_\iter}^2.
\end{align}
Note that $\partial h(\yiter)\neq\varnothing$ by~\Cref{thm:NC-SC/TR-SGDAM/trust-region-proximal_is_gradient_normalized}. Hence, $\forall \subgradyiter\in\partial h(\yiter)$, such that
\begin{align}
    -f(\xiter,\ystar_\iter) + h(\ystar_\iter)+f(\xiter,\yiter)- h(\yiter)\geq \inner{\subgradyiter-\grady{\xiter}{\yiter} }{\ystar_\iter-\yiter}+\frac{\mu}{2}\norm{\yiterstar-\yiter}^2,
\end{align}
then we can obtain
\begin{align}
    \norm{\subgradyiter-\grady{\xiter}{\yiter}}\norm{\yiterstar-\yiter}\geq \inner{\subgradyiter-\grady{\xiter}{\yiter} }{\ystar_\iter-\yiter}+{\mu}\norm{\yiterstar-\yiter}^2,
\end{align}
which implies that
\begin{align}\label{eq:app/TR-SGDAM/Proof_Lemma1_UB4}
  {\mu}\norm{\yiterstar-\yiter}\leq  \dist{\grady{\xiter}{\yiter},\partial h(\yiter)}.
\end{align}
Now, by substituting~\eqref{eq:app/TR-SGDAM/Proof_Lemma1_UB4} into~\eqref{eq:app/TR-SGDAM/Proof_Lemma1_UB3}, we have completed the proof.
\end{proofof}

\begin{lemma}\label[lemma]{lem:app/TR-SGDAM/tech_lems/dual_descent_ineq}
Suppose that Assumption~\ref{assum: smoothness_heavy-tailed_noise}\ref{assum:smoothness} and Assumption~\ref{assum: strongly_concave_for_y} hold. After one iteration of \trsgdam, we have
\begin{align}\label{eq:app/TR-SGDAM/tech_lems/dual_descent_ineq}
     F(\xiterM,\yiterM) - F(\xiter,\yiter) \leq& \etax\dist{0,\nabla\pf(\xiterM)+\partial r(\xiterM)} +  \etax\condnum\dist{\grady{\xiterM}{\yiterM},\partial h(\yiterM)}\nnl 
    &-\etay \dist{\grady{\xiter}{\yiter},\partial h(\yiter)}+2\etay\norm{\erroryiterM} + 3L\etay^2 + 2L\etax\etay,
\end{align}
where $F(\xiter,\yiter)\eqdef f(\xiter,\yiter)-h(\yiter)+r(\xiter)$ and $\erroryiterM\eqdef \momyiterM - \grady{\xiterM}{\yiterM}$.
\end{lemma}
 
\begin{proofof}
On the one hand, based on the $L$-smothness of $f(\xiter,\cdot)$ with respect to $\y$, we have
$f(\xiter,\yiter)-f(\xiter,\yiterM)-\inner{\grady{\xiter}{\yiterM}}{\yiter-\yiterM}\geq -L/2\|\yiter-\yiterM\|$, which means
\begin{align}
    &\ f(\xiter,\yiterM) - h(\yiterM) \leq f(\xiter, \yiter) -h(\yiterM)+ \inner{\grady{\xiter}{\yiterM}}{\yiterM-\yiter} + \frac{L\etay^2}{2} & \nnl 
    &\ = f(\xiter, \yiter) -h(\yiterM)+ \inner{\momyiterM}{\yiterM-\yiter}+ \inner{\grady{\xiter}{\yiterM}-\grady{\xiterM}{\yiterM}}{\yiterM-\yiter}&\nnl
    &\;+\inner{\grady{\xiterM}{\yiterM}-\momyiterM}{\yiterM-\yiter} + \frac{L\etay^2}{2} &\nnl 
    &\ \leq  f(\xiter, \yiter) -h(\yiterM)+ \inner{\momyiterM}{\yiterM-\yiter} + L\etax\etay + \etay\norm{\erroryiterM}+ \frac{L\etay^2}{2};
\end{align}
on the other hand, by~\Cref{lem:NC-SC/TR-SGDAM/Key_descent_ineq}, there exists $g_h(\yiter)\in h(\yiter)$, such that
\begin{align}
    &\ h(\yiterM)-h(\yiter)- \inner{\momyiterM}{\yiterM-\yiter}\geq   \etay\norm{-\momyiterM+g_h(\yiter)} &\nnl
    &\ = \etay\norm{\grady{\xiter}{\yiter}-\grady{\xiterM}{\yiterM}+\grady{\xiterM}{\yiterM}- \momyiterM+ g_h(\yiter)-\grady{\xiter}{\yiter}}& \nnl 
    &\ \geq \etay\xkh{ \norm{g_h(\yiter)-\grady{\xiter}{\yiter}} - \norm{\grady{\xiter}{\yiter}-\grady{\xiterM}{\yiterM}} -\norm{\erroryiterM}
    }& \nnl 
    &\ \geq \etay\dist{\grady{\xiter}{\yiter},\partial h(\yiter)} - L\etay(\etax+\etay)-\etay\norm{\erroryiterM}.
\end{align}
Therefore, we get
\begin{align}\label{eq:app/TR-SGDAM/Proof_Lemma2/UB1}
   &\ \zkh{f(\xiter,\yiterM)-h(\yiterM)}-\zkh{f(\xiter,\yiter)-h(\yiter)}& \nnl
   &\ \leq -\etay \dist{\grady{\xiter}{\yiter},\partial h(\yiter)} + 2\etay\norm{\erroryiterM} + 2L\etax\etay +\frac{3L\etay^2}{2}
\end{align}

Now, we consider the upper bound $f(\xiterM,\yiterM)+r(\xiterM) - f(\xiter,\yiterM) - r(\xiter)$. By the $L$-smoothness of $f(\cdot,\yiterM)$ with respect to $x$ and the convexity of $r$, $\forall\subgradxiterM\in r(\xiterM)$ we have
\begin{align}\label{eq:app/TR-SGDAM/Proof_Lemma2/UB2}
   &\ f(\xiterM,\yiterM)+r(\xiterM) - f(\xiter,\yiterM) - r(\xiter) \leq
    \inner{\gradx{\xiter}{\yiterM}+\subgradxiterM}{\xiterM-\xiter} + \frac{L\etax^2}{2} & \nnl 
    &\ = \inner{\nabla\pf(\xiterM)+\subgradxiterM}{\xiterM-\xiter} + \inner{\gradx{\xiter}{\yiterM}-\gradx{\xiterM}{\yiterM}}{\xiterM-\xiter}&\nnl
    &\ +\inner{\gradx{\xiterM}{\yiterM}-\gradx{\xiterM}{\yiterMstar}}{\xiterM-\xiter} +\frac{L\etax^2}{2}&\nnl 
    &\ \leq \etax \norm{\nabla\pf(\xiterM)+\subgradxiterM} + L\etax\norm{\yiterM-\yiterMstar} + \frac{3L\etax^2}{2} & \nnl
    &\leq \etax \norm{\nabla\pf(\xiterM)+\subgradxiterM} + \etax\condnum\dist{\grady{\xiterM}{\yiterM},\partial h(\yiterM)} + \frac{3L\etax^2}{2}, 
\end{align}
where the last inequality holds because we have already obtained~\eqref{eq:app/TR-SGDAM/Proof_Lemma1_UB4}. 

We get the following by adding inequalities~\eqref{eq:app/TR-SGDAM/Proof_Lemma2/UB1} and~\eqref{eq:app/TR-SGDAM/Proof_Lemma2/UB2}:
\begin{align}
    F(\xiterM,\yiterM) - F(\xiter,\yiter) \leq& \etax\dist{0,\nabla\pf(\xiterM)+\partial r(\xiterM)} +  \etax\condnum\dist{\grady{\xiterM}{\yiterM},\partial h(\yiterM)}\nnl 
    &-\etay \dist{\grady{\xiter}{\yiter},\partial h(\yiter)}+2\etay\norm{\erroryiterM} + 3L\etay^2 + 2L\etax\etay
\end{align}
Note that we have replaced the term $\norm{\nabla\pf(\xiterM)+\subgradxiterM}$ with $\dist{\nabla\pf(\xiterM)+\partial r(\xiterM)}$ on R.H.S. of the above inequality, because~\eqref{eq:app/TR-SGDAM/Proof_Lemma2/UB2} holds for any $\subgradxiterM\in\partial r(\xiterM)$.
\end{proofof}

\subsubsection{Prove the complexity}
\begin{proofof}
By~\Cref{lem:app/TR-SGDAM/tech_lems/primal_descent_ineq} and~\Cref{lem:app/TR-SGDAM/tech_lems/dual_descent_ineq}, we can get the following inequality
\begin{align}\label{eq:app/TR-SGDAM/Proof_thm/UB1}
    &\ \Psi(\xiterP)-F(\xiterP,\yiterP)+\etax\dist{0,\nabla\pf(\xiterP)+\partial r(\xiterP)}+\etay\dist{\grady {\xiterP}{\yiterP},\partial h(\yiterP)} &\nnl
    &\ -\zkh{\Psi(\xiter)-F(\xiter,\yiter)+\etax\dist{0,\nabla\pf(\xiter)+\partial r(\xiter)}+\etay\dist{\grady {\xiter}{\yiter},\partial h(\yiter)}} &\nnl 
    &\ \leq (3\condnum\etax-\etay) \dist{\grady {\xiter}{\yiter},\partial h(\yiter)} + 2\etay\norm{\erroryiter}+2\etax\norm{\errorxiter} + 3L\etay^2 + 3\condnum L\etax^2 + 2L\etax\etay   
\end{align}
Then, we add~\eqref{eq:app/TR-SGDAM/Proof_thm/UB1} to~\eqref{eq:app/TR-SGDAM/tech_lems/primal_descent_ineq} to get 
\begin{align}
    \mathrm{V}_\iterP - \mathrm{V}_\iter \leq& -\etax\dist{0,\nabla\pf(\xiterP)+\partial r(\xiterP)}+(5\condnum\etax-\etay) \dist{\grady {\xiter}{\yiter},\partial h(\yiter)}\nnl 
    &+ 2\etay\norm{\erroryiter}+4\etax\norm{\errorxiter} + 3L\etay^2 + 6\condnum L\etax^2 + 2L\etax\etay\nnl 
    =& -\etax\dist{0,\nabla\pf(\xiterP)+\partial r(\xiterP)}+ 10\condnum\etax\norm{\erroryiter}+4\etax\norm{\errorxiter} + 75\condnum^2 L\etax^2 + 16\condnum L\etax^2,
\end{align}
where $\mathrm{V}_\iter\eqdef 2\Psi(\xiter)-F(\xiter,\yiter)+\etax\dist{0,\nabla\pf(\xiter)+\partial r(\xiter)}+\etay\dist{\grady {\xiter}{\yiter},\partial h(\yiter)}$, and the last equality holds beacuse we let $5\condnum\etax=\etay$. We rule $\mathrm{V}_0 = \mathrm{V}_1$, $\zeta_{\y,1}=\zeta_{\y,0}$ and $\zeta_{\y,1}=\zeta_{\y,0}$, then we sum both sides of the above inequality from $t=0$ to $T-1$ and multiply by $(T\etax)^{-1}$ to get
\begin{align}\label{eq:app/TR-SGDAM/Proof_thm/UB2}
    \frac{1}{T}\sumIter \dist{0,\nabla\pf(\xiter)+\partial r(\xiter)} \leq \frac{\mathrm{V}_1-\mathrm{V}_T}{T\etax} + \frac{1}{T}\sumIter  \xkh{10\condnum\norm{\erroryiter}+4\norm{\errorxiter}} + 75\condnum^2 L\etax + 16\condnum L\etax
\end{align}
Note that
\begin{align}\label{eq:app/TR-SGDAM/Proof_thm/UB3}
    \mathrm{V}_1-\mathrm{V}_T \leq& \Psi(\x_1)-\Psi(\x_T) + \Psi(\x_1)-F(\x_1,\y_1) -\zkh{\Psi(\x_T)-F(\x_T,\y_T)}\nnl 
    &+ \etax\dist{0,\nabla\pf(\x_1)+\partial r(\x_1)}+\etay\dist{\grady {\x_1}{\y_1},\partial h(\y_1)}\nnl 
    \leq& \underbrace{\Psi(\x_1)-\min_\x \Psi(\x)+\Psi(\x_1)-F(\x_1,\y_1)}_{\eqdef \Lambda_1} \nnl &+\etax\underbrace{\xkh{\dist{0,\nabla\pf(\x_1)+\partial r(\x_1)}+5\condnum\dist{\grady {\x_1}{\y_1},\partial h(\y_1)}}}_{\eqdef \partial_1},
\end{align}
where the last inequality holds because $\Psi(\x_T)-F(\x_T,\y_T) = \pf(\x_T)-(f(\x_T,\y_T)-h(\x_T))\geq 0$. Substituting~\eqref{eq:app/TR-SGDAM/Proof_thm/UB3} into~\eqref{eq:app/TR-SGDAM/Proof_thm/UB2} and taking  expectation $\E[\cdot]$ on both sides of the inequality, we get
\begin{align}\label{eq:app/TR-SGDAM/Proof_thm/UB4}
    \frac{1}{T}\sumIter \dist{0,\nabla\pf(\xiter)+\partial r(\xiter)} \leq \frac{\partial_1}{T}+\frac{\Lambda_1}{T\etax}  + 75\condnum^2 L\etax + 16\condnum L\etax  + \frac{1}{T}\sumIter  \xkh{10\condnum\Exp{\norm{\erroryiter}}+4\Exp{\norm{\errorxiter}}}
\end{align}
Now, we need to upper bound $\Exp{\norm{\erroryiter}}$ and $\Exp{\norm{\errorxiter}}$. The analysis of this part is analogous to the analysis of~\cite{Zijian_Liu_nsgd,Suntao_nsgd,From_Gradient_Clipping_to_Normalization} for \algname{NSGD} with momentum. Concretetly, we first introduce following notations
    \begin{align}
        D_{\x,\iter} \eqdef \pgradX f(\xiterM,\yiterM)-\pgradX f(\xiter,\yiter),\quad \varepsilon_{\x,\iter}\eqdef \pgradX f(\xiter,\yiter,\xi_\iter)-\pgradX f(\xiter,\yiter) 
    \end{align}
    Then, we can decompose ${\errorxiter}$ into following
    \begin{align}
     \xerror{\iter} &= \momxiter -\pgradX f(\xiter,\yiter)\nnl
     &= \beta \momxiterM +(1-\beta)\pgradX f(\xiter,\yiter,\xi_\iter) - \pgradX f(\xiter,\yiter)   = \beta \xerror{\iterM} + \beta \disX{\iter} + (1-\beta)\resX{\iter}
\end{align}
By recursion, we have
\begin{align}
    \xerror{\iter} = \beta^\iter \xerror{1} +(1-\beta)\sum_{s=1}^t \beta^{t-s} \resX{s} + \sum_{s=1}^t \beta^{t-s+1}\disX{s}
\end{align}
Thus, we have
\begin{align}
    \norm{\xerror{\iter}}\leq \beta^t \norm{\xerror{1}} + (1-\beta)\norm{\sum_{s=1}^t \beta^{t-s} \resX{s}} + \sum_{s=1}^t \beta^{t-s+1}\norm{\disX{s}}
\end{align}
Note by the $L$-smoothness of $f$, we have
\begin{align}
    \sum_{s=1}^t \beta^{t-s+1}\norm{\disX{s}}&\leq \sum_{s=1}^t \beta^{t-s+1}L\xkh{\norm{\x_{s-1}-\x_s}+\norm{\y_{s-1}-\y_s}}\nnl
    &\leq \sum_{s=1}^t \beta^{t-s+1}L\xkh{\eta_\x+\eta_\y}\leq \frac{\beta L(\etax+\etay)}{1-\beta} \leq \frac{6\condnum L\etax\beta}{1-\beta} 
\end{align}
By the filtration $\fil_t\eqdef \sigma\xkh{\xi_1,\dots,\xi_\iter}$ we defined. We know that $\beta^{\iter-s}\resX{s}$ is $\fil_s$-measurable and $\Exp{\beta^{\iter-s}\resX{s}|\fil_{s-1}}=0$ due to the unbiased estimator $\pgradX f(\x_s,\y_s,\xi_s)$. Moreover, by the Assumptions~\ref{assum: smoothness_heavy-tailed_noise}\ref{assum:p-BCM}, we have $\Exp{\norm{\beta^{\iter-s}\resX{s}}}<\infty$. Thus, we can get
\begin{align}
 \Exp{\norm{\sum_{s=1}^t \beta^{t-s} \resX{s}}} =& \Exp{\xkh{\norm{\sum_{s=1}^t \beta^{t-s} \resX{s}}^p}^{\frac{1}{p}}}\overset{\text{\tiny Jensen's Ineq}}{\leq} \xkh{\Exp{\norm{\sum_{s=1}^t \beta^{t-s} \resX{s}}^p}}^{\frac{1}{p}}\nnl
    \stackAlign{\text{\tiny \Cref{lem: app/Supp_lems/p_momment_batch_martingales}}}{\leq} \xkh{2\sum_{s=1}^t \beta^{p(t-s)}\Exp{\norm{\resX{s}}^p}}^{\frac{1}{p}} \leq\xkh{2\sum_{s=1}^t \beta^{p(t-s)}\sigma^p}^{\frac{1}{p}}\leq 2\sigma\xkh{\sum_{s=1}^t \beta^{p(t-s)}}^{\frac{1}{p}}
\end{align}
Then, we get
\begin{align}
     (1-\beta)\Exp{\norm{\sum_{s=1}^t \beta^{t-s} \resX{s}}} \leq& 2(1-\beta)\sigma\xkh{\sum_{s=1}^t \beta^{p(t-s)}}^{\frac{1}{p}} = 2(1-\beta)\sigma\xkh{\sum_{s=0}^{t-1} \beta^{ps}}^{\frac{1}{p}}\leq\frac{2(1-\beta)\sigma}{\xkh{1-\beta^p}^{\frac{1}{p}}}\nnl
     \stackAlign{\text{\tiny $\beta^p\leq \beta$}}{\leq} 2\sigma(1-\beta)^{1-\frac{1}{p}}
\end{align}
Finally, we get
\begin{align}
    \Exp{\norm{\xerror{\iter}}} \leq \xkh{\beta^t+2(1-\beta)^{1-\frac{1}{p}}}\sigma+ \frac{6\condnum L\beta\etax}{1-\beta} 
\end{align}
Based on the exact same analysis, we also have $\Exp{\norm{\yerror{\iter}}} \leq \xkh{\beta^t+2(1-\beta)^{1-\frac{1}{p}}}\sigma+ \frac{6\condnum L\beta\etax}{1-\beta}$ holds true. Futhermore, we can get
\begin{align}\label{eq:app/TR-SGDAM/Proof_thm/UB5}
    \frac{1}{T}\sumIter \Exp{\norm{\xerror{\iter}}}, \frac{1}{T} \sumIter \Exp{\norm{\yerror{\iter}}} \leq  \frac{\sigma}{T(1-\beta)} + 2(1-\beta)^{1-\frac{1}{p}} \sigma +  \frac{6\condnum L\beta\etax}{1-\beta} 
\end{align}
Now, we substitute~\eqref{eq:app/TR-SGDAM/Proof_thm/UB5} into~\eqref{eq:app/TR-SGDAM/Proof_thm/UB4} to get
\begin{align}
        \frac{1}{T}\sumIter \dist{0,\nabla\pf(\xiter)+\partial r(\xiter)} \leq& \frac{\partial_1}{T}+\frac{\Lambda_1}{T\etax}  + 75\condnum^2 L\etax + 16\condnum L\etax   + \frac{84\condnum^2L\beta\etax}{1-\beta}\nnl 
        &+ \frac{14\condnum\sigma}{T(1-\beta)} + 28\condnum(1-\beta)^{1-\frac{1}{p}}\sigma\nnl 
        \leq& \frac{\partial_1}{T}+\frac{\Lambda_1}{T\etax} + \frac{91\condnum^2L\etax}{1-\beta} + \frac{14\condnum\sigma}{T(1-\beta)} + 28\condnum(1-\beta)^{1-\frac{1}{p}}\sigma
\end{align}
\paragraph{Case 1. If we can know $\Lambda_1$, $L$, $\sigma$ and $p$}Since we let 
\begin{align}
    \etax = \frac{\sqrt{\Lambda(1-\beta)/L}}{\condnum\sqrt{T}}\;\text{and}\;  \beta = 1-\min\dkh{1,\max\dkh{\frac{1}{T^{\frac{p}{2p-1}}},\xkh{\frac{\Lambda_1L}{\sigma^2T}}^\frac{p}{3p-2}}},
\end{align}
we can obtain
\begin{align}
    \frac{1}{T}\sumIter \dist{0,\nabla\pf(\xiter)+\partial r(\xiter)} \leq& \frac{\partial_1}{T} + 92\condnum\sqrt{\frac{\Lambda L}{T(1-\beta)}}+\frac{14\condnum\sigma}{T(1-\beta)} + 28\condnum(1-\beta)^{1-\frac{1}{p}}\sigma\nnl 
    \leq& \frac{\partial_1}{T} + 92\condnum\cdot\max\dkh{\sqrt{\frac{\Lambda_1L}{T}},\xkh{\frac{\Lambda_1L}{T}}^\frac{p-1}{3p-2}\cdot\sigma^\frac{p}{3p-2}} + \frac{14\condnum\sigma}{T^\frac{p-1}{2p-1}}\nnl &+ 28\condnum\max\dkh{
    \frac{\sigma}{T^\frac{p-1}{2p-1}} ,\xkh{\frac{\Lambda_1L}{T}}^\frac{p-1}{3p-2}\cdot\sigma^\frac{p}{3p-2} 
    }\nnl 
    \leq& \frac{\partial_1}{T} + 92\condnum\sqrt{\frac{\Lambda_1L}{T}} + 120\condnum\xkh{\frac{\Lambda_1L}{T}}^\frac{p-1}{3p-2}\cdot\sigma^\frac{p}{3p-2} + \frac{42\condnum\sigma}{T^\frac{p-1}{2p-1}}
\end{align}
Hence, to achive $ (1/T)\sumIter \dist{0,\nabla\pf(\xiter)+\partial r(\xiter)} \leq \eps$, the total gradient complexity is upper bounded by
\begin{align}
    \mathcal{O}\xkh{\frac{\partial_1}{\eps}+\frac{\condnum^2\Lambda_1L}{\eps^2}+\frac{(\condnum\sigma)^{\frac{2p-1}{p-1}}}{\eps^\frac{2p-1}{p-1}} + \frac{\condnum^\frac{3p-2}{p-1}\Lambda_1 L\sigma^\frac{p}{p-1}}{\eps^\frac{3p-2}{p-1}}}
\end{align}
\paragraph{Case 2. If we don't konw $\Lambda_1$, $L$, $\sigma$ and $p$} We let $\etax = \frac{1}{\condnum T^{0.75}}$ and $\beta = 1-\frac{1}{\sqrt{T}}$, then we obtain
\begin{align}
    \frac{1}{T}\sumIter \dist{0,\nabla\pf(\xiter)+\partial r(\xiter)} \leq \frac{\partial_1}{T} + \frac{\condnum\Lambda_1+91\condnum L}{T^{\nicefrac{1}{4}}} + \frac{14\condnum\sigma}{\sqrt{T}} + \frac{28\condnum\sigma}{T^\frac{p-1}{2p}}
\end{align}
Hence, in such a case, to achive $ (1/T)\sumIter \dist{0,\nabla\pf(\xiter)+\partial r(\xiter)} \leq \eps$, the total gradient complexity is upper bounded by
\begin{align}
    \mathcal{O}\xkh{\frac{\partial_1}{\eps} + \frac{(\condnum\Lambda_1+\condnum L)^4}{\eps^4}
    + \xkh{\frac{\condnum\sigma}{\eps}}^2 + \xkh{\frac{\condnum\sigma}{\eps}}^{\frac{2p}{p-1}} 
    }
\end{align}
\end{proofof}

\clearpage

\subsection{Proof of~\Cref{thm:NC-SC/TR-SGDAmax/thm_Complexity_TR-SGDAmax}}\label{sec:app/convergence_analysis_strgdamax}
First, we prove the following lemma.
\begin{lemma}\label[lemma]{lem:app/TR-SGDAmax/tech_lems/dual_error}
    Suppose that Assumption~\ref{assum: smoothness_heavy-tailed_noise} and~\ref{assum: strongly_concave_for_y} hold. Let $\{\xiter,\yiter\}_{t=1}^T$ be the output of \algname{Stoc-TRGDAmax}, $\yiterstar$ be defined as the maxmizer of ${f(\xiter,\y)-h(\y)}$ and $\delta_1\eqdef \|\y_1-\ystar_1\|$. Given any $\eps>0$, we let $\etay = \min\dkh{\frac{2}{72^{p/(p-1)}\mu}\xkh{\frac{\eps}{\condnum\sigma }}^\frac{p}{p-1}, \frac{2^\frac{1}{p-1}}{\mu\condnum^\frac{p}{2(p-1)}}  }$.  
    \begin{itemize}
        \item If $\eps < 2L\delta_1$, $\etax\leq \frac{2\delta_1}{3\condnum}$, and
        \begin{align}
           K-1 = p\ln\xkh{\frac{6L\delta_1}{\eps}}\cdot\max\dkh{\frac{72^\frac{p}{p-1}\sigma^\frac{p}{p-1}\condnum^\frac{p}{p-1}}{\eps^\frac{p}{p-1}}, \frac{\condnum^\frac{p}{2(p-1)}}{2^\frac{2-p}{p-1}} },
        \end{align}
        then for any $t=2,\hdots, T$, we have $\E[\|\yiter-\yiterstar\|]\leq \eps/2L$.
        \item If $\delta_1\leq \eps/2L$, $\etax\leq \frac{\eps}{3\condnum L}$, and
        \begin{align}
            K-1= (p\ln3)\cdot \max\dkh{
    \frac{72^\frac{p}{p-1}(\sigma\condnum)^\frac{p}{p-1}}{\eps^\frac{p}{p-1}}, \frac{\condnum^\frac{p}{2(p-1)}}{2^\frac{2-p}{p-1}}
    },
        \end{align}
        then for any $t=1,\hdots,T$, we have $\E[\|\yiter-\yiterstar\|]\leq \eps/2L$.
    \end{itemize}
\end{lemma}

\begin{proofof}
Let's temporarily ignore the outer loop, and for simplicity, we denote $\wtyk\eqdef\yiterk{k}$, $\wtxik\eqdef\xi_{\iter,k}$ and $\widetilde{\zeta}_{\y,k}\eqdef \pgradY f(\xiter,\wtyk,\wtxik) - \pgradY f(\xiter,\wtyk)$. We still denote $\yiterstar$ as the maxmizer of ${f(\xiter,\y)-h(\y)}$. First, by the fixed-point property of the proximal operator, we have
\begin{align}
    \norm{\wtykP-\yiterstar}^p =& \norm{\prox{\etay h}{\wtyk+\etay\pgradY f(\xiter,\wtyk,\wtxik)}-\prox{\etay h}{\yiterstar + \etay\pgradY f(\xiter,\yiterstar)}}^p\nnl 
    \stackAlign{(a)}{\leq} \norm{\wtyk-\yiterstar+\etay\xkh{\pgradY f(\xiter,\wtyk,\wtxik)-\pgradY f(\xiter,\yiterstar)}}^p\nnl 
    \stackAlign{\text{\tiny \Cref{lem: app/Supp_lems/holder_smooth_norm_raise_to_power_p}}}{\leq} \norm{\wtyk-\yiterstar}^p + \frac{p\etay}{\norm{\wtyk-\yiterstar}^{2-p}}\inner{\wtyk-\yiterstar}{\pgradY f(\xiter,\wtyk,\wtxik)-\pgradY f(\xiter,\yiterstar)}\nnl  &+ 2^{2-p}\etay^p \norm{\pgradY f(\xiter,\wtyk,\wtxik)-\pgradY f(\xiter,\yiterstar)}^p\nnl 
    \leq& \norm{\wtyk-\yiterstar}^p+\frac{p\etay}{\norm{\wtyk-\yiterstar}^{2-p}}\inner{\wtyk-\yiterstar}{\wterrory{k}}+ 2\etay^p \|\wterrory{k}\|^p \nnl &+ \frac{p\etay}{\norm{\wtyk-\yiterstar}^{2-p}}\inner{\wtyk-\yiterstar}{\pgradY f(\xiter,\wtyk)-\pgradY f(\xiter,\yiterstar)}\nnl &+ 2\etay^p\norm{\pgradY f(\xiter,\wtyk)-\pgradY f(\xiter,\yiterstar)}^p\nnl 
    \stackAlign{(b)}{\leq} \xkh{1-\frac{\mu p\etay}{2}}\norm{\wtyk-\yiterstar}^p+\frac{p\etay}{\norm{\wtyk-\yiterstar}^{2-p}}\inner{\wtyk-\yiterstar}{\wterrory{k}}+ 2\etay^p \|\wterrory{k}\|^p\nnl & + \xkh{2\etay^pL^{p/2} - \frac{p\etay\mu^{\frac{2-p}{2}}}{2}}\inner{\yiterstar-\wtyk}{\pgradY f(\xiter,\wtyk)-\pgradY f(\xiter,\yiterstar)}^\frac{p}{2}\nnl 
    \stackAlign{(c)}{\leq} \xkh{1-\frac{\mu \etay}{2}}\norm{\wtyk-\yiterstar}^p+\frac{2\etay}{\norm{\wtyk-\yiterstar}^{2-p}}\inner{\wtyk-\yiterstar}{\wterrory{k}}+ 2\etay^p \|\wterrory{k}\|^p,
\end{align}
where $(a)$ holds due to the nonexpensiveness of proximal operator, and $(c)$ holds because we let $\etay$ be smaller than$\frac{2^{1/(p-1)}}{\mu\condnum^{p/2(p-1)}}$. Here we explain why $(b)$ holds. First, because $f(\xiter,\cdot)$ is $L$-smooth and $\mu$-strongly concave with respct to $\y$ , we have following two well-known inequalities:
\begin{align}
    &\mu\norm{\wtyk-\yiter}^2 \leq \inner{\yiterstar-\wtyk}{\pgradY f(\xiter,\wtyk)-\pgradY f(\xiter,\yiterstar)}\tag{\text{Ineq.1}} \\
    &\inner{\yiterstar-\wtyk}{\pgradY f(\xiter,\wtyk)-\pgradY f(\xiter,\yiterstar)} \geq \frac{\norm{\pgradY f(\xiter,\wtyk)-\pgradY f(\xiter,\yiterstar)}^2}{L}\tag{\text{Ineq.2}} 
\end{align}
Then, we have
\begin{align}
 &\  \frac{1}{2\norm{\wtyk-\yiterstar}^{2-p}}\inner{\wtyk-\yiterstar}{\pgradY f(\xiter,\wtyk)-\pgradY f(\xiter,\yiterstar)} & \nnl 
&\ = \frac{1}{-2\norm{\wtyk-\yiterstar}^{2-p}}\inner{\yiterstar-\wtyk}{\pgradY f(\xiter,\wtyk)-\pgradY f(\xiter,\yiterstar)}\nnl 
 &\overset{\text{Ineq.1}}{\leq} -\xkh{\frac{\mu}{\inner{\yiterstar-\wtyk}{\pgradY f(\xiter,\wtyk)-\pgradY f(\xiter,\yiterstar)}}}^\frac{2-p}{2}\cdot \frac{\inner{\yiterstar-\wtyk}{\pgradY f(\xiter,\wtyk)-\pgradY f(\xiter,\yiterstar)}}{2} &\nnl 
 &\ = -\frac{\mu^\frac{2-p}{2}}{2}\inner{\yiterstar-\wtyk}{\pgradY f(\xiter,\wtyk)-\pgradY f(\xiter,\yiterstar)}^\frac{p}{2},
\end{align}
and
\begin{align}
     &\  \frac{1}{2\norm{\wtyk-\yiterstar}^{2-p}}\inner{\wtyk-\yiterstar}{\pgradY f(\xiter,\wtyk)-\pgradY f(\xiter,\yiterstar)}\overset{\text{Ineq.2}}{\leq} -\frac{\mu\norm{\wtyk-\yiterstar}^2}{2\norm{\wtyk-\yiterstar}^{2-p}} = -\frac{\mu}{2}\norm{\wtyk-\yiterstar}^p.
\end{align}
Hence, $(b)$ holds true. Before we take the expectation on both sides of the inequality, note that
\begin{align}
   & {\Exp{\frac{2\etay}{\norm{\wtyk-\yiterstar}^{2-p}}\inner{\wtyk-\yiterstar}{\wterrory{k}}\mid \xiter,\yiterM,\widehat{\xi}_{\iter,1},\hdots,\widehat{\xi}_{\iter,k-1}}} = 0,\nnl &{\Exp{\|\wterrory{k}\|^p\mid \xiter,\yiterM,\widehat{\xi}_{\iter,1},\hdots,\widehat{\xi}_{\iter,k-1}}}\leq \sigma^p.
\end{align}
Therefore, we can get
\begin{align}
{\Exp{\norm{\wtykP-\yiterstar}^p\mid \xiter,\yiterM,\widehat{\xi}_{\iter,1},\hdots,\widehat{\xi}_{\iter,k-1}}} \leq \xkh{1-\frac{\mu \etay}{2}}^{1/p}\norm{\wtyk-\yiterstar}+ 2\etay^p \sigma
\end{align}
After taking the conditional expectation $\Exp{\cdot\mid \xiter,\yiterM}$ on both sides of above inequality, we get 
\begin{align}
    \Exp{\norm{\wtykP-\yiterstar}^p\mid \xiter,\yiterM}\leq \xkh{1-\frac{\mu \etay}{2}}\Exp{\norm{\wtyk-\yiterstar}^p\mid \xiter,\yiterM}+ 2\etay^p \sigma^p,
\end{align}
which means
\begin{align}
    \Exp{\norm{\yiter-\yiterstar}^p\mid \xiter,\yiterM} \leq& \xkh{1-\frac{\mu \etay}{2}}^{K-1} \norm{\yiterM-\yiterstar}^p + \frac{4\etay^{p-1}\sigma^p}{\mu}\nnl 
    \leq& \rho\norm{\yiterM-\yiterstar}^p + \frac{8\sigma^p\ln^{p-1}(1/\rho)}{\mu^p(K-1)^{p-1}}
\end{align}
Here, we use $\frac{\mu\etay}{2}\leq \frac{2^{(2-p)/(p-1)}}{\condnum^{p/2(p-1)}}\leq \frac{1}{\condnum^{p/2(p-1)}}\leq 1$ since $\condnum\geq 1$. The last inequality holds because we let $\etay = \frac{2\ln(1/\rho)}{\mu(K-1)
}\Leftrightarrow  K-1 = \frac{2\ln(1/\rho)}{\mu\etay}$, where $\rho$ is to be determined later and $0<\rho<1$.
\begin{align}
\Exp{\norm{\yiter-\yiterstar}\mid \xiter,\yiterM}\overset{\text{\tiny Jensen's Ineq}}{\leq}
    \xkh{\Exp{\norm{\yiter-\yiterstar}^p\mid \xiter,\yiterM}}^\frac{1}{p} \leq  \rho^{1/p}\norm{\yiterM-\yiterstar} + \frac{8\sigma\ln^\frac{p-1}{p}(1/\rho)}{\mu(K-1)^{\frac{p-1}{p}}}
\end{align}
Now, take full expectation on the both sides of above inequality, we can get
\begin{align}
    \Exp{\norm{\yiter-\yiterstar}} &\leq  \rho^{1/p}\Exp{\norm{\yiterM-\yiterstar}} + \frac{8\sigma\ln^\frac{p-1}{p}(1/\rho)}{\mu(K-1)^{\frac{p-1}{p}}}\nnl 
    &\leq \rho^{1/p}\Exp{\norm{\yiterM-\yiterMstar}}+\rho^{1/p}\Exp{\norm{\yiterstar-\yiterMstar}} + \frac{8\sigma\ln^\frac{p-1}{p}(1/\rho)}{\mu(K-1)^{\frac{p-1}{p}}}\nnl 
    &\leq \rho^{1/p}\Exp{\norm{\yiterM-\yiterMstar}}+\rho^{1/p}\condnum\etax + \frac{8\sigma\ln^\frac{p-1}{p}(1/\rho)}{\mu(K-1)^{\frac{p-1}{p}}},
\end{align}
where we used $\norm{\yiterstar-\yiterMstar}\leq \condnum \norm{\xiter-\xiterM}\leq\condnum\etax$. Still through recursion, we can obtain
\begin{align}
    \Exp{\norm{\yiter-\yiterstar}} \leq \rho^\frac{t-1}{p}\norm{\y_1-\ystar_1} + \frac{\rho^{1/p}\condnum\etax}{1-\rho^{1/p}} + \frac{8\sigma\ln^\frac{p-1}{p}(1/\rho)}{(1-\rho^{1/p})\mu(K-1)^{\frac{p-1}{p}}}
\end{align}
In the following discussion, we let $r\eqdef \eps/2L\delta_1$, where $\delta_1\eqdef \|\y_1-\y_1\|$ and $\eps$ can be any positive constant, representing the numerical precision we want $\|\yiter-\yiterstar\|$ to achieve. 
\begin{itemize}
    \item If $0<r<1$, we let $\rho = (r/3)^{1/p}$. Then, we have $0<\rho<1$ obviously, and $\rho^{(t-1)/p}\leq \rho^{1/p} = r/3$, which means $\rho^\frac{t-1}{p}\norm{\y_1-\ystar_1} \leq \eps/6L$ holds true for any $t = 2,\hdots,T$. Moreover, we have
    \begin{align}
        \frac{\rho^{1/p}\condnum\etax}{1-\rho^{1/p}} = \frac{r\condnum\etax}{3-r} \leq \frac{\eps}{6L} \Leftrightarrow \etax\leq \frac{\delta_1(3-r)}{3\condnum}\Leftarrow \etax\leq \frac{2\delta_1}{3\condnum},
    \end{align}
    and 
    \begin{align}
        &\ \frac{8\sigma\ln^\frac{p-1}{p}(1/\rho)}{(1-\rho^{1/p})\mu(K-1)^{\frac{p-1}{p}}} = \frac{8\sigma \xkh{p\ln(3/r)}^{\frac{p-1}{p}}}{(1-r/3)\mu (K-1)^\frac{p-1}{p}}\leq \frac{12\sigma\xkh{p\ln(3/r)}^{\frac{p-1}{p}}}{\mu (K-1)^\frac{p-1}{p}} \leq \frac{\eps}{6L}&\nnl 
        &\ \Leftrightarrow K-1\geq \frac{72^\frac{p}{p-1}\sigma^\frac{p}{p-1}p\ln(3/r)\condnum^\frac{p}{p-1}}{\eps^\frac{p}{p-1}} = \frac{72^\frac{p}{p-1}\sigma^\frac{p}{p-1}p\ln(6L\delta_1/\eps)\condnum^\frac{p}{p-1}}{\eps^\frac{p}{p-1}}.
    \end{align}
    In addition, we let $K-1 \geq 2^\frac{p-2}{p-1}{p\ln(6L\delta_1/\eps)\condnum^\frac{p}{2(p-1)}} = 2^\frac{p-2}{p-1}{\ln(1/\rho)\condnum^\frac{p}{2(p-1)}}$ to make sure $\etay = \frac{2\ln(1/\rho)}{\mu(K-1)
}\leq \frac{2^{1/(p-1)}}{\mu\condnum^\frac{p}{2(p-1)}}$. Therefore, we require $K-1$ to be
\begin{align}
\max\dkh{\frac{72^\frac{p}{p-1}\sigma^\frac{p}{p-1}p\ln(6L\delta_1/\eps)\condnum^\frac{p}{p-1}}{\eps^\frac{p}{p-1}}, \frac{p\ln(6L\delta_1/\eps)\condnum^\frac{p}{2(p-1)}}{2^\frac{2-p}{p-1}} }, 
\end{align}
which means
\begin{align}
    \etay = \frac{2\ln(1/\rho)}{\mu(K-1)
} = \min\dkh{\frac{2}{72^{p/(p-1)}\mu}\xkh{\frac{\eps}{\condnum\sigma }}^\frac{p}{p-1}, \frac{2^\frac{1}{p-1}}{\mu\condnum^\frac{p}{2(p-1)}}  }.
\end{align}
Hence, we can conlude that for any $t=2,\hdots,T$, we have $\E[\|\yiter-\yiterstar\|]\leq\eps/2L$ if 
\begin{align}
    \begin{cases}
        \etax\leq \frac{2\delta_1}{3\condnum}, \\
        \etay = \min\dkh{\frac{2}{72^{p/(p-1)}\mu}\xkh{\frac{\eps}{\condnum\sigma }}^\frac{p}{p-1}, \frac{2^\frac{1}{p-1}}{\mu\condnum^\frac{p}{2(p-1)}}  },\\ 
        K-1= p\ln\xkh{\frac{6L\delta_1}{\eps}}\cdot\max\dkh{\frac{72^\frac{p}{p-1}\sigma^\frac{p}{p-1}\condnum^\frac{p}{p-1}}{\eps^\frac{p}{p-1}}, \frac{\condnum^\frac{p}{2(p-1)}}{2^\frac{2-p}{p-1}} },
    \end{cases}
\end{align}
and $\eps<2L\delta_1$.
\item If $r\geq 1$, which means $\|\y_1-\ystar_1\|\leq \eps/2L$. In such a case, we let $\rho = 1/3^p$, then $\rho^{(t-1)/p}\|\y_1-\ystar_1\|\leq \eps/6L$. Moreover, we have
\begin{align}
    \frac{\rho^{1/p}\condnum\etax}{1-\rho^{1/p}} = \frac{\condnum\etax}{2} \leq \frac{\eps}{6L}\Leftrightarrow \etax\leq \frac{\eps}{3\condnum L},
\end{align}
and 
\begin{align}
    &\ \frac{8\sigma\ln^\frac{p-1}{p}(1/\rho)}{(1-\rho^{1/p})\mu(K-1)^{\frac{p-1}{p}}} = \frac{12\sigma (p\ln3)^\frac{p-1}{p} }{\mu (K-1)^{\frac{p-1}{p}}}\leq \frac{\eps}{6L}\Leftrightarrow K-1\geq \frac{72^\frac{p}{p-1} p\ln3(\sigma\condnum)^\frac{p}{p-1}}{\eps^\frac{p}{p-1}}.
\end{align}
In addition, we let $K-1\geq 2^\frac{p-2}{p-1}(p\ln3)\condnum^\frac{p}{2(p-1)}$ to make sure $\etay \leq \frac{2^{1/(p-1)}}{\mu\condnum^{p/2(p-1)}}$. Therefore, we require $K-1$ to be
\begin{align}
    (p\ln3)\cdot \max\dkh{
    \frac{72^\frac{p}{p-1}(\sigma\condnum)^\frac{p}{p-1}}{\eps^\frac{p}{p-1}}, \frac{\condnum^\frac{p}{2(p-1)}}{2^\frac{2-p}{p-1}}
    },
\end{align}
which means 
\begin{align}
    \etay = \frac{2\ln(1/\rho)}{\mu(K-1)} = \min\dkh{\frac{2}{72^{p/(p-1)}\mu}\xkh{\frac{\eps}{\condnum\sigma }}^\frac{p}{p-1}, \frac{2^\frac{1}{p-1}}{\mu\condnum^\frac{p}{2(p-1)}}  }
\end{align}
Hence, we can conlude that for any $t=1,\hdots,T$, we have $\E[\|\yiter-\yiterstar\|]\leq\eps/2L$ if 
\begin{align}
    \begin{cases}
        \etax\leq \frac{\eps}{3\condnum L}, \\
        \etay = \min\dkh{\frac{2}{72^{p/(p-1)}\mu}\xkh{\frac{\eps}{\condnum\sigma }}^\frac{p}{p-1}, \frac{2^\frac{1}{p-1}}{\mu\condnum^\frac{p}{2(p-1)}}  },\\
        K-1= (p\ln3)\cdot \max\dkh{
    \frac{72^\frac{p}{p-1}(\sigma\condnum)^\frac{p}{p-1}}{\eps^\frac{p}{p-1}}, \frac{\condnum^\frac{p}{2(p-1)}}{2^\frac{2-p}{p-1}}
    },   
    \end{cases}
\end{align}
and $\eps\geq 2L\delta_1$.
\end{itemize}
\end{proofof}

\begin{lemma}\label[lemma]{lem:app/TR-SGDAmax/tech_lems/primal_descent}
    Suppose that Assumption~\ref{assum: smoothness_heavy-tailed_noise} and~\ref{assum: strongly_concave_for_y} hold. Let $\{\xiter,\yiter\}_{t=1}^T$ be the output of \trsgdamax, and $\yiterstar$ be defined as the maxmizer of ${f(\xiter,\y)-h(\y)}$. Then, for any $t=2,\hdots,T$, we have following inequality:
    \begin{align}\label{eq:app/TR-SGDAmax/tech_lems/primal_descent}
          &\ \etax\Exp{\dist{0,{\nabla\pf(\xiter)+\partial r(\xiter)}}} & \nnl 
  &\ \leq \E\zkh{ \pf(\xiterM) + r(\xiterM)}-\E\zkh{ \pf(\xiter) + r(\xiter)}+ 2L\etax\Exp{\norm{\yiterM-\ystar_\iterM}} + \frac{4\etax\sigma}{B^{(p-1)/p}} + 3\condnum L\etax^2.
    \end{align}
\end{lemma}
\begin{proofof}
    See proof of~\Cref{lem:app/TR-SGDAM/tech_lems/primal_descent_ineq}, from~\eqref{eq:app/TR-SGDAM/Proof_Lemma1_UB1} to~\eqref{eq:app/TR-SGDAM/Proof_Lemma1_UB3}. Note that the validity of~\eqref{eq:app/TR-SGDAmax/tech_lems/primal_descent} does not require any additional conditions for $\etax$.
\end{proofof}
\clearpage
\subsubsection{Prove the complexity}
\begin{proofof}
    We start with~\eqref{eq:app/TR-SGDAmax/tech_lems/primal_descent} in~\Cref{lem:app/TR-SGDAmax/tech_lems/primal_descent}, summing both sides of~\eqref{eq:app/TR-SGDAmax/tech_lems/primal_descent} from $t=2$ to $T$, then we get
    \begin{align}
       &\ \sum_{t=2}^T \etax\Exp{\dist{0,{\nabla\pf(\xiter)+\partial r(\xiter)}}}\nnl  &\ \leq \Delta_1 + \sum_{t=1}^{T-1} 2L\etax\Exp{\norm{\yiter-\ystar_\iter}} + \frac{4(T-1)\etax\sigma}{B^{(p-1)/p}} + 3(T-1)\condnum L\etax^2
    \end{align}
    Multiplying both sides of the above inequality by $(\etax(T-1))^{-1}$, we get
    \begin{align}
      &\ \frac{1}{T-1}\sum_{t=2}^T \Exp{\dist{0,{\nabla\pf(\xiter)+\partial r(\xiter)}}}\nnl  &\ \leq \frac{\Delta_1}{(T-1)\etax} + \frac{1}{T-1}\sum_{t=1}^{T-1} 2L\Exp{\norm{\yiter-\ystar_\iter}} + \frac{4\sigma}{B^{(p-1)/p}} + 3\condnum L\etax.
    \end{align}
    We first let $\etax = \sqrt{\frac{\Delta_1}{3(T-1)\condnum L}}$, $T-1\geq {\frac{3\Delta_1\condnum L}{\eps^2}}$ and $B = \max\dkh{\upround{\xkh{\frac{4\sigma}{\eps}}^\frac{p}{p-1}},1}$ to make sure that
    \begin{align}
        \frac{1}{T-1}\sum_{t=2}^T \Exp{\dist{0,{\nabla\pf(\xiter)+\partial r(\xiter)}}}\leq \frac{1}{T-1}\sum_{t=1}^{T-1} 2L\Exp{\norm{\yiter-\ystar_\iter}} + 3\eps
    \end{align}
    \begin{itemize}
        \item If $\eps<2L\delta_1$, then $T-1\geq \frac{3\Delta_1\condnum L}{\eps^2}\geq \frac{3\Delta_1\condnum}{4L\delta_1^2}$, which means $\etax\leq 2\delta_1/3\condnum$. By~\Cref{lem:app/TR-SGDAmax/tech_lems/dual_error}, we set  $K-1 = p\ln\xkh{\frac{6L\delta_1}{\eps}}\cdot\max\dkh{\frac{72^\frac{p}{p-1}\sigma^\frac{p}{p-1}\condnum^\frac{p}{p-1}}{\eps^\frac{p}{p-1}}, \frac{\condnum^\frac{p}{2(p-1)}}{2^\frac{2-p}{p-1}} }$ to make sure that for any $t=2,\hdots,T$, we have $\E[\|\yiter-\yiterstar\|]\leq \eps/2L$, which means
        \begin{align}
            \frac{1}{T-1}\sum_{t=2}^T \Exp{\dist{0,{\nabla\pf(\xiter)+\partial r(\xiter)}}}\leq \frac{2L \delta_1}{T-1} + \frac{(T-2)\eps}{T-1} + 3\eps \leq \frac{2L \delta_1}{T-1} + 4\eps.
        \end{align}
        Therefore, If $T-1 = \max\dkh{\frac{3\Delta_1\condnum L}{\eps^2}, \frac{2L\delta_1}{\eps}}$, then we can have $(1/(T-1))\sum_{t=2}^T \Exp{\dist{0,{\nabla\pf(\xiter)+\partial r(\xiter)}}}\leq 5\eps$, and the total gradient complexity is upper bounded by
        \begin{align}
            (TB+TK) \asymleq& {
            \max\dkh{\frac{\Delta_1\condnum L}{\eps^2}, \frac{L\delta_1}{\eps}}\times\left[
            \ln\xkh{\frac{L\delta_1}{\eps}}\cdot\max\dkh{\frac{\sigma^\frac{p}{p-1}\condnum^\frac{p}{p-1}}{\eps^\frac{p}{p-1}}, {\condnum^\frac{p}{2(p-1)}}}+ \max\dkh{{\xkh{\frac{\sigma}{\eps}}^\frac{p}{p-1}},1}
            \right]
            }\nnl 
            \asymleq& L\ln\xkh{\frac{L\delta_1}{\eps}}\cdot\max\dkh{
            \frac{\delta_1\condnum^\frac{p}{2p-2}}{\eps} , \frac{\Delta_1\condnum^\frac{3p-2}{2p-2}}{\eps^2}, \frac{\delta_1(\condnum\sigma)^\frac{p}{p-1}}{\eps^\frac{2p-1}{p-1}} , \frac{\Delta_1 \condnum^\frac{2p-1}{p-1}\sigma^\frac{p}{p-1}}{\eps^\frac{3p-2}{p-1}}
            }
        \end{align}
        \item If $\eps\geq 2L\delta_1$. We let $T-1 = \frac{3\Delta_1\condnum L}{\eps^2}$, then $\etax = \eps/3\condnum L$. By~\Cref{lem:app/TR-SGDAmax/tech_lems/dual_error}, we set  $K-1 = p\ln3\cdot\max\dkh{\frac{72^\frac{p}{p-1}\sigma^\frac{p}{p-1}\condnum^\frac{p}{p-1}}{\eps^\frac{p}{p-1}}, \frac{\condnum^\frac{p}{2(p-1)}}{2^\frac{2-p}{p-1}} }$ to make sure that for any $t=1,\hdots,T$, we have $\E[\|\yiter-\yiterstar\|]\leq \eps/2L$, which means $\frac{1}{T-1}\sum_{t=2}^T \Exp{\dist{0,{\nabla\pf(\xiter)+\partial r(\xiter)}}}\leq 4\eps$ holds. Therefore, the total gradient complexity is upper bounded by
        \begin{align}
           TK+TB \asymleq& \frac{\Delta_1\condnum L}{\eps^2} \times \zkh{\max\dkh{\frac{\sigma^\frac{p}{p-1}\condnum^\frac{p}{p-1}}{\eps^\frac{p}{p-1}}, {\condnum^\frac{p}{2(p-1)}}}+ \max\dkh{{\xkh{\frac{\sigma}{\eps}}^\frac{p}{p-1}},1}} \nnl 
           \asymleq& \max\dkh{\frac{\Delta_1L\condnum^\frac{3p-2}{2p-2}}{\eps^2},\frac{\Delta_1 L\condnum^\frac{2p-1}{p-1}\sigma^\frac{p}{p-1}}{\eps^\frac{3p-2}{p-1}}}
        \end{align}
    \end{itemize}
\end{proofof}
\clearpage

\section{Supplementary Lemmas}\label{app:Supp_lems}

\begin{lemma}[Hölder's inequality]\label[lemma]{lem:app/Supp_lems/Holder_Ineq}
Let $q_1,q_2>1$ and $1/q_1+1/q_2=1$ (or $q_1=q_2=1$), then for any sequences $\{a_t\}_{t=1}^T$ and $\{b_t\}_{t=1}^T$ we have
\begin{align}\label{eq: app/Supp_lems/Holder_Ineq_sum}
    \sumIter |a_t b_t| \leq \xkh{\sumIter |a_t|^{q_1}}^{\frac{1}{q_1}} \xkh{\sumIter |b_t|^{q_2}}^\frac{1}{q_2} 
\end{align}
Moreover, for the probability space $(\Omega,\mathcal{F},\mathbb{P})$, let $X$ and $Y$ are random variables defined on $\Omega$. Hölder's inequality also means
\begin{align}\label{eq: app/Supp_lems/Holder_Ineq_exp}
    \Exp{|XY|}\leq \xkh{\Exp{|X|}^{q_1}}^{\frac{1}{q_1}}\xkh{\Exp{|Y|}^{q_2}}^{\frac{1}{q_2}},
\end{align}
\end{lemma}

\begin{lemma}\label[lemma]{lem: app/Supp_lems/holder_smooth_norm_raise_to_power_p}
    Given two real vectors $v,w\in \real{d}$. Suppose that $p\in[1,2]$ and $v\neq 0$, then we have
    \begin{align}
        \norm{v+w}^p \leq \norm{v}^p + \frac{p\inner{v}{w}}{\norm{v}^{2-p}} + 2^{2-p}\norm{w}^p
    \end{align}
\end{lemma}
\begin{proofof}
    The above lemma can be regarded as a corollary of \cite[Theorem 6.3]{smoothness_parameter_of_power_of_euclidean_norm}. Concretely, inequalty (29) in the proof of~\cite[Theorem 6.3]{smoothness_parameter_of_power_of_euclidean_norm} tells us that for any $x_1,x_2\neq 0$, $\nu\in[0,1]$, and any unit vector $h\in\mathbb{S}^{d-1}$, $|(1+\nu)(\nicefrac{\inner{x_2}{h}}{\norm{x_2}^{\nu-1}}-\nicefrac{\inner{x_1}{h}}{\norm{x_1}^{\nu-1}})|\leq 2^{1-\nu}(1+\nu)\|x_1-x_2\|^\nu$ holds. Moreover, we have
    \begin{align}
        & \left| (1+\nu)\inner{\frac{x_2}{\|x_2\|^{\nu-1}}-\frac{x_1}{\|x_1\|^{\nu-1}}}{h}\right|\leq 2^{1-\nu}(1+\nu)\|x_1-x_2\|^\nu\nnl 
        \Rightarrow & \sup_{h\in\mathbb{S}^{d-1}} (1+\nu)\inner{\frac{x_2}{\|x_2\|^{\nu-1}}-\frac{x_1}{\|x_1\|^{\nu-1}}}{h} \leq 2^{1-\nu}(1+\nu)\|x_1-x_2\|^\nu\nnl 
        \stackAlign{(a)}{\Leftrightarrow}  (1+\nu)\norm{\frac{x_2}{\|x_2\|^{\nu-1}}-\frac{x_1}{\|x_1\|^{\nu-1}} } \leq 2^{1-\nu}(1+\nu)\|x_1-x_2\|^\nu\nnl 
        \Leftrightarrow & \norm{\nabla (\|x_2\|^{\nu+1})-\nabla (\|x_1\|^{\nu+1})}\leq 2^{1-\nu}(1+\nu)\|x_1-x_2\|^\nu,
    \end{align}
    Here, $(a)$ holds true because the dual norm of the $\ell_2$ norm on $\real{d}$ is still the $\ell_2$ norm. We let $\|x\|^{\nu+1}$ be denoted as $f(x)$. Assuming $x\neq 0$ and given any $y\in\real{d}$, then there is at most one $t$ such that $x+t(y-x)=0$. Let's first assume that there exists a $t'\in(0,1)$ such that $x+t'(y-x)=0$. For sufficiently small $\epsilon$, we have $0<t'-\epsilon<t'+\epsilon<1$, and by the fundamental theorem of calculus
    \begin{align}
        f(y)-f(x) =& \int_{t'+\epsilon}^1 \inner{\nabla f(x+t(y-x))}{y-x}dt + \int_{0}^{t'-\epsilon}\inner{\nabla f(x+t(y-x))}{y-x}dt\nnl 
        &+ f(x+(t'+\epsilon)(y+x))-f(x+(t'-\epsilon)(y-x)),
    \end{align}
    then subtracting $\inner{\nabla f(x)}{y-x} = ([1-(t'+\epsilon)]+(t'-\epsilon)+2\epsilon)\inner{\nabla f(x)}{y-x}$ from both sides yields 
    \begin{align}
        f(y)-f(x) - \inner{\nabla f(x)}{y-x} =& \int_{t'+\epsilon}^1 \inner{\nabla f(x+t(y-x))-\nabla f(x)}{y-x}dt\nnl &+ \int_{0}^{t'-\epsilon}\inner{\nabla f(x+t(y-x))-\nabla f(x)}{y-x}dt\nnl
        &-2\epsilon\inner{\nabla f(x)}{y-x} + f(x+(t'+\epsilon)(y+x))-f(x+(t'-\epsilon)(y-x)).
    \end{align}
    Meanwhile, we have
    \begin{align}
         &  \int_{t'+\epsilon}^1 \inner{\nabla f(x+t(y-x))-\nabla f(x)}{y-x}dt+ \int_{0}^{t'-\epsilon}\inner{\nabla f(x+t(y-x))-\nabla f(x)}{y-x}dt &\nnl 
          &\leq \norm{y-x}\xkh{\int_{t'+\epsilon}^1 \norm{\nabla f(x+t(y-x))-\nabla f(x)}dt + \int_{0}^{t'-\epsilon}\norm{\nabla f(x+t(y-x))-\nabla f(x)}dt} &\nnl 
        &\leq 2^{1-\nu}(1+\nu)\norm{y-x}^{\nu+1}\xkh{\int_{t'+\epsilon}^1 t^\nu dt + \int_{0}^{t'-\epsilon}t^\nu dt} &\nnl 
       & = 2^{1-\nu}\norm{y-x}^{\nu+1}\xkh{1-(t+\epsilon)^{\nu+1}+(t-\epsilon)^{\nu+1}}. &
    \end{align}
    Finally we can get
    \begin{align}
         f(y)-f(x) - \inner{\nabla f(x)}{y-x}\leq& 2^{1-\nu}\norm{y-x}^{\nu+1}\xkh{1-(t+\epsilon)^{\nu+1}+(t-\epsilon)^{\nu+1}}-2\epsilon\inner{\nabla f(x)}{y-x}\nnl  &+ f(x+(t'+\epsilon)(y+x))-f(x+(t'-\epsilon)(y-x)).
    \end{align}
   Taking $\lim_{\epsilon\rightarrow 0}$ on both sides of the above inequality, we obtain inequality $\|y\|^{\nu+1}-\|x\|^{\nu+1}-\inner{\frac{(\nu+1)x}{\|x\|^{1-\nu}}}{y-x}\leq 2^{1-\nu}\norm{y-x}^{\nu+1}$. If $t'= 1$ such that $x+t'(y-x)=0$, i.e., $y=0$. The inequality still holds. Let $\nu = p-1$, $x=v$ ($v\neq 0$) and $y=v+w$, then the lemma is proved. 
\end{proofof}

\begin{lemma}\label[lemma]{lem:app/Supp_lems/smoothFunction_NormGrad_upperBound_functionValue}
Given a differentiable function $f:\real{d}\mapsto \real{}$, suppose $f$ is $L$-smooth, and let $f^\star = \sup_\x f(\x)< +\infty$, for any $\x\in\real{d}$, we have
\begin{align}
    \norm{\nabla f(\x)}\leq \sqrt{2L(f^\star-f(\x))}
\end{align}
\end{lemma}
\begin{proof}
    The above inequality is well known for $L$-smooth functions; here we provide a proof for completeness. Based on the quadratic lower bounds of $L$-smooth functions, we can obtain
    \begin{align}
       f^\star-f(\x)\geq f(\x+(1/L)\nabla f(\x)) - f(\x) -\frac{\norm{\nabla f(\x)}^2}{L} \geq - \frac{\norm{\nabla f(\x)}^2}{2L},
    \end{align}
    then we immediately proved the lemma.
\end{proof}

\begin{lemma}{(\cite[Theorem 2]{best_poss_bounds_of_vonBahr},~\cite[Proposition 1.8]{Ineq_for_r_abs_mom})}\label[lemma]{lem:app/Supp_lems/p_moment_sum_martingales}
    Given a sequence of filterations $\dkh{\fil_\iter}_{t\geq 0}$. Let $p\in[1,2]$, and the {random variables} $X_1,X_2,\dots,X_n$ be a martingale sequence with respect to $\dkh{\fil_\iter}_{t\geq 0}$, i.e., $X_i$ is $\fil_i$-measurable and $\Exp{X_i|\fil_{i-1}}=0$ a.s. If $\Exp{|X_i|^p}<\infty$ for all $i=1,\dots,n$, then we have
    \begin{align}\label{eq:app/Supp_lems/p_moment_sum_martingales}
        \Exp{\left |\sum_{i=1}^n X_i \right|^p} \leq 2^{2-p} \sum_{i=1}^n \Exp{|X_i|^p}
    \end{align}
\end{lemma}

\begin{lemma}{(\cite[Lemma 7]{From_Gradient_Clipping_to_Normalization})}\label[lemma]{lem: app/Supp_lems/p_momment_batch_martingales}
    Given a sequence of filterations $\dkh{\fil_\iter}_{t\geq 1}$. Let $p\in[1,2]$, and the {random vectors} $X_1,X_2,\dots,X_n$ be a martingale sequence with respect to $\dkh{\fil_\iter}_{t\geq 1}$, i.e., $X_i$ is $\fil_i$-measurable and $\Exp{X_i|\fil_{i-1}}=0$ a.s. If $\Exp{\|X_i\|^p}<\infty$ for all $i=1,\dots,n$, then we have
\begin{align}\label{eq:app/Supp_lems/p_momment_batch_martingales}
        \Exp{\norm{\sum_{i=1}^n X_i}^n}\leq 2\sum_{i=1}^n\Exp{\norm{X_i}^p}
    \end{align}
\end{lemma}

\end{document}